\newif\ifHAL
\HALtrue

\ifHAL
\documentclass[10pt,a4paper]{article}
\else
\documentclass[sn-mathphys-num,pdflatex,10pt]{sn-jnl}%
\fi

\usepackage{fullpage}
\usepackage{graphicx}%
\usepackage{multirow}%
\usepackage{amsmath,amssymb,amsfonts}%
\usepackage{amsthm}%
\usepackage{mathrsfs}%
\usepackage[title]{appendix}%
\usepackage{xcolor}%
\usepackage{textcomp}%
\usepackage{manyfoot}%
\usepackage{booktabs}%
\usepackage{algorithm}%
\usepackage{algorithmicx}%
\usepackage{algpseudocode}%
\usepackage{listings}%

\ifHAL
\usepackage[colorlinks,citecolor=cyan,linkcolor=magenta]{hyperref}
\fi

\usepackage{amsthm}
\usepackage{bm}
\usepackage{subcaption}
\usepackage{graphicx}
\usepackage{xspace}
\usepackage{tikz}
\usepackage{marginnote}
\usepackage{mathabx}

\newcommand{\vertiii}[1]{{\|\kern-0.25ex | #1
		| \kern-0.25ex \|}}

\newcommand{\diff}{A}
\newcommand{\ddd}{\text{\rm D}}
\newcommand{\dn}{\text{\rm N}}

\newcommand{\mean}[1]{\{\kern-1.1mm\{#1\}\kern-1.1mm\}}                  
\newcommand{\jump}[1]{[\![#1]\!]_F}                        
\newcommand{\jumpK}[1]{[\![#1]\!]_{\partial K}}                        

\newcommand{\ltwo}[2]{\|{#1}\|_{#2}}

\newcommand{\ndg}[1]{| \kern -.25mm \|{#1}| \kern -.25mm \|}
\newcommand{\nsdg}[1]{| \kern -.25mm \|{#1}| \kern -.25mm \|_{\rm s}}
\newcommand{\su}{\sum_{K\in \mesh}}
\newcommand{\sua}{\sum_{K\in \mesha}}

\newcommand{\bnabla}{\nabla_{\mesh}}
\newcommand{\nno}{\nonumber}
\newcommand{\mbf}[1]{\mbox{\boldmath$\rm{#1}$}}

\newcommand{\fes}{\hat{V}_{h}^k}
\newcommand{\fesz}{\hat{V}_{h0}^k}
\newcommand{\fesD}{\hat{V}_{h0,\rm{D}}^k}
\newcommand{\fesS}{\hat{V}_{hg,\rm{D}}^k}
\newcommand{\fesE}{\hat{V}_{K}^k}
\newcommand{\n}{{\boldsymbol{n}}}
\newcommand{\ba}{{\boldsymbol{a}}}
\newcommand{\w}{{\boldsymbol{w}}}
\newcommand{\Fall}{\mathcal{F}_h}
\newcommand{\FD}{\mathcal{F}_h^{\rm D}}
\newcommand{\FN}{\mathcal{F}_h^{\rm N}}

\newcommand{\mesh}{\mathcal{T}_h}
\newcommand{\vertice}{\mathcal{V}_h}
\newcommand{\norm}[2]{\|{#1}\|_{{#2}}}
\newcommand{\ncdg}[1]{| \kern -.25mm \|{#1}| \kern -.25mm \|_{\rm DG}}

\newcommand{\bx}{\boldsymbol{x}}

\newcommand{\bphi}{\boldsymbol{\phi}}
\newcommand{\bpsi}{\boldsymbol{\psi}}
\newcommand{\bH}{\boldsymbol{H}}
\newcommand{\bL}{\boldsymbol{L}}

\renewcommand{\hat}[1]{\widehat{#1}}
\makeatletter
\newcommand*{\rom}[1]{\text{\expandafter\@slowromancap\romannumeral #1@}}
\makeatother

\newcommand{\upi}{^{\mathrm{i}}}
\newcommand{\upb}{^{\mathrm{b}}}

\newcommand{\Fb}{\mathcal{F}_h\upb}

\newcommand{\Fint}{\mathcal{F}_h\upi}
\newcommand{\Fa}{\mathcal{F}_{\ba}}
\newcommand{\dF}{\partial F}
\newcommand{\dK}{\partial K}
\newcommand{\dKi}{\partial K\upi}

\newcommand{\dKN}{\partial K^{\dn}}
\newcommand{\dKD}{\partial K^{\ddd}}
\newcommand{\FK}{\mathcal{F}_{\dK}}
\newcommand{\Ihk}{\mathcal{\hat{I}}_h^k}
\newcommand{\FKi}{\mathcal{F}_{\dKi}}

\newcommand{\FKN}{\mathcal{F}_{\dKN}}
\newcommand{\FKD}{\mathcal{F}_{\dKD}}

\newcommand{\verticei}{\vertice\upi}
\newcommand{\verticeb}{\vertice\upb}

\newcommand{\verticeK}{\mathcal{V}_K}
\newcommand{\mesha}{\mathcal{T}_{\ba}}

\newcommand{\IKM}{I^k_{{\rm KM}}}
\newcommand{\ImKM}{I^k_{{\rm mKM}}}
\newcommand{\IamKM}{I^{k,\ba}_{{\rm mKM}}}
\newcommand{\oma}{\omega_{\ba}}

\ifHAL

\newtheorem{theorem}{Theorem}[section]
\newtheorem{proposition}[theorem]{Proposition}
\newtheorem{lemma}[theorem]{Lemma}
\newtheorem{corollary}[theorem]{Corollary}
\newtheorem{remark}[theorem]{Remark}
\newtheorem{definition}[theorem]{Definition}

\else

\theoremstyle{thmstyleone}%
\newtheorem{theorem}{Theorem}[section]

\theoremstyle{thmstyletwo}%
\newtheorem{remark}{Remark}%
\newtheorem{lemma}[theorem]{Lemma}
\newtheorem{corollary}[theorem]{Corollary}

\theoremstyle{thmstylethree}%
\newtheorem{definition}{Definition}%
\fi

\begin{document}

\ifHAL

\title{$hp$-error analysis of mixed-order hybrid high-order methods for elliptic problems on simplicial meshes}

\author{
	Zhaonan Dong\thanks{
	Inria, 48 rue Barrault, 75647 Paris, France,
	and CERMICS, ENPC, Institut Polytechnique de Paris, 6 \& 8 avenue B.~Pascal, 77455 Marne-la-Vall\'{e}e, France
	{\tt{zhaonan.dong@inria.fr}}.}
	\and
        Alexandre Ern\thanks{
	CERMICS, ENPC, Institut Polytechnique de Paris, 6 \& 8 avenue B.~Pascal, 77455 Marne-la-Vall\'{e}e, France,
	and  Inria, 48 Rue Barrault, 75647 Paris, France
	{\tt{alexandre.ern@enpc.fr}}.}
}

\fi

\ifHAL

\maketitle

\begin{abstract}
We present both $hp$-a priori and $hp$-a posteriori error analysis of a mixed-order hybrid high-order (HHO) method to approximate second-order elliptic problems on simplicial meshes.
Our main result on the $hp$-a priori error analysis is a $\frac12$-order $p$-suboptimal error estimate. This result is, to our knowledge, the first of this kind for hybrid nonconforming methods and matches the state-of-the-art for other nonconforming methods (as discontinuous Galerkin methods) with general (mixed Dirichlet/Neumann) boundary conditions. Our second main result is a residual-based $hp$-a posteriori upper error bound, comprising residual, normal flux jump, tangential jump, and stabilization estimators (plus data oscillation terms). The first three terms are $p$-optimal and only the latter is $\frac12$-order $p$-suboptimal. This result is, to our knowledge, the first $hp$-a posteriori error estimate for HHO methods. A novel approach based on the partition-of-unity provided by hat basis functions and on local Helmholtz decompositions on vertex stars is devised to estimate the nonconforming error. Finally, we establish local lower error bounds. Remarkably, the normal flux jump estimator is only $\frac12$-order $p$-suboptimal, as it can be bounded by the stabilization owing to the local conservation property of HHO methods. Numerical examples illustrate the theory.
\end{abstract}

\fi

\section{Introduction} \label{Introduction}

Second-order elliptic PDEs are widely used in the modeling of diffusion phenomena.   In the present work, we consider the following model problem:
\begin{equation}\label{Problem}
\left\{ \begin{alignedat}{2}
- \nabla {\cdot} (\diff \nabla u) &=  f   &\quad  &\text{in }  \Omega , \\
 u &= g_{\ddd} &\quad   &\text{on } \Gamma_{\ddd}, \\
 (\diff\nabla u){\cdot} \n_{\Omega} &= g_{\dn} &\quad   &\text{on }  \Gamma_{\dn},
\end{alignedat}
\right.
\end{equation}
where the domain  $\Omega$ is a Lipschitz polytopal domain (open, connected, bounded set) in $\mathbb{R}^d$, $d\in \{2,3\}$, with boundary $\partial \Omega$ and unit outward normal $\n_\Omega$. The boundary $\partial\Omega$  is split into two disjoint parts $\Gamma_{\ddd}$ and $\Gamma_{\dn}$ with $|\Gamma_{\ddd}| > 0$. In addition, the load $f\in L^2(\Omega)$, $g_{\ddd}$ is a restriction to $\Gamma_{\ddd}$ of a function in  $H^{\frac{1}{2}}(\partial \Omega)$,  $g_{\dn}\in L^2(\Gamma_{\dn})$
and $A$ is a piecewise scalar-valued diffusion coefficient such that $0<A_{\Omega}^{\flat}  \leq A(\bx) \leq A_{\Omega}^{\sharp}$ for a.e.~$\bx\in \Omega$.

The hybrid high-order (HHO) method was introduced in \cite{DiPEL:14} for linear diffusion and in \cite{DiPEr:15} for locking-free linear elasticity. As shown in \cite{CoDPE:16}, the HHO method is
closely related to hybridizable discontinuous Galerkin (HDG) and weak Galerkin (WG) methods.
These links have been leveraged, e.g., in \cite{DongErn2021biharmonic,ErnSteins2024}
to devise a unified convergence analysis for the biharmonic problem and the acoustic wave equation.
We also refer the reader to \cite{Bridging_HHO} and \cite{cicuttin2021hybrid} for links to the
nonconforming virtual element method (ncVEM), and to \cite{HHOHMM} for links to
multiscale hybrid-mixed (MHM) methods. HHO methods are formulated in terms of broken cell and face polynomial spaces. The equal-order HHO method corresponds to cell and face unknowns having the same degree $k\geq0$. Instead, the mixed-order HHO method corresponds to cell unknowns having degree $(k+1)$ and face unknowns having degree  $k\geq0$. Considering cell unknowns of degree $(k-1)$ with $k\geq1$ is also possible. One salient advantage of the mixed-order setting with cell unknowns of degree $(k+1)$ is that $h$-optimal convergence can be achieved by using the simple Lehrenfeld--Sch\"oberl (LS) HDG stabilization \cite{lehsc:16}, in contrast with the more sophisticated HHO stabilization needed in the other settings.

The goal of the present work is to derive $hp$-a priori and $hp$-a posteriori
error estimates for the mixed-order HHO method on simplicial meshes. Owing to the links highlighted above, the present results extend to HDG and WG methods in the same setting.
Although HHO (and HDG, WG) methods can deal with polytopa{l} meshes, we focus here on simplicial meshes because some of the $hp$-analysis tools we are going to invoke are only available on such meshes (and on tensor-product meshes as well). Our main result concerning the $hp$-a priori error analysis is
Theorem~\ref{Theorem: a priori}, where we derive a $\frac12$-order $p$-suboptimal and $h$-optimal error estimate. The only $hp$-a priori HHO error estimate we are aware of is derived in \cite{Aghili17hpHHO} for the equal-order HHO method and leads to a $1$-order $p$-suboptimal error bound.
Here, we use the same $hp$-scaling of the stabilization bilinear form, but the simpler form
of the LS stabilization allows us to prove an error estimate with a tighter scaling in the polynomial degree.
Notice in passing that $\frac12$-order $p$-suboptimality corresponds to the state of the art for classical discontinuous Galerkin (dG) methods \cite{hpDGHHS,GeHaMe10,DGpolybook} under general Dirichlet/Neumann boundary conditions, and was also
obtained in \cite{EggerWal13} for a hybrid dG method applied to Stokes flow. 
In the case of homogeneous Dirichlet conditions over the whole boundary,
it is shown in \cite{Stwi10} for dG methods applied to the Poisson problem
and in \cite{LeLeSc19} for HDG methods applied to the Stokes problem that
an a priori error estimate with full $p$-optimality can be achieved. However,
counterexamples in \cite{GeHaMe10} confirm that the $\frac12$-order $p$-suboptimality
is sharp in the case of inhomogeneous Dirichlet boundary conditions.

The second main contribution of the present work is to derive a residual-based $hp$-a posteriori error estimate for the mixed-order HHO method in dimensions $d\in\{2,3\}$. Our main result concerning the (global) upper error bound is Theorem~\ref{thm:upperbound}, where all the terms in the upper bound are $hp$-optimal except one term which is $\frac12$-order $p$-suboptimal. To the best of our knowledge, this is the first such estimate for HHO methods, whereas $h$-a posteriori error estimates for HHO methods were derived previously in \cite{di2016posteriori,bertrand2023stabilization,carstensen2024adaptive}, focusing on either the equal-order HHO setting or a stabilization-free variant of the method.

The main challenge in deriving an upper error bound for nonconforming methods is to estimate the nonconforming error, which essentially measures by how much the discrete solution departs from $H^1$. A first possibility is to invoke a nodal-averaging operator mapping to $H^1$, as done, e.g., in \cite{KP,Ernstephansen08,Ainsworth07,ErnVoharlik} for dG methods. However, $p$-optimal approximation results for nodal-averaging operators are so far available only on tensor-product \cite{BurmanErn07} and triangular \cite{hsw} meshes, whereas the best bound available on tetrahedral meshes is $p$-suboptimal by one order \cite[Lemma 7.6]{EggerWal13}.
An alternative to using a nodal-averaging operator is to invoke a (global) Helmholtz decomposition on the
nonconforming error, as in \cite{dari1996posteriori, Carstensen2002,Carstensen2007unifying} for Crouzeix--Raviart finite elements and in \cite{Becker03,cangiani2023aposteriori} for dG methods. Here, we adopt this technique, but we introduce a novel idea in that we additionally use the partition of unity provided by the hat basis functions to invoke a \emph{local} Helmholtz decomposition on each vertex star (the subdomain covered by the mesh cells sharing the vertex). The benefit is that each vertex star is simply connected, and recent results on the stability constant in the local Helmholtz decomposition are available \cite{GuzSal21}. Instead, the stability constant for a global Helmholtz decomposition in a domain with $N$ holes grows unfavorably with $N$ \cite{bertrand2023stabilization}. Finally, we emphasize that our novel idea for bounding the nonconforming error can be directly applied to the a posteriori analysis of many other nonconforming discretization methods, such as dG, HDG, and WG.

We also address the efficiency of our a posteriori error estimate by establishing local lower error bounds. Our main result is Theorem~\ref{thm:lowerbound} which leads to $\frac32$-order $p$-suboptimality (only the tangential jump estimator leads to such suboptimality). Numerical experiments, though, indicate only $\frac12$-order suboptimality in $p$. Our proof of the lower error bound uses bubble function techniques inspired from~\cite{MelenkWohlmuth01}, but we introduce a novel argument in the proof in that we invoke the local conservation property of the HHO method to improve the efficiency result on the normal flux jump from $\frac32$-order to $\frac12$-order $p$-suboptimality. Another interesting numerical observation is
that the normal flux jump is not the dominant component of the a posteriori error estimate for HHO methods, in contrast to the situation classically encountered with conforming finite elements
\cite{CarstensenVerfurth99}.

The rest of this work is organised as follows. We present the weak formulation of the model problem together with the discrete setting in Section \ref{sec:Weak form and HHO methods}. In Section \ref{sec:HHO}, we introduce the HHO method, and in Section \ref{sec:HHO a priori}, we derive the $hp$-a priori error estimate. In Section \ref{sec: a posteriori error analysis}, we present the residual-based $hp$-a posteriori error analysis, leading to a (global) upper error bound and (local) lower error bounds. Numerical experiments are presented in Section \ref{sec:Numerical example} to illustrate the theory. Finally, in Section~\ref{sec:proofs}, we collect several (technical) proofs related to the $hp$-a posteriori error analysis.

\section{Weak form and discrete setting}\label{sec:Weak form and HHO methods}
In this section, we introduce some basic notation, the weak formulation of the model problem, and the discrete setting to formulate and analyze the HHO discretization.

\subsection{Basic notation and weak formulation} \label{Model problem}

We use standard notation for the Lebesgue and Sobolev spaces and, in particular, for the fractional-order Sobolev spaces, we consider the Sobolev--Slobodeckij seminorm based on the double integral. For an open, bounded, Lipschitz set $S$ in $\mathbb{R}^d$, $d\in\{1,2,3\}$, we denote by $(v,w)_S$ the $L^2(S)$-inner product, and we employ the same notation when $v$ and $w$ are vector- or matrix-valued fields. We denote by $\nabla w$ the (weak) gradient of $w$. We use boldface notation to denote vectors in $\mathbb{R}^d$, as well as $\mathbb{R}^d$-valued fields and functional spaces composed of such fields.

Setting $H^1_{g,\ddd}(\Omega):= \{v\in H^1(\Omega)\:|\: v|_{\Gamma_{\ddd}} = g_\ddd \}$,
the weak formulation of \eqref{Problem} is as follows: Find $u\in H^1_{g,\ddd}(\Omega)$, such that
\begin{equation}\label{eq:prob}
(\diff{\nabla u}, {\nabla v})_{\Omega}  =  ( f , v )_{\Omega}  +(g_{\dn},v)_{\Gamma_{\dn}},
\end{equation}
for all $v\in H^1_{0,\ddd}(\Omega)$. The well-posedness of~\eqref{eq:prob} follows from the Lax--Milgram Lemma, see, e.g., \cite[Proposition 31.21]{Ern_Guermond_FEs_II_2021}.

\subsection{Mesh}

Let $\mesh$ be a simplicial mesh that covers the domain $\Omega$ exactly and is compatible with  the boundary partition as well as  the domain partition on which $A$ is piecewise constant.  A generic mesh cell is denoted by $K\in\mesh$, its diameter by $h_K$, and its unit outward normal by $\n_K$. We set $A_K:=A|_K$. We let $\mathrm{st}(K)$ denote the collection of cells $\hat{K}\in \mesh$ sharing at least one vertex with the cell $K$ ($\mathrm{st}(K)$ is often called cell star). Similarly, $\mathrm{es}(K)$ denotes the collection of cells $\hat{K}\in \mesh$ sharing at least one vertex with $\mathrm{st}(K)$ ($\mathrm{es}(K)$ is often called extended cell star). For all $k\ge0$, $\mathbb{P}^{k}(K)$ denotes the space of $d$-variate polynomials on $K$ of degree at most $k$, and $\Pi_K^k$ denotes the $L^2$-orthogonal projection onto $\mathbb{P}^{k}(K)$. Moreover, $\mathbb{P}^{k}(\mesh):=\{ v_h\in L^2(\Omega)\,|\, v_h|_K\in\mathbb{P}^{k}(K)\}$ denotes the broken polynomial space of order $k$ on the mesh $\mesh$ (classically considered in dG methods).

The mesh faces are collected in the set $\Fall$, which is split as $\Fall = \Fint \cup \Fb$, where $\Fint$ is the collection of interfaces (shared by two distinct mesh cells) and  $\Fb$ the collection
of boundary faces. Moreover, we split $\Fb$ into the Dirichlet subset, $\FD$, and the Neumann subset, $\FN$. Let $\n_F$ denote the unit normal vector orienting the mesh face $F\in\Fall$. For all $F\in \Fint$, the direction of $\n_F$ is arbitrary, but fixed, whereas we set $\n_F:=\n_{\Omega}|_F$ for all $F\in\Fb$. For every mesh cell $K\in\mesh$, the partition of its boundary $\dK$ is defined as $\dK=\dKi\cup\dKD\cup\dKN$ with obvious notation, and the mesh faces composing $\dK$ are collected in the set $\FK$, which is partitioned as $\FK=\FKi \cup \FKD \cup \FKN$ with obvious notation. For all $F\in\Fall$, $\Pi_F^k$ denotes
the $L^2$-orthogonal projection onto $\mathbb{P}^k(F)$.

The set of mesh vertices  is denoted by $\vertice$ and is decomposed into the subset of interior vertices, $\verticei$, and the subset of boundary vertices, $\verticeb$. For all $\ba\in \vertice$, $\mesha$ denotes the collection of mesh cells which share $\ba$ and $\oma$ the corresponding open subdomain (often called vertex star). In addition, we define $\Fa$ as the collection of faces in $\Fall$ which share $\ba$.

For all $s>\frac12$, we define the broken Sobolev spaces $H^s(\mesh;\mathbb{R}^q):= \{w\in L^2(\Omega;\mathbb{R}^q)\:|\:w_K:=w|_K\in H^s(K;\mathbb{R}^q), ~ \forall K\in \mesh\}$, $q\in\{1,d\}$. We define the jump $\jump{w}$ of any function $w\in H^s(\mesh; \mathbb{R}^q)$ across any mesh interface $F=\dK_1 \cap\dK_2\in \Fint$ as  $\jump{w} : = w_{K_1}|_F -  w_{K_2}|_F$, where  $\n_F$ points from $K_1$ to $K_2$. For any boundary face $F=\dK \cap \partial \Omega\in \Fb$, we set $\jump{w} : = w_{K}|_F$. For all $K\in\mesh$, we define $\jumpK{w}|_F:=\jump{w}$ for all $F\in \FK$. Finally, we define the broken gradient $\bnabla$ as the gradient operator acting  cellwise on $H^{1}(\mesh;\mathbb{R}^q)$.

\subsection{$hp$-analysis tools}\label{Analysis tools}

Let us briefly review the main $hp$-analysis tools used in this work.
We use the symbol $C$ (sometimes with a subscript) to denote any positive generic constant whose value can change at each occurrence as long as it is independent of the mesh size $h$ and the  polynomial degree $k$. The value of $C$ can  depend on the mesh shape-regularity and the space dimension $d$.

\begin{lemma}[Discrete trace inequality]\label{lemma: Inverse inequality}
The following holds for all $v\in\mathbb{P}^k(K)$, all $K\in\mesh$, and all $k\ge0$,
\begin{equation}\label{trace_inv}
\|v\|_{\dK} \leq C \frac{ (k+1)}{ h_K^{\frac12}}  \|v\|_{ K},
\end{equation}
\end{lemma}
\begin{proof}
A proof can be found in \cite[Theorem 5]{warburton2003constants} (with explicit constant in terms of $d$).
\end{proof}


\begin{lemma}[Local $L^2$-orthogonal projection]\label{lemma: L^2 projection}
The following holds for $v\in H^1(K)$, all $K\in\mesh$, and all $k\geq0$,
\begin{equation}\label{L2 projection Polynomial approximation}
\|{v} - \Pi^{k}_{ K} (v)\|_{\partial K}
\leq C \left(\frac{h_K}{k+1}\right)^{\frac12} |{v} |_{H^1(K)}.
\end{equation}
\end{lemma}
\begin{proof}
A proof can be found in \cite{Chernov12,Melenkurzer14}.
\end{proof}

\begin{lemma}[Local Babu\v{s}ka--Suri operator]\label{lemma: hp-Polynomial approximation}
There exists a positive constant $C_{\rm BS}$ such that, for all $k\geq1$ and all $K\in\mesh$, there exists an operator $\mathcal{I}^{k}_{{\rm BS},K} :L^2(K)\rightarrow \mathbb{P}^{k}(K)$, called Babu\v{s}ka--Suri approximation operator, such that,
for all $r\in\{0,\ldots,k\}$, all $m\in \{0,\ldots, r\}$, and all $v\in H^r(K)$,
\begin{equation}\label{Polynomial approximation}
|{v} - \mathcal{I}^{k}_{{\rm BS},K} (v)|_{H^m(K)}
\leq C_{\rm BS} \left(\frac{h_K}{k}\right)^{r-m} \|{v} \|_{H^r(K)}.
\end{equation}
\end{lemma}
\begin{proof}
A proof can be found in \cite{babuvska1987optimal}.
\end{proof}

\begin{lemma}[Global $hp$-Karkulik--Melenk operator]\label{lemma: hp-KM orginal}
There exists a constant $C_{\rm KM}$ such that, for all $k\geq1$, there exists an operator $\IKM: H^{1}_{0,\ddd}(\Omega) \rightarrow  \mathbb{P}^{k}(\mesh) \cap H^{1}_{0,\ddd} (\Omega)$, called Karkulik--Melenk interpolation operator, such that, for all $v\in H^{1}_{0,\ddd}(\Omega)$ and all $K\in\mesh$,
\begin{equation}\label{eq: KM approxiamtion}
\Big(\frac{k}{h_K}\Big)^{2}\|{v} - \IKM (v)\|^2_{K} + \Big(\frac{k}{h_K}\Big)\|{v} - \IKM (v)\|^2_{\partial K}  + \|\nabla \IKM (v)\|^2_{K}
 \leq C_{\rm KM} \|v\|^2_{H^1(\mathrm{st}(K))}.
\end{equation}
\end{lemma}
\begin{proof}
A proof can be found in \cite{melenk_rough,KarKulikMelenk15}.
\end{proof}

\begin{corollary}[Modified $hp$-Karkulik--Melenk operator]\label{cor: global interpolation of KM}
There exists a constant $C_{\rm mKM}$ such that, for all $k\geq1$, there exists an operator
$\ImKM:H^{1}_{0,\ddd}(\Omega) \rightarrow  \mathbb{P}^{k}(\mesh) \cap H^{1}_{0,\ddd} (\Omega)$, called modified Karkulik--Melenk interpolation operator, such that, for all $v\in H^{1}_{0,\ddd}(\Omega)$ and all $K\in\mesh$,
\begin{equation}\label{eq: Modified KM approximation}
\Big(\frac{k}{h_K}\Big)^{2}\|{v} - \ImKM (v)\|^2_{K} + \Big(\frac{k}{h_K}\Big)\|{v} - \ImKM (v)\|^2_{\partial K}  + \|\nabla \ImKM (v)\|^2_{K}
 \leq C_{\rm mKM} \| \nabla v\|^2_{\mathrm{es}(K)}.
\end{equation}
\end{corollary}
\begin{proof}
The idea is to set, for all $v\in H^{1}_{0,\ddd}(\Omega)$,
\[
\ImKM(v) : = I^1_{{\rm av},\ddd} (v) + \IKM (v -  I^1_{{\rm av},\ddd} (v) ),
\]
where the (first-order) nodal-averaging operator $I^1_{{\rm av},\ddd}:
H^{1}_{0,\ddd}(\Omega) \rightarrow  \mathbb{P}^{1}(\mesh) \cap H^{1}_{0,\ddd} (\Omega)$ is devised, e.g.,
in \cite{ErnGuermond:17} when $\Gamma_{\ddd}=\partial\Omega$ and in \cite{2017Smoothed} when $\Gamma_{\ddd}$ is
a proper subset of $\partial\Omega$. This operator satisfies, for all $v\in H^{1}_{0,\ddd}(\Omega)$,
\begin{equation}\label{eq: quasi approxiamtion}
\Big(\frac{1}{h_K}\Big)^{2}\|{v} - I^1_{{\rm av},\ddd} (v)\|^2_{K} + \Big(\frac{1}{h_K}\Big)\|{v} - I^1_{{\rm av},\ddd} (v)\|^2_{\partial K}  + \|\nabla I^1_{{\rm av},\ddd} (v)\|^2_{K}
\leq C \|\nabla v\|^2_{\mathrm{st}(K)}.
\end{equation}
We can now prove \eqref{eq: Modified KM approximation}. Using the approximation results \eqref{eq: KM approxiamtion} and \eqref{eq: quasi approxiamtion} and the mesh shape-regularity, we infer that
\begin{equation*}
\begin{split}
\Big(\frac{k}{h_K}\Big)^{2}\|{v} - \ImKM (v)\|^2_{K}  &= \Big(\frac{k}{h_K}\Big)^{2} \|({v} -  I^1_{{\rm av},\ddd} (v)) - \IKM (v -  I^1_{{\rm av},\ddd} (v) )\|^2_{K} \\
& \leq C \|v -  I^1_{{\rm av},\ddd} (v)\|^2_{H^1(\mathrm{st}(K))}
\leq C \sum_{\hat{K}\in \mathrm{st}(K)}  \|\nabla v\|^2_{\mathrm{st}(\hat{K})}
\leq  C \|\nabla v\|^2_{\mathrm{es}(K)}.
\end{split}
\end{equation*}
This proves the bound on the first term on the left-hand side of~\eqref{eq: Modified KM approximation}, and
the other two terms can be bounded in a similar way.
\end{proof}

\begin{remark}[Applications]
We use the Babu\v{s}ka--Suri approximation operator to establish our $hp$-a priori
error estimate. Instead, we use the modified Karkulik--Melenk interpolation operator in the $hp$-a posteriori error analysis to establish the upper error bound. The advantage of the modified Karkulik--Melenk interpolation operator with respect to the original one is to invoke only the $H^1$-seminorm on the right-hand side of~\eqref{eq: Modified KM approximation} (compare with~\eqref{eq: KM approxiamtion}).
\end{remark}

\section{HHO method} \label{sec:HHO}

Let $k\geq 0$ be the polynomial degree. We focus on the mixed-order HHO method where, for all $K\in \mesh$, the local HHO space is
\begin{equation}\label{def: HHO space}
\fesE: =\mathbb{P}^{k+1}(K)
\times
\mathbb{P}^{k}(\FK) ,\qquad \mathbb{P}^{k}(\FK) : = \bigtimes_{F\in \FK} \mathbb{P}^{k}(F).
\end{equation}
A generic element in  $\fesE$ is denoted by $\hat{v}_K := (v_K, v_{\dK})$ with $v_K \in \mathbb{P}^{k+1}(K)$ and $v_{\dK} \in \mathbb{P}^{k}(\FK)$. The first component of the pair $\hat{v}_K$ aims at representing the solution inside the
mesh cell and the second its trace at the cell boundary.

\subsection{Reconstruction and stabilization}

The HHO method is formulated locally by means of a reconstruction and a stabilization operator. The local reconstruction operator $R_K^{k+1}: \fesE \rightarrow \mathbb{P}^{k+1}(K)$ is such that, for all $\hat{v}_K:= (v_K,v_{\dK}) \in \fesE$,
$R_K^{k+1}(\hat{v}_K)\in \mathbb{P}^{k+1}(K)$ is determined by solving the following well-posed problem:
\begin{equation}\label{reconstruction}
\begin{aligned}
( \nabla R_K^{k+1}(\hat{v}_K), \nabla w)_K :={}& (\nabla  {v}_K, \nabla w)_K
- (v_K -v_{\dK} , \nabla  w {\cdot} \n_K)_{\dK},
\end{aligned}
\end{equation}
for all $w \in \mathbb{P}^{k+1}(K) / \mathbb{R}$ and $(R_K^{k+1}(\hat{v}_K), 1 )_K =  ( v_K,1 )_K$. (Notice that \eqref{reconstruction} actually holds for all $w \in \mathbb{P}^{k+1}(K)$). Integration by parts gives
\begin{equation}\label{eq:rec_ipp}
( \nabla R_K^{k+1}(\hat{v}_K), \nabla w)_K = -(v_K,  \Delta  w)_K
+ (v_{\dK} ,   \nabla  w {\cdot} \n_K)_{\dK}.
\end{equation}
The local stabilization bilinear form $S_{\dK}$ is defined such that, for all $(\hat{v}_K, \hat{w}_K)\in \fesE \times \fesE$,
\begin{equation}\label{def: stabilization}
\begin{aligned}
S_{\dK}(\hat{v}_K,\hat{w}_K)
:={}&\frac{ (k+1)^2}{h_K}
\big( v_{\dK}-\Pi^k_{\dK}(v_K|_{\dK}), w_{\dK}-\Pi^k_{\dK}(w_K|_{\dK})\big)_{\dK},
\end{aligned}
\end{equation}
where $\Pi^k_{\dK}$ denotes the $L^2$-orthogonal projection onto $\mathbb{P}^{k}(\FK)$.
The reconstruction and stabilization operators are combined together to build
the local bilinear form $a_K$ on $\fesE \times \fesE$ such that, for all $K\in \mesh$,
\begin{equation} \label{eq:def_aK}
a_K(\hat{v}_K,\hat{w}_K):= \diff_K ( \nabla R_K^{k+1}(\hat{v}_K),\nabla R_K^{k+1}(\hat{w}_K))_K
+ \diff_K S_{\dK}(\hat{v}_K,\hat{w}_K).
\end{equation}

\begin{lemma}[Useful property]
The following holds for all $\hat{v}_K \in \fesE$ and all $K\in\mesh$:
\begin{equation} \| \nabla (R_K^{k+1}(\hat{v}_K)-v_K) \|_{K}^2
\leq
C   S_{\dK}(\hat{v}_K,\hat{v}_K). \label{local HHO  bound}
\end{equation}
\end{lemma}

\begin{proof}
Using the definition \eqref{reconstruction} of the reconstruction operator with $w := R_K^{k+1}(\hat{v}_K)-v_{K} \in \mathbb{P}^{k+1}(K)$ gives
\begin{equation*}
\begin{aligned}
\| \nabla (R_K^{k+1}(\hat{v}_K)-v_K) \|_{K}^2 :={}&
- (v_K -v_{\dK} ,  \nabla  (R_K^{k+1}(\hat{v}_K)-v_{K}) {\cdot} \n_K)_{\dK}.
\end{aligned}
\end{equation*}
Since $  \nabla (R_K^{k+1}(\hat{v}_K)-v_{K}) {\cdot} \n_K \in  \mathbb{P}^{k}(\FK)$, using the Cauchy--Schwarz inequality, the definition \eqref{def: stabilization} of the stabilization, and the discrete trace inequality \eqref{trace_inv} implies that
\begin{equation*}
\begin{aligned}
\| \nabla (R_K^{k+1}(\hat{v}_K)-v_K) \|_{K}^2 ={}&
- (\Pi_{\dK}^{k} (v_K|_{\dK}) -v_{\dK}, \nabla  (R_K^{k+1}(\hat{v}_K)-v_{K}){\cdot} \n_{K} )_{\dK}\\
& \leq C  S_{\dK}(\hat{v}_K,\hat{v}_K)^{\frac12} \| \nabla (R_K^{k+1}(\hat{v}_K)-v_K) \|_{K}.
\end{aligned}
\end{equation*}
This concludes the proof of \eqref{local HHO  bound}.
\end{proof}

\subsection{Global discrete problem}

We define the global HHO space as
\begin{equation}
\fes : = \mathbb{P}^{k+1}(\mesh) \times \mathbb{P}^{k}(\Fall), \quad \mathbb{P}^{k+1}(\mesh): = \bigtimes_{K\in \mesh}\mathbb{P}^{k+1}(K), \quad  \mathbb{P}^{k}(\Fall) : =  \bigtimes_{F\in \Fall} \mathbb{P}^{k}(F).
\end{equation}
A generic element in $\fes$ is denoted by $\hat{v}_h:=(v_{\mesh},v_{\Fall})$
with $v_{\mesh}:=(v_K)_{K\in\mesh}$ and $v_{\Fall}:=(v_F)_{F\in\Fall}$. For all $K\in \mesh$, the local components of $\hat{v}_h$  are collected in the pair
$\hat{v}_K:=(v_K,v_{\dK})\in\fesE$ with $v_{\dK}|_F: = v_F$  for all $F \in \FK$. Similarly, let $R^{k+1}_{\mesh}(\hat{v}_h)\in  \mathbb{P}^{k+1}(\mesh)$ be  such that  $R^{k+1}_{\mesh}(\hat{v}_h)|_K = R_K^{k+1}(\hat{v}_K)$ for all $K\in \mesh$. To deal with the  Dirichlet boundary condition, we define  the (affine) subspaces
\begin{subequations}
\begin{align}
&\fesD : = \{\hat{v}_h\in \fes\;|\; v_F=0,\;\forall F\in\FD\}, \\
&\fesS : = \{\hat{v}_h\in \fes\;|\; v_F= \Pi_{F}^{k} (g_{\ddd}|_F),\;\forall F\in\FD\}.
\end{align}
\end{subequations}
The discrete HHO problem is as follows: Find $\hat{u}_h\in \fesS$ such that
\begin{equation}\label{discrete problem}
a_h(\hat{u}_h,\hat{w}_h) = \ell_h(\hat{w}_{h}), \qquad \forall \hat{w}_h\in \fesD,
\end{equation}
where the global discrete bilinear form $a_h$ and the global linear form $\ell_h$ are assembled cellwise as
\begin{equation}\label{HHO:bilinear form}
a_h(\hat{v}_h, \hat{w}_h):= \su a_K(\hat{v}_K, \hat{w}_K), \qquad
\ell_h(\hat{w}_{h}):=  \su \big\{ (f,{w}_K)_K + (g_{\dn},w_{\dK})_{\dKN}\big\}.
\end{equation}
It is well-known that the discrete problem~\eqref{discrete problem} is amenable to static condensation, i.e.,
the cell unknowns can be eliminated locally in every mesh cell, leading to a global problem where the only remaining unknowns
are those attached to the mesh faces, i.e.,
those in $ \mathbb{P}^k(\Fall\backslash \FD)$.

An important property of the HHO method we exploit in the a posteriori error analysis is local conservation. For all $K\in\mesh$ and all $F\in \FK$,
we define the flux
\begin{equation}\label{eq:equilibrated fluxes}
\phi_{K,F}(\hat{u}_K):= -A_K\nabla R_K^{k+1}(\hat{u}_K) {\cdot} \n_K|_F + A_K\frac{  (k+1)^2} {h_K}\Pi_{F}^k ({u_K}|_F) - {u_{F}} \in \mathbb{P}^{k}(F).
\end{equation}
Then, the following holds true \cite{CoDPE:16}: (i) At every interface $F= \dK_1\cap \dK_2 \in\Fint$, we have
\begin{subequations}\label{conservation property}
\begin{equation}\label{eq:equilibrated fluxes interior faces}
\phi_{K_1,F}(\hat{u}_{K_1}) + \phi_{K_2,F}(\hat{u}_{K_2})  =0;
\end{equation}
(ii) At every Neumann boundary face $F= \dK\cap \Gamma_{\dn} \in\FN$, we have
\begin{equation}\label{eq:equilibrated fluxes Neumann faces}
\phi_{K,F}(\hat{u}_{K}) +\Pi^{k}_{F}(g_{\dn}|_F)  =0.
\end{equation}
\end{subequations}

\subsection{Stability and well-posedness}

We equip the local HHO space $\fesE$  with the $H^1$-like seminorm such that, for all $\hat{v}_K\in \fesE$,
\begin{equation}\label{H2_seminorm_elem}
|\hat{v}_K|^2_{\fesE}: = \|\nabla v_K\|_K^2
+ \frac{ (k+1)^2}{h_K} \| v_{\dK} - \Pi_{\dK}^k(v_K|_{\dK})\|_{\dK}^2.
\end{equation}

\begin{lemma}[Local stability and boundedness] \label{lem:stab_bnd}
There is a real number $\alpha>0$, depending only on the mesh shape-regularity and the space dimension $d$, such that, for all $\hat{v}_K\in \fesE$ and all $K\in\mesh$,
\begin{equation}\label{local equivalent}
\alpha|\hat{v}_K|^2_{\fesE}
\leq \|\nabla R_K^{k+1}(\hat{v}_K) \|_{ K}^2
+ S_{\dK}(\hat{v}_K,\hat{v}_K)
\leq \alpha^{-1} |\hat{v}_K|^2_{\fesE}.
\end{equation}
\end{lemma}
\begin{proof}
The proof proceeds as in \cite{CoDPE:16} using the discrete trace inequality \eqref{trace_inv}; this is the reason why the stabilization is scaled by $(k+1)^2$ and not just by $(k+1)$.
\end{proof}

We equip the space $\fesD$ with the norm
\begin{equation}\label{H1_seminorm}
\|\hat{v}_h\|^{2}_{\textsc{hho}}:= \su A_K |\hat{v}_K|^2_{\fesE},\qquad \forall \hat{v}_h \in \fesD.
\end{equation}
The fact that $\|{\cdot}\|_{\textsc{hho}}$ defines indeed a norm on $\fesD$ is shown, e.g., in \cite[Lemma~1.6]{cicuttin2021hybrid}.

\begin{corollary}[Coercivity and well-posedness]\label{lemma: coercivit}
The discrete bilinear form $a_h$ is coercive on $\fesD$, and the discrete problem \eqref{discrete problem} is well-posed.
\end{corollary}

\begin{proof}
Summing \eqref{local equivalent} multiplied by $A_K$ over all the mesh cells shows the following coercivity and continuity properties:
\begin{equation}\label{coercivity}
\alpha\|\hat{v}_h\|_{\textsc{hho}}^2 \leq a_h(\hat{v}_h,\hat{v}_h)  \leq \alpha^{-1} \|\hat{v}_h\|_{\textsc{hho}}^2, \qquad \forall \hat{v}_h\in\fesD.
\end{equation}
The well-posedness of \eqref{discrete problem} with homogeneous Dirichlet boundary conditions then follows from the Lax--Milgram lemma. For inhomogeneous Dirichlet boundary conditions, the proof follows by introducing a lifting, say $\hat{g}_{h} \in \fesS$, such that $g_F= \Pi_{F}^{k} (g_{\ddd}|_F)$, for all $F\in\FD$.
\end{proof}

\section{$hp$-a priori error estimate} \label{sec:HHO a priori}

In this section, we establish our $hp$-a priori error estimate.

\subsection{Approximation}

For all $K\in \mesh$, we define the local reduction operator
$\mathcal{\hat{I}}^k_K : H^1(K) \rightarrow \fesE$ such that, for all $v\in H^1(K)$,
\begin{equation}\label{interpoltation}
\mathcal{\hat{I}}^k_K(v):= \big(\Pi_{K}^{k+1}(v),\Pi_{\dK}^k (v|_{\dK})\big) \in \fesE.
\end{equation}
Moreover, the elliptic projection $\mathcal{E}^{k+1}_K:H^1(K)\rightarrow \mathbb{P}^{k+1}(K)$ is defined such that
\begin{equation}\label{ell_proj}
\begin{alignedat}{2}
(\nabla (\mathcal{E}^{k+1}_K(v) - v), \nabla w)_K &= 0,&\qquad&\forall w \in  \mathbb{P}^{k+1}(K) /\mathbb{R}, \\
(\mathcal{E}^{k+1}_K(v)- v, 1)_K &=0.&
\end{alignedat}
\end{equation}
One readily verifies by proceeding as in \cite[Lemma~3]{DiPEL:14} that
\begin{equation}
R_K^{k+1}\circ \mathcal{\hat{I}}^k_K = \mathcal{E}^{k+1}_K.  \label{elliptic projection}
\end{equation}

\begin{lemma}[Bound on stabilization] \label{lem:local_red}
The following holds for all $K\in\mesh$ and all $v\in H^1(K)$:
\begin{equation}
S_{\dK}(\mathcal{\hat{I}}^k_K(v),\mathcal{\hat{I}}^k_K(v))\leq
C(k+1) \|\nabla (v - \mathcal{I}^{k+1}_{{\rm BS},K}(v) \|_{K}^2. \label{stabilization bound}
\end{equation}
\end{lemma}

\begin{proof}
Recalling the definition \eqref{def: stabilization} of $S_{\dK}$, the definition  \eqref{interpoltation} of $\mathcal{\hat{I}}^k_K$ and since $\Pi^k_{\dK} \circ \Pi^{k}_{\dK} = \Pi^k_{\dK}$, we have
\begin{equation*}
S_{\dK}(\mathcal{\hat{I}}^k_K(v),\mathcal{\hat{I}}^k_K(v))
=
\frac{(k+1)^2}{h_K}
\| \Pi^{k}_{\dK}((v- \Pi^{k+1}_K (v))|_{\dK}) \|_{\dK}^2
\le
\frac{(k+1)^2}{h_K}
\|v- \Pi^{k+1}_K (v) \|_{\dK}^2,
\end{equation*}
where we used the  $L^2(\dK)$-stability of $\Pi_{\dK}^k$.
Then, we invoke the approximation result on the $L^2$-orthogonal projection, see \eqref{L2 projection Polynomial approximation},  and that  $ \Pi^{k+1}_{K} \circ \mathcal{I}^{k+1}_{{\rm BS},K}= \mathcal{I}^{k+1}_{{\rm BS},K}$, giving
\begin{align*}
\frac{(k+1)^2}{h_K}
\| v- \Pi^{k+1}_K (v) \|_{\dK}^2
&=
\frac{(k+1)^2}{h_K}
\| v-  \mathcal{I}^{k+1}_{{\rm BS},K}(v) - \Pi^{k+1}_K ( v - \mathcal{I}^{k+1}_{{\rm BS},K}(v)) \|_{\dK}^2 \\
& \le
C(k+1)  \| \nabla (v - \mathcal{I}^{k+1}_{{\rm BS},K}(v) \|_{K}^2.
\end{align*}
Combining the above two bounds  proves \eqref{stabilization bound}.
\end{proof}

For all $K\in\mesh$ and all $v\in H^{1+s}(K)$, $s>\frac12$, we consider the following norm:
\begin{equation}\label{def: sharp norm}
\begin{aligned}
\|v\|^2_{\sharp, K}:=  \|\nabla v\|_K^2
+
\frac{h_K}{k+1}  \|\nabla v \|_{\dK}^2.
\end{aligned}
\end{equation}

\begin{lemma}[Approximation]\label{lem: approxmation}
The following holds for all $K\in \mesh$ and all $v\in H^{1+s}(K)$, $s>\frac{1}{2}$:
\begin{equation}\label{error bound elliptic projection}
\|v - \mathcal{E}^{k+1}_K (v)\|_{\sharp, K} \leq C (k+1)^{\frac12} \|v -  \mathcal{I}^{k+1}_{{\rm BS},K}(v)\|_{\sharp, K}.
\end{equation}
\end{lemma}

\begin{proof}
Using the triangle inequality, we have
\begin{equation*}
\|v - \mathcal{E}^{k+1}_K (v)\|_{\sharp, K} \leq \|v - \mathcal{I}^{k+1}_{{\rm BS},K}(v)\|_{\sharp, K} + \| \mathcal{I}^{k+1}_{{\rm BS},K}(v)-  \mathcal{E}^{k+1}_{K}(v)\|_{\sharp, K},
\end{equation*}
so that we only need to bound the second term on the right-hand side. Owing to the discrete trace inequality \eqref{trace_inv}, we infer that
\begin{equation*}
\|\mathcal{E}^{k+1}_K(v) - \mathcal{I}^{k+1}_{{\rm BS},K} (v)\|_{\sharp, K} \le C(k+1)^{\frac12}\|\nabla(\mathcal{E}^{k+1}_K(v) - \mathcal{I}^{k+1}_{{\rm BS},K}(v))\|_K,
\end{equation*}
and it remains to bound $\|\nabla(\mathcal{E}^{k+1}_K(v) -  \mathcal{I}^{k+1}_{{\rm BS},K}(v))\|_K$.
Since $\mathcal{E}^{k+1}_K \circ  \mathcal{I}^{k+1}_{{\rm BS},K}=   \mathcal{I}^{k+1}_{{\rm BS},K}$, we infer that
\begin{align*}
\|\nabla(\mathcal{E}^{k+1}_K(v) - \mathcal{I}^{k+1}_{{\rm BS},K}(v))\|_K
= \|\nabla(\mathcal{E}^{k+1}_K(v - \mathcal{I}^{k+1}_{{\rm BS},K}(v)))\|_K
\le \|\nabla(v - \mathcal{I}^{k+1}_{{\rm BS},K}(v))\|_K,
\end{align*}
where the last bound follows from the stability of the elliptic projection. Combining the above bounds completes the proof.
\end{proof}

\subsection{Consistency}

Let $u\in H^1_{g,\ddd}(\Omega)$ be the exact solution to~\eqref{eq:prob}.
We define the consistency error $\delta_h \in (\fesD)^\prime$ such that
\begin{equation} \label{eq:def_delta}
\langle \delta_h,\hat{w}_h \rangle:= \ell_h(\hat{w}_h) - a_h(\Ihk(u), \hat{w}_h),
\qquad \forall \hat{w}_h\in \fesD,
\end{equation}
where $\langle{\cdot},{\cdot}\rangle$ denotes the duality pairing between $(\fesD)'$ and $\fesD$, and where
the global reduction operator $\Ihk:H^1(\Omega)\rightarrow \fes$ is defined
such that, for all $v\in H^1(\Omega)$,
\begin{equation} \label{def:Ihk}
\Ihk(v):=\big( (\Pi_K^{k+1}(v|_K))_{K\in\mesh},(\Pi_F^{k}(v|_F))_{F\in\Fall} \big)\in \fes,
\end{equation}
observing that $v$  is single-valued on every $F\in\Fint$. Notice that the local components of $\Ihk(v)$ attached to $K$ and its faces are
$\mathcal{\hat{I}}^k_K(v|_K)$ for all $K\in\mesh$.
The above definition of the consistency error is classical in the context of HHO methods. It is also rather natural since, as for most nonconforming discretization methods, it avoids defining any extension of the discrete bilinear form. Examples for other nonconforming methods can be found, e.g., in \cite[Part~VIII]{Ern_Guermond_FEs_II_2021}.

\begin{lemma}[Consistency]\label{lemma:consistency}
Assume that $u\in H^{1+s}(\Omega)$ with $s>\frac12$. The following holds true:
\begin{equation} \label{consistency}
  \|\delta_h\|_{\textsc{hho}'}:= \sup_{\hat{w}_h\in\fesD}
\frac{|\langle \delta_h,\hat{w}_h \rangle|}{\|\hat{w}_h\|_{\textsc{hho}}}
\leq C \left( \su A_K(k+1)\|u- \mathcal{I}^{k+1}_{{\rm BS},K} (u)\|^2_{\sharp,K}\right)^{\frac12}.
\end{equation}
\end{lemma}

\begin{proof}
Let $\hat{w}_h\in \fesD$. Using the definition of $\ell_h$ in \eqref{HHO:bilinear form},  the PDE and the boundary conditions satisfied by the exact solution $u$, and integrating by parts cellwise, we infer that
\begin{equation*}
\ell_h(\hat{w}_h) = \su A_K\Big\{ (\nabla u, \nabla {w}_K)_K
- ( \nabla u{\cdot} \n_K , w_K)_{\dK}
+ ( \nabla u{\cdot} \n_K  ,w_{\dK})_{\dKN} \Big\}.
\end{equation*}
The assumption $u \in H^{1+s} (\Omega)$ with $s>\frac12$ implies that $(\nabla u {\cdot}\n_K)|_{\dK}$ is meaningful  in $L^2(\dK)$ and single-valued at every mesh interface. Moreover, since $w_{\dK}$ is single-valued on $\dKi$ and vanishes on $\dKD$, we infer that
\begin{align*}
\ell_h(\hat{w}_h) = \su A_K\Big\{ (\nabla u, \nabla {w}_K)_K
- ( \nabla u{\cdot} \n_K , w_K - w_{\dK})_{\dK} \Big\}.
\end{align*}
Since $a_h$ is assembled cellwise and the local
components of $\Ihk(u)$ are $\mathcal{\hat{I}}^k_K(u|_K)$ for all $K\in\mesh$,
we infer that
$a_h(\Ihk(u), \hat{w}_h) = \su a_K(\mathcal{\hat{I}}^k_K(u|_K),\hat{w}_K)$.
Using the definition \eqref{eq:def_aK} of $a_K$, the definition  \eqref{reconstruction} of
$R_K^{k+1}(\hat{w}_K)$ and the identity \eqref{elliptic projection} leads to
\begin{align*}
a_h(\Ihk(u), \hat{w}_h) =&{}
\su A_K\Big\{ ( \nabla \mathcal{E}^{k+1}_K(u|_K), \nabla {w}_K)_K
- ( \nabla \mathcal{E}^{k+1}_K(u|_K) {\cdot} \n_K, w_K - w_{\dK})_{\dK} \\
&+ S_{\dK}(\mathcal{\hat{I}}^k_K(u|_K),\hat{w}_K) \Big\}.
\end{align*}
Defining the function $\eta$ cellwise as $\eta|_K:=u|_K-\mathcal{E}^{k+1}_K(u|_{K})$ for all $K\in\mesh$,
we infer that
\begin{equation}\label{the error relation}
\begin{aligned}
\langle \delta_h,\hat{w}_h \rangle
= {} \su A_K\Big\{ 
- ( \nabla\eta{\cdot} \n_K, w_K - w_{\dK})_{\dK} - S_{\dK}(\mathcal{\hat{I}}^k_K(u|_K),\hat{w}_K) \Big\},
\end{aligned}
\end{equation}
where we used that $(\nabla \eta,\nabla w_K)_K=0$.
Let us denote by $\mathcal{T}_{1,K}$ the first term  and by
$\mathcal{T}_{2,K}$ the second term inside braces on the right-hand side of~\eqref{the error relation}. We bound $\mathcal{T}_{1,K}$ by the Cauchy--Schwarz inequality, the triangle inequality, the trivial bound $\|w_{K} -\Pi_{\dK}^{k}(w_K|_{\dK})\|_{\dK} \le \|w_{K} -\Pi_{K}^{k}(w_K)\|_{\dK}$,
and the approximation result \eqref{L2 projection Polynomial approximation} on the $L^2$-projection. Recalling the definition~\eqref{def: sharp norm} of the
$\|{{\cdot}}\|_{\sharp,K}$-norm, this yields
\begin{equation*}
\begin{aligned}
|\mathcal{T}_{1,K}|
&
\le \|\nabla \eta\|_{\dK} \big( \|\Pi_{\dK}^{k}(w_K|_{\dK})-w_{\dK}\|_{\dK} + \|w_{K} -\Pi_{\dK}^{k}(w_K|_{\dK})\|_{\dK}\big) \\
&\le \Big( \frac{h_K}{k+1}\Big)^{\frac12} \|\nabla \eta\|_{\dK}
\Big( \frac{ k+1}{h_K}\Big)^{\frac12}  \big(   \|\Pi_{\dK}^{k}(w_K|_{\dK})-w_{\dK}\|_{\dK} + \|w_{K} -\Pi_{K}^{k}(w_K)\|_{\dK}\big) \\
& \le \| \eta \|_{\sharp,K} \left( \Big( \frac{ k+1}{h_K}\Big)^{\frac12}   \|\Pi_{\dK}^{k}(w_K|_{\dK})-w_{\dK}\|_{\dK} +  C \|\nabla w_{K} \|_{K} \right)
\le C     \| \eta \|_{\sharp,K}|\hat{w}_K|_{\fesE}.
\end{aligned}
\end{equation*}
Moreover, owing to~\eqref{stabilization bound} and the upper bound in~\eqref{local equivalent}, we have
\begin{equation*}
|\mathcal{T}_{2,K}| \le S_{\dK}(\mathcal{\hat{I}}^k_K(u|_K),\mathcal{\hat{I}}^k_K(u|_K))^{\frac12}
S_{\dK}(\hat{w}_K,\hat{w}_K)^{\frac12} \le C     (k+1)^{\frac12}\| \nabla(u - \mathcal{I}^{k+1}_{{\rm BS},K}(u))\|_{K}|\hat{w}_K|_{\fesE}.
\end{equation*}
Altogether, this implies that
\begin{equation*}
|\langle \delta_h,\hat{w}_h \rangle|
\le C  \left( \su A_K\Big\{\| \eta \|^2_{\sharp,K} + (k+1)\| \nabla(u - \mathcal{I}^{k+1}_{{\rm BS},K}(u))\|_{K}^2 \Big\}
\right)^{\frac12} \|\hat{w}_h\|_{\textsc{hho}}.
\end{equation*}
Invoking Lemma~\ref{lem: approxmation} to bound $\| \eta \|_{\sharp,K}$ completes the proof.
\end{proof}

\subsection{Error estimate}\label{sec: a priori error estimate}

We are now ready to establish our main result concerning the $hp$-a priori error analysis. The estimate is $\frac12$-order $p$-suboptimal.

\begin{theorem}[$hp$-a priori error estimate]\label{Theorem: a priori}
Let $u$ be the weak solution to~\eqref{eq:prob}, and let $\hat{u}_h$
be the discrete solution to~\eqref{discrete problem}.
Assume that $u\in H^{1+s}(\Omega)$ with $s>\frac{1}{2}$. The following holds:
\begin{subequations}
\begin{equation} \label{eq: R err1}
\su  A_K\Big\{ \|\nabla (u-R_K^{k+1}(\hat{u}_K))\|_K^2 + \|\nabla (u-u_K)\|_K^2 +S_{\dK}(\hat{u}_K,\hat{u}_K) \Big\} \leq C 	\su A_K(k+1) \|u- \mathcal{I}^{k+1}_{{\rm BS},K}(u)\|_{\sharp,K}^2.
\end{equation}
Moreover,  if  $u \in H^{l+1}(\mesh; \mathbb{R})$, with $l\in \{1,\dots, k+1\}$, we have
\begin{equation} \label{eq:err2}
\su A_K  \Big\{\|\nabla (u-R_K^{k+1}(\hat{u}_K))\|_K^2 + \|\nabla (u-u_K)\|_K^2
 + S_{\dK}(\hat{u}_K,\hat{u}_K) \Big\}
\leq C \su A_K (k+1) \Big(\frac{ h_K}{k+1 }\Big)^{2l}\|u\|_{H^{l+1}(K)}^2.
\end{equation}
\end{subequations}
\end{theorem}

\begin{proof}
(i) Proof of \eqref{eq: R err1}. We set $\hat{e}_h : = \Ihk(u) - \hat{u}_h$ and observe that $\hat{e}_h\in \fesD$. Moreover, since $a_h(\hat{e}_h,\hat{e}_h)= -\langle \delta_h,\hat{e}_h \rangle$, the coercivity property \eqref{coercivity} implies that
\begin{equation*}
\alpha \|\hat{e}_h\|_{\textsc{hho}}^2 \le a_h(\hat{e}_h,\hat{e}_h)= -\langle \delta_h,\hat{e}_h \rangle
\le \|\delta_h\|_{\textsc{hho}'} \|\hat{e}_h\|_{\textsc{hho}},
\end{equation*}
so that $\|\hat{e}_h\|_{\textsc{hho}}\le \frac{1}{\alpha}\|\delta_h\|_{\textsc{hho}'}$.
Since $\su A_K \big\{\|\nabla R_K^{k+1}(\hat{e}_K)\|_K^2 + S_{\dK}(\hat{e}_K,\hat{e}_K) \big\}\le \alpha^{-1} \|\hat{e}_h\|_{\textsc{hho}}^2$ follows from \eqref{coercivity}, we infer from Lemma \ref{lemma:consistency} that
\begin{equation*}
\su A_K \Big\{ \|\nabla R_K^{k+1}(\hat{e}_K)\|_K^2 + S_{\dK}(\hat{e}_K,\hat{e}_K) \Big\}
\leq C \su A_K (k+1)\|u-\mathcal{I}^{k+1}_{{\rm BS},K}(u)\|^2_{\sharp,K}.
\end{equation*}
Since $u-R_K^{k+1}(\hat{u}_K)=(u-\mathcal{E}^{k+1}_K(u))+R_K^{k+1}(\hat{e}_K)$ because $\hat{u}_K = \hat{I}^k_K(u) - \hat{e}_K$,
the triangle inequality combined with Lemma~\ref{lem:local_red}, Lemma~\ref{lem: approxmation},  and the above bound
proves that
$$
\su  A_K\Big\{ \|\nabla (u-R_K^{k+1}(\hat{u}_K))\|_K^2  +S_{\dK}(\hat{u}_K,\hat{u}_K) \Big\} \leq C 	\su A_K(k+1) \|u- \mathcal{I}^{k+1}_{{\rm BS},K}(u)\|_{\sharp,K}^2.
$$
Moreover, using $\|\nabla (u-u_K)\|_K^2 \leq 2(\|\nabla (u-R_K^{k+1}(\hat{u}_K))\|_K^2 + \|\nabla (u_K-R_K^{k+1}(\hat{u}_K))\|_K^2)$ and  \eqref{local HHO  bound}, we infer that
\begin{align*}
\su A_K \|\nabla (u-u_K)\|_K^2 &\leq 2\su  A_K\Big\{ \|\nabla (u-R_K^{k+1}(\hat{u}_K))\|_K^2 + CS_{\dK}(\hat{u}_K,\hat{u}_K) \Big\} \\
&\leq C	\su A_K(k+1) \|u- \mathcal{I}^{k+1}_{{\rm BS},K}(u)\|_{\sharp,K}^2.
\end{align*}
This completes the proof of  \eqref{eq: R err1}.

(ii) The proof of  \eqref{eq:err2} follows by invoking on all $K\in\mesh$ the approximation properties of $\mathcal{I}_{{\rm BS},K}^{k+1}$
(see Lemma~\ref{lemma: hp-Polynomial approximation}), together with a multiplicative trace inequality (see, e.g.,
\cite[Sec.~12.3.2]{Ern_Guermond_FEs_I_2021}) to estimate the term involving the normal derivative in $\|u- \mathcal{I}^{k+1}_{{\rm BS},K}(u)\|_{\sharp,K}$.
\end{proof}

\section{$hp$-a posteriori error analysis}\label{sec: a posteriori error analysis}

In this section, we perform the $hp$-a posteriori error analysis of the above HHO discretization.
We establish both (global) upper and (local) lower error bounds.
Since we are interested in the $hp$-a posteriori error analysis, we only consider in what follows
the case $k\geq1$ so that we can invoke the (modified) Karkulik--Melenk interpolation operator. 
We shall also assume that the Dirichlet datum satisfies $g_{\ddd} \in H^{\frac12}(\partial\Omega)$. For all $K\in\mesh$, we define the following data oscillation terms:
\begin{subequations} \label{eq:def_oscill} \begin{align}
O_{K}(f) &:= A_K^{-\frac12}\Big(\frac{h_K}{k+1}\Big)\ltwo{f - \Pi_{K}^{k+1}(f)}{K}, \label{eq:def_Okf}\\
O_{K}(g_\dn)&:= A_K^{-\frac12}\Big(\frac{h_K}{k+1}\Big)^{\frac12}\ltwo{g_\dn-\Pi_{\dK}^{k} (g_\dn|_{\dK}) }{\dKN},\\
O_{K}(g_{\ddd}) &:= A_K^{\frac12}\Big(\frac{ h_K}{k+1}\Big)^{\frac12} \ltwo{ \nabla  (g_\ddd-  \Pi_{\dK}^{k+1} (g_{\ddd}|_{\dK})){\times} \n_{\Omega} }{\dKD}.
\end{align} \end{subequations}

It is useful to define some contrast factors related to the diffusion coefficient $A$.
For all $K\in\mesh$ and all $\ba \in \vertice$, we set
\begin{equation} \label{eq:def_contrast}
\chi_K(A) := A_K \max_{\hat{K}  \in \mathrm{es}(K)} A_{\hat{K}}^{-1},
\qquad
\chi_\ba(A) := A_{\ba}^{\sharp} (A_{\ba}^{\flat})^{-1},
\end{equation}
with $A_{\ba}^{\sharp}: = \max_{K\in \mesha} A_{K}$ and  $A_{\ba}^{\flat}: = \min_{K\in \mesha} A_{K}$. Moreover, it is convenient to set
\begin{subequations} \label{eq:def_AflatF}\begin{alignat}{2}
A_F^{\flat}&: = \min(A_{K_1},A_{K_2}), &\quad &\forall F = \partial {K_1} \cap \partial {K_2} \in \Fint, \\
A_F^{\flat}&:= A_K, &\quad &\forall F = \dK \cap \partial {\Omega} \in \Fb.
\end{alignat} \end{subequations}
For all $K\in\mesh$, we also define $A^\flat_{\dK}|_F := A^\flat_F$ for all $F\in\FK$.

\subsection{Global upper error bound}

We decompose the error into two components as follows:
\begin{equation}\label{def: error splitting}
e:=u- u_{\mesh}=(u-u_c)+(u_c- u_{\mesh})=:e_c+e_d, \qquad u_c\in H^1_{g,\ddd}(\Omega),
\end{equation}
where $u_{\mesh}$ denotes the cell component of the HHO solution such that $u_{\mesh}|_K : = u_K$ for all $K\in\mesh$.  The function $u_c$ is constructed from $u_{\mesh}$ as detailed in Section \ref{sec: nonconforming bound}. The precise definition of $u_c$ is irrelevant for bounding $e_c$ (we only use that $e_c\in H^1_{0,\ddd}(\Omega)$), and is only relevant for bounding $e_d$. We call $e_c$ the conforming error and $e_d$ the nonconforming error.

\subsubsection{Bound on  conforming error $e_c$}

In this section, we derive two bounds on the conforming error $e_c$. The first bound
avoids the normal flux jump (classically considered in the context of finite elements) and is
$\frac12$-order $p$-suboptimal. The second bound includes the normal flux jump and is $p$-optimal.
The first bound is, however, interesting in its own right. We will also see that, in the context of the lower
error bound, the normal flux jump leads to $\frac12$-order $p$-suboptimality anyway. The proofs of the following two results
are postponed to Section~\ref{sec:proofs}.

\begin{lemma}[Conforming error ($p$-suboptimal bound)]\label{Lemma: Conforming estimator p suboptimal}
The following holds:
\begin{align}
\|A^{\frac12}\nabla e_c\|_{\Omega}  \leq{}& C_{\rm c,1} \bigg\{  \su   \Big\{ A_K^{-1} \Big(\frac{h_K}{k+1}\Big)^2\|\Pi_{K}^{k+1}(f)+ A_K\Delta R_K^{k+1}(\hat{u}_K)\|_K^2
+A_K(k+1)S_{\dK}(\hat{u}_K,\hat{u}_K) \nno \\
&+ O_K(f)^2 + O_K(g_\dn)^2  \Big\}  \bigg\}^{\frac12}  + \|A^{\frac12}\bnabla e_d\|_{\Omega},
\label{eq:nonconforming error bound without normal flux jump}
\end{align}
where the constant $C_{\rm c,1}$ depends on the mesh shape-regularity and on $\max_{K\in\mesh} \chi_K(A)^{\frac12}$, but is independent of $h$ and $k$.
\end{lemma}

\begin{lemma}[Conforming error ($p$-optimal bound)]\label{Lemma: Conforming estimator}
The following holds:
\begin{align}
\|A^{\frac12}\nabla e_c\|_{\Omega}  \leq {}& C_{\rm c,2} \bigg\{  \su \Big\{  A_K^{-1} \Big(\frac{h_K}{k+1}\Big)^2\|\Pi_{K}^{k+1}(f)+ A_K\Delta R_K^{k+1}(\hat{u}_K)\|_K^2 +A_KS_{\dK}(\hat{u}_K,\hat{u}_K) \nno \\
& + A_K^{-1} \Big(\frac{h_K}{k+1}\Big) \Big(   \| \jumpK{A\nabla R_{\mesh}^{k+1}(\hat{u}_h)} {\cdot}\n_K \|_{\dKi}^2
+  \ltwo{A_K\nabla R_K^{k+1}(\hat{u}_K)  {\cdot}\n_{K} - \Pi_{\dK}^k (g_{\dn}|_{\dK})}{\dKN}^2 \Big) \nno \\
&+O_K(f)^2 +  O_K(g_{\dn})^2  \Big\} \bigg\}^{\frac12} + \|A^{\frac12}\bnabla e_d\|_{\Omega},
\label{eq:conforming error bound p-optimal}
\end{align}
where the constant $C_{\rm c,2}$ depends on the mesh shape-regularity and on $\max_{K\in\mesh} \chi_K(A)^{\frac12}$, but is independent of $h$ and $k$.
\end{lemma}

\begin{remark}[Alternative bound]
The bounds on the conforming error derived in Lemma~\ref{Lemma: Conforming estimator p suboptimal} and in Lemma~\ref{Lemma: Conforming estimator} hinge on the $hp$-approximation properties of the modified Karkulik--Melenk interpolation operator. If one considers instead the classical HHO interpolation operator based on local $L^2$-orthogonal projection onto cells and faces, one can derive an upper bound on the conforming error containing only the stabilization term and the data oscillation, but the price to pay is a $p$-suboptimality by half-order scaling the stabilization term. Details are omitted for brevity.
\end{remark}

\subsubsection{Bound on  nonconforming error $e_d$} \label{sec: nonconforming bound}

We start by defining the function $u_c\in H^1_{g,\ddd}(\Omega)$ introduced in \eqref{def: error splitting}.
To this purpose, in the spirit of \cite{ErnVo:20}, we solve local minimization problems in $H^1(\oma)$ with suitable boundary conditions for every mesh
vertex $\ba\in\vertice$ (recall that $\oma$ denotes the star associated with the vertex $\ba$).
Let $\psi_{\ba}$ be the hat basis function equal to $1$ at $\ba$ and  having a support in the vertex star $\oma$. Recall that the hat basis functions satisfy the following partition-of-unity property:
\begin{equation} \label{eq:PU}
\sum_{\ba\in \vertice} \psi_\ba = 1.
\end{equation}

\begin{definition}[Patchwise and global potential reconstruction]\label{def: Patchwise and global potential reconstruction}
For all $\ba\in \vertice$, let $u^{\ba}_c\in H^1_{g,\ddd}(\oma)$ be the solution of the following well-posed problem:
\begin{equation}\label{def: Patch reconstruction}
(A \nabla u^{\ba}_c, \nabla v_{\ba})_{\oma}= (A \bnabla (\psi_{\ba} u_{\mesh}), \nabla v_{\ba})_{\oma},
\qquad \forall v_{\ba} \in H^1_{0,\ddd} (\oma),
\end{equation}
with
\begin{equation}\label{def: Vertx space}
H^1_{g,\ddd}(\oma): = \{v\in H^1(\oma)\:|\: v|_{\partial \oma \cap  \Gamma_{\ddd}} = \psi_{\ba} g_\ddd  \text{ and }   v|_{\partial \oma \cap \Omega} =0 \}.
\end{equation}
An equivalent definition is
\begin{equation}\label{local minimizatiion interiori}
u^{\ba}_c: = \arg \min_{\rho_{\ba}\in H^1_{g,\ddd}(\oma)}\| A^{\frac12} (\nabla \rho_{\ba} - \bnabla(\psi_{\ba} u_{\mesh}))\|_{\oma}.
\end{equation}
Then, extending $u^{\ba}_c$ by zero to $\Omega$,  we set
\begin{equation}\label{def: patch reconstruction on the whole domain}
u_c: = \sum_{\ba\in \vertice} u^{\ba}_c.
\end{equation}
\end{definition}
Notice that we indeed have $u_c \in H^1_{g,\ddd}(\Omega)$ as required in \eqref{def: error splitting}; this follows from the partition-of-unity property~\eqref{eq:PU} and the definition~\eqref{def: Vertx space} of $H^1_{g,\ddd}(\oma)$.
The proof of the following result
is postponed to Section~\ref{sec:proofs}.

\begin{lemma}[Local nonconforming error]\label{lemma: nonconforming error bound on patches}
Let $u^{\ba}_c$ solve \eqref{def: Patch reconstruction} and set $e_{d}^{\ba}:= u^{\ba}_c - \psi_{\ba}  u_{\mesh}$
for all $\ba \in \vertice$. The following holds:
\begin{align}
\|A^{\frac12}\bnabla e_{d}^{\ba}\|_{\oma} \le {}&
C_{\rm d}^{\ba} \bigg\{ \!\sua\!  \Big\{ A_K S_{\dK}(\hat{u}_K,\hat{u}_K) + O_K(g_{\ddd})^2 \Big\}
+ \!\! \sum_{F\in \Fa \cap \Fint  } \!\!
A_F^\flat \Big(\frac{h_F}{k+1}\Big) \ltwo{\jump{ \nabla   u_{\mesh}} {\times} \n_F}{F}^2 \nno \\
& + \!\!\sum_{F\in \Fa \cap \FD }\!\! A_F^\flat \Big(\frac{h_F}{k+1}\Big)
\ltwo{ \nabla ( u_{\mesh}-\Pi_{\dK}^{k+1}(g_\ddd|_{\dK}))  {\times} \n_{\Omega}}{F}^2 \bigg\}^{\frac12},
\label{eq:nonconforming error bound on patch}
\end{align}
where the constant $C_{\rm d}^{\ba}$ depends on the mesh shape-regularity and on $\chi_{\ba}(A)^{\frac12}$, but is independent of $h$ and $k$.
\end{lemma}

\begin{corollary}[Global bound on nonconforming error]\label{corollary: nonconforming error total}
The following holds:
\begin{align}
\|A^{\frac12}\bnabla e_{d}\|_{\Omega}
\leq {}& C_{\rm d} \bigg\{ \!\su\! \Big\{  A_K S_{\dK}(\hat{u}_K,\hat{u}_K)
+ O_K(g_{\ddd})^2
+ A_{\dK}^{\flat} \Big( \frac{  h_K}{k+1} \Big)  \ltwo{\jumpK{ \nabla  u_{\mesh}} {\times} \n_{K}}{\dKi}^2 \nno \\
& +  A_K \Big(\frac{  h_K}{k+1} \Big)
\ltwo{ \nabla( u_{K}-  \Pi_{\dK}^{k+1} (g_{\ddd}|_{\dK})) {\times} \n_{\Omega}}{\dKD}^2
\Big\}  \bigg\}^{\frac12},
\label{eq:nonconforming error bound all}
\end{align}
where the constant $C_{\rm d}$ depends on the mesh shape-regularity and on $\max_{\ba\in\vertice} \chi_{\ba}(A)^{\frac12}$,  but is independent of $h$ and $k$.
\end{corollary}

\begin{proof}
Using that $u_c|_K = \sum_{\ba\in \verticeK} u^{\ba}_c|_K$ for all $K\in \mesh$,
where $\verticeK$ is the set of vertices of $K$,
the partition-of-unity property~\eqref{eq:PU}, and the triangle inequality, we infer that
\begin{align*}
\|A^{\frac12}\bnabla e_{d}\|_{\Omega}^2 &= \su A_K\norm{\nabla (u_c -  u_K)}{K}^2
= \su A_K \bigg\|\sum_{\ba \in \verticeK} \nabla (u_c^{\ba} - \psi_{\ba}u_K)\bigg\|_K^2 \\
&\leq \su \sum_{\ba \in \verticeK} (d+1)A_K\norm{\nabla (u^{\ba}_c -\psi_{\ba}  u_K )}{K}^2
= \sum_{\ba\in\vertice} (d+1)\|A^{\frac12}\bnabla e_{d}^{\ba}\|^2_{\oma}.
\end{align*}
Invoking Lemma~\ref{lemma: nonconforming error bound on patches} and using that
$h_F\leq h_K$ for all $F\in \dK$ and all $K\in \mesh$ completes the proof.
\end{proof}

\begin{remark}[Constant $C_{\rm d}$]
The constant $C_{\rm d}$ in \eqref{eq:nonconforming error bound all}
does not depend on the topology of $\Omega$. The reason is that we do not invoke
a global Helmholtz decomposition in $\Omega$, but instead invoke a local decomposition
in each vertex star $\oma$ (see Section~\ref{sec:proof_nonconf} for further insight).
\end{remark}

\subsubsection{Main result}

For all $K\in\mesh$, we define the following error indicators:
\begin{subequations} \label{each term} \begin{align}
\eta_{K,{\rm res}}&:= A_K^{-\frac12}\Big(\frac{h_K}{k+1}\Big) \ltwo{\Pi_{K}^{k+1}(f)+  A_K\Delta R_K^{k+1}(\hat{u}_K)}{K}, \\
\eta_{K,{\rm sta}}&:= A_K^{\frac12}S_{\dK}(\hat{u}_K,\hat{u}_K)^{\frac12}, \\
\eta_{K,{\rm tan}} &:= (A_{\dK}^{\flat})^{\frac12} \Big(\frac{h_K}{k+1}\Big)^{\frac12}
\Big\{ \ltwo{ \jumpK{\nabla  u_{\mesh}} {\times} \n_K }{\dKi}
+ \ltwo{ \nabla  (u_K- \Pi_{\dK}^{k+1} (g_{\ddd}|_{\dK})){\times} \n_{\Omega} }{\dKD}\Big\}, \\
\eta_{K,{\rm nor}}&:= A_K^{-\frac12}\Big(\frac{h_K}{k+1}\Big)^{\frac12} \Big\{
\| \jumpK{A\nabla R_{\mesh}^{k+1}(\hat{u}_h)} {\cdot}\n_K \|_{\dKi}
+ \ltwo{A_K\nabla R_K^{k+1}(\hat{u}_K) {\cdot}\n_{\Omega}  -\Pi_{\dK}^k ({g}_{\dn}|_{\dK}) }{\dKN}\Big\}, \\
O_{K,{\rm dat}} &:= O_K(f) + O_K(g_{\dn}) + O_K(g_{\ddd}), \label{eq:def_OK}
\end{align}\end{subequations}
where the three data oscillation terms on the right-hand side of~\eqref{eq:def_OK} are
defined in~\eqref{eq:def_oscill}.

\begin{theorem}[$hp$-upper error bound]\label{thm:upperbound}
Let $u$ be the weak solution to~\eqref{eq:prob}, and let $\hat{u}_h$
be the discrete solution to~\eqref{discrete problem}.
The following holds:
\begin{align}
\su \Big\{ \|A^{\frac12} \nabla e\|_{ K}^2
+ A_K S_{\dK}(\hat{u}_K,\hat{u}_K) \Big\} \leq {}& C_{\rm u} \bigg\{\su \Big\{
\eta_{K,{\rm res}}^2 + \eta_{K,{\rm tan}}^2+ \eta_{K,{\rm sta}}^2 + O_{K,{\rm dat}}^2 \Big\} \nno \\
& + \min\bigg(\su k\eta_{K,{\rm sta}}^2, \su \eta_{K,{\rm nor}}^2 \bigg) \bigg\}, \label{full bound}
\end{align}
where $C_{\rm u}$ depends on the mesh shape-regularity and on $\max_{K\in\mesh}\chi_K(A)^{\frac12}$
but is independent of $h$ and $k$.
\end{theorem}

\begin{proof}
Combining the bounds from Lemma~\ref{Lemma: Conforming estimator p suboptimal}
and Lemma~\ref{Lemma: Conforming estimator} gives
\[
\|A^{\frac12}\nabla e_c\|_{\Omega} \leq C \bigg\{ \!\su\! \Big\{
\eta_{K,{\rm res}}^2 + \eta_{K,{\rm sta}}^2 + O_K(f)^2 + O_K(g_{\dn})^2 \Big\}
+ \min\bigg(\su k\eta_{K,{\rm sta}}^2, \su \eta_{K,{\rm nor}}^2 \bigg) \bigg\}^{\frac12},
\]
where the constant $C$ depends on $\max_{K\in\mesh} \chi_K(A)^{\frac12}$.
Moreover, we can rewrite the bound of Corollary~\ref{corollary: nonconforming error total}
as follows:
\[
\|A^{\frac12}\bnabla e_{d}\|_{\Omega}
\leq C \bigg\{ \!\su\! \Big\{ \eta_{K,{\rm sta}}^2 + \eta_{K,{\rm tan}}^2
+ O_K(g_{\ddd})^2 \Big\}  \bigg\}^{\frac12},
\]
where the constant $C$ depends on $\max_{\ba\in\vertice} \chi_\ba(A)^{\frac12}$.
Combining the above two bounds and since $\max_{\ba\in\vertice}\chi_\ba(A)
\le \max_{K\in\mesh}\chi_K(A)$ proves that
$\su \|A^{\frac12} \nabla e \|_{ K}^2$ is bounded by the right-hand side
of~\eqref{full bound}. Since $\su A_K S_{\dK}(\hat{u}_K,\hat{u}_K) =
\su \eta_{K,{\rm sta}}^2$, the proof is complete.
\end{proof}

\begin{remark}[Estimator without normal flux jump]
Notice that \eqref{full bound} implies that
\[
\su \Big\{ \|A^{\frac12}\nabla e \|_{ K}^2
+ A_K S_{\dK}(\hat{u}_K,\hat{u}_K) \Big\} \leq C_{\rm u} \su \Big\{
\eta_{K,{\rm res}}^2 + \eta_{K,{\rm tan}}^2+ (k+1)\eta_{K,{\rm sta}}^2 + O_{K,{\rm dat}}^2 \Big\}.
\]
This upper bound does not contain the normal flux jump, which is often the dominant component
of the error estimator for $H^1$-conforming FEM \cite{CarstensenVerfurth99}. The price to pay
is a $\frac12$-order $p$-suboptimality for the stabilization term. Our numerical experiments
confirm that the normal flux jump term does not dominate the total error estimator.
\end{remark}

\subsection{Local lower error bound} \label{sec: lower bound}

In this section, we establish a local lower error bound. Specifically,
we bound the local error indicators $\eta_{K,{\rm res}}$, $\eta_{K,{\rm nor}}$,
and $\eta_{K,{\rm tan}}$, for all $K\in\mesh$, in terms of the error
$e=u-u_{\mesh}$ in (a neighborhood of) $K$ and the data oscillation indicators
defined in~\eqref{eq:def_oscill}.
We do not bound the local error indicator $\eta_{K,{\rm sta}}$ since it is present on
both sides of the upper error bound~\eqref{full bound}.
This is classical in a posteriori error estimates for nonconforming methods. The proof of the following result is postponed to Section~\ref{sec:proofs}.

\begin{theorem}[$hp$-local lower error bound]\label{thm:lowerbound}
The following holds for all $K\in\mesh$:
\begin{subequations} \begin{align}
\eta_{K,{\rm res}} \leq {}& C_{\rm l} (k+1) \big( \|A^{\frac12} \nabla e\|_K + A_K^{\frac12} S_{\dK}(\hat{u}_K,\hat{u}_K)^{\frac12} + O_K(f)\big), \label{eq:lower bound elem}\\
\eta_{K,{\rm nor}} \leq {}& C_{\rm l} (k+1)^{\frac12} \bigg\{ \sum_{\mathcal{K} \in \omega_K}
A_{\mathcal{K}}S_{\partial \mathcal{K}}(\hat{u}_\mathcal{K},\hat{u}_\mathcal{K})
\bigg\}^{\frac12}, \label{eq:lower jump residual}\\
\eta_{K,{\rm tan}} \leq {}& C_{\rm l} (k+1)^{\frac32} \bigg\{ \sum_{\mathcal{K} \in \omega_K}
\ltwo{A^{\frac12}\nabla e}{\mathcal{K}}^2 \bigg\}^{\frac12},\label{eq: lower bound tangential jump}
\end{align} \end{subequations}
where $\omega_K$ collects all the mesh cells (including $K$) sharing at least an interface with $K$, and the constant $C_{\rm l}$ depends on the mesh shape-regularity and on the diffusion contrast factor $\chi'_K(A) := A_K^{-1} \max_{K'\in \omega_K} A_{K'}$, but is independent of $h$ and $k$.
\end{theorem}

\begin{remark}[Bound on normal flux jump] \label{discuss of p suboptimality}
We observe that the upper bound on $\eta_{K,{\rm nor}}$ has only
a $\frac12$-order $p$-suboptimality. This rather sharp result is achieved by exploiting the
local conservation property of the HHO method, and is in contrast with the upper bound
that could be obtained using bubble function techniques and which would feature
a $\frac32$-order $p$-suboptimality (details not shown for brevity).
\end{remark}

\section{Numerical examples} \label{sec:Numerical example}

In this section, we present numerical examples to illustrate our theoretical results. Our goal here is to illustrate the main findings of the above theoretical analysis and also to illustrate how the present a posteriori error estimators can be used to drive an $h$-adaptive procedure that behaves in a robust way with respect to the polynomial degree. A further step forward, which we leave to future work, is to devise and test a full $hp$-adaptive procedure. This entails substantial algorithmic developments, in particular to set up a criterion to select between $h$- and $p$-refinement locally. We notice that $hp$-adaptive procedures are already available in various contexts in the literature, as reflected, e.g.,
in \cite{MeWo01,hsw,daersmvo18}.

\subsection{Example 1: Convergence rates for smooth solution}

We select $f$, $\diff := I_{2\times 2}$ and Dirichlet boundary conditions on the unit square $\Omega := (-1,1)^2$, so that the exact solution is
\begin{equation}
u(x,y) := \sin(\pi x) \sin(\pi y).
\end{equation}
We employ the polynomial degrees $k\in\{0,\dots,3\}$ and a sequence of successively refined triangular meshes consisting of $\{32$, $128$, $512$, $2048$, $8192\}$ cells.

Let us first verify the convergence rates obtained with the HHO methods with $k\in\{0,\dots,3\}$. We measure the error in the energy norm defined as $\big\{\su A_K \|\nabla (u-u_K)\|_K^2 +S_{\dK}(\hat{u}_K,\hat{u}_K) \big\}^{\frac12}$. The energy error and the a posteriori estimator on the right-hand side of~\eqref{full bound} (with constant $C_{\rm u}$ set to one) are reported in Figure~\ref{Ex1:error_uniform_refinement}. The rates are computed as a function of DoFs, which denotes the total number of globally coupled discrete unknowns (that is, the total number of face unknowns except those located on the boundary faces). We observe that the energy error and the a posteriori estimator both converge at the optimal rate $\mathcal{O}(\textup{DoFs}^{-\frac{(k+1)}{2}})$. The convergence rate of the energy error is optimal in view of the result of Theorem~\ref{Theorem: a priori}. Moreover, the effectivity index, defined as the ratio of the a posteriori estimator to the energy error, remains well behaved as a function of DoFs. The effectivity index takes values between 2 and 2.8 for $k\geq1$, whereas the effectivity index is almost 3 for $k=0$.

\begin{figure}[!tb]
\begin{center}
\begin{tabular}{cc}
\includegraphics[width=0.49\linewidth]{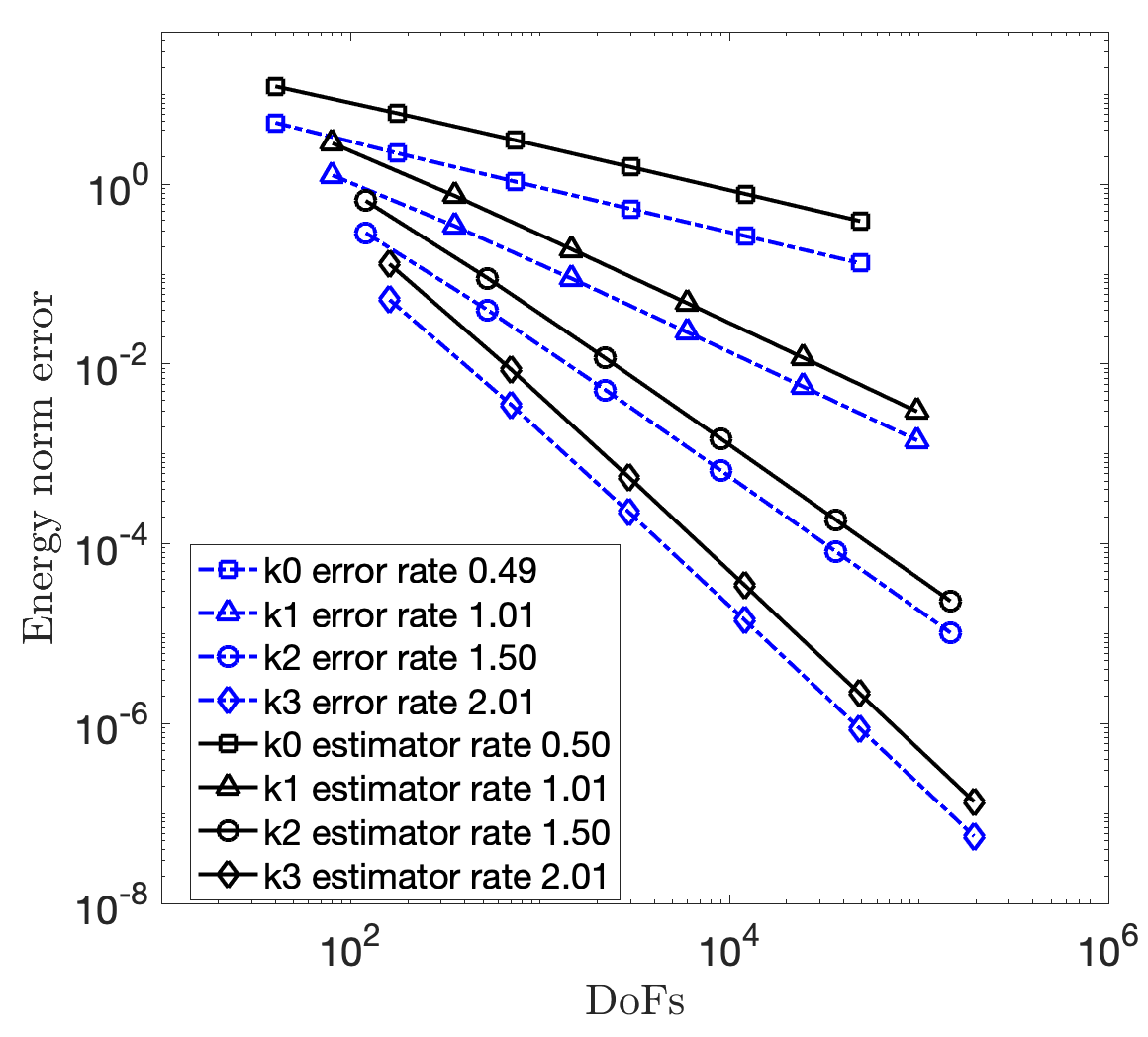} &
\includegraphics[width=0.45\linewidth]{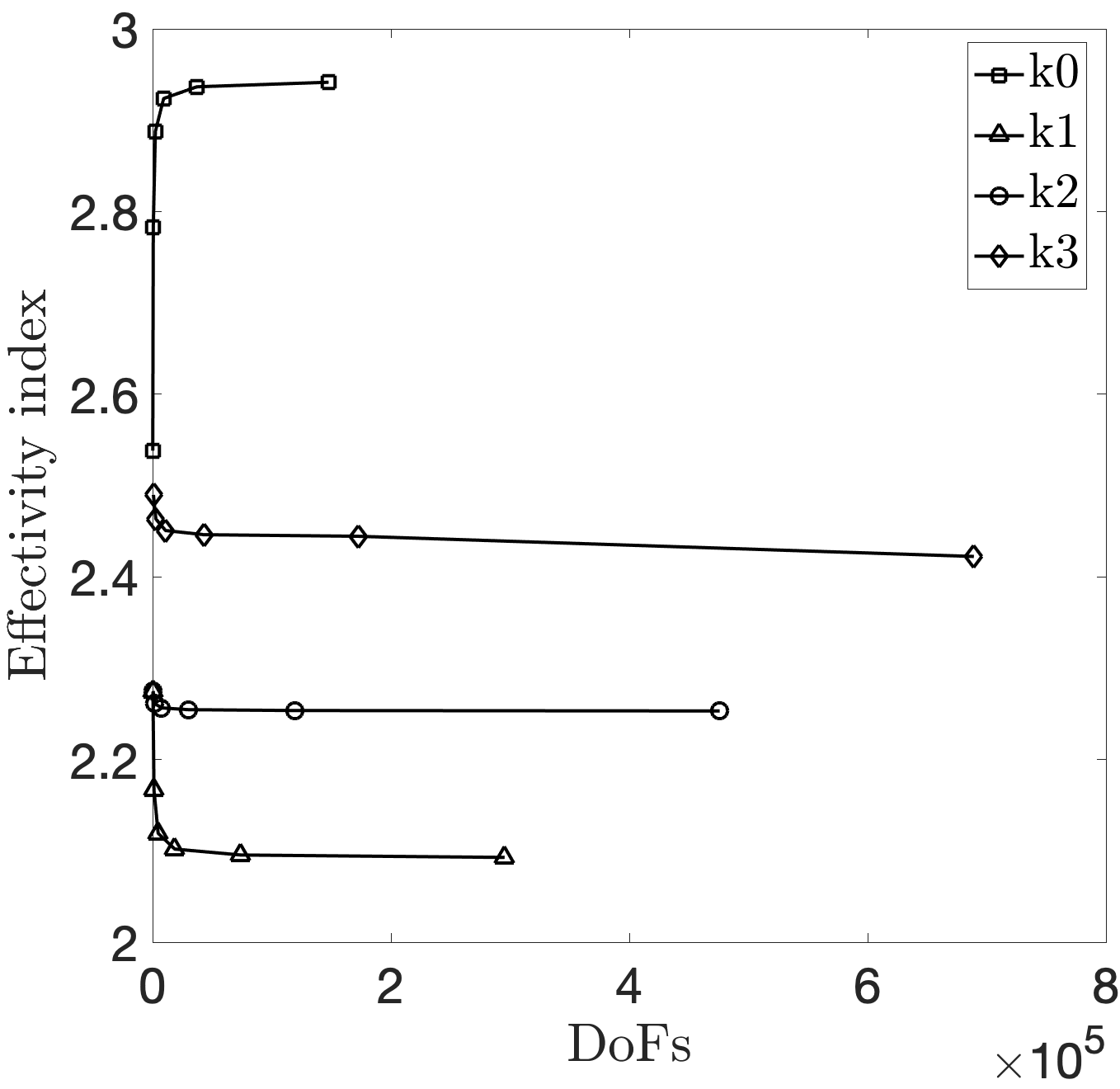}
\end{tabular}
\end{center}
\caption{Example 1. Energy error and a posteriori error estimator for $k\in\{0,1,2,3\}$ as a function of DoFs (left) and effectivity index as a function of DoFs (right).}\label{Ex1:error_uniform_refinement}
\end{figure}

As the results in Theorem~\ref{thm:upperbound} (upper error bound) and Theorem~\ref{thm:lowerbound} (lower error bound) differ by an algebraic rate in the polynomial degree  $k$, we investigate the dependence of the effectivity index on $k$.  In Figure \ref{Ex1:p-refienement}, we report the effectivity index as a function of the polynomial degree $k\in\{1,\dots,9\}$ on a mesh consisting of $128$ triangular cells. We observe an algebraic rate of $p^{\frac12}$, which matches the statement in Theorem~\ref{thm:upperbound}.

\begin{figure}[!tb]
\begin{center}
\includegraphics[width=0.45\linewidth]{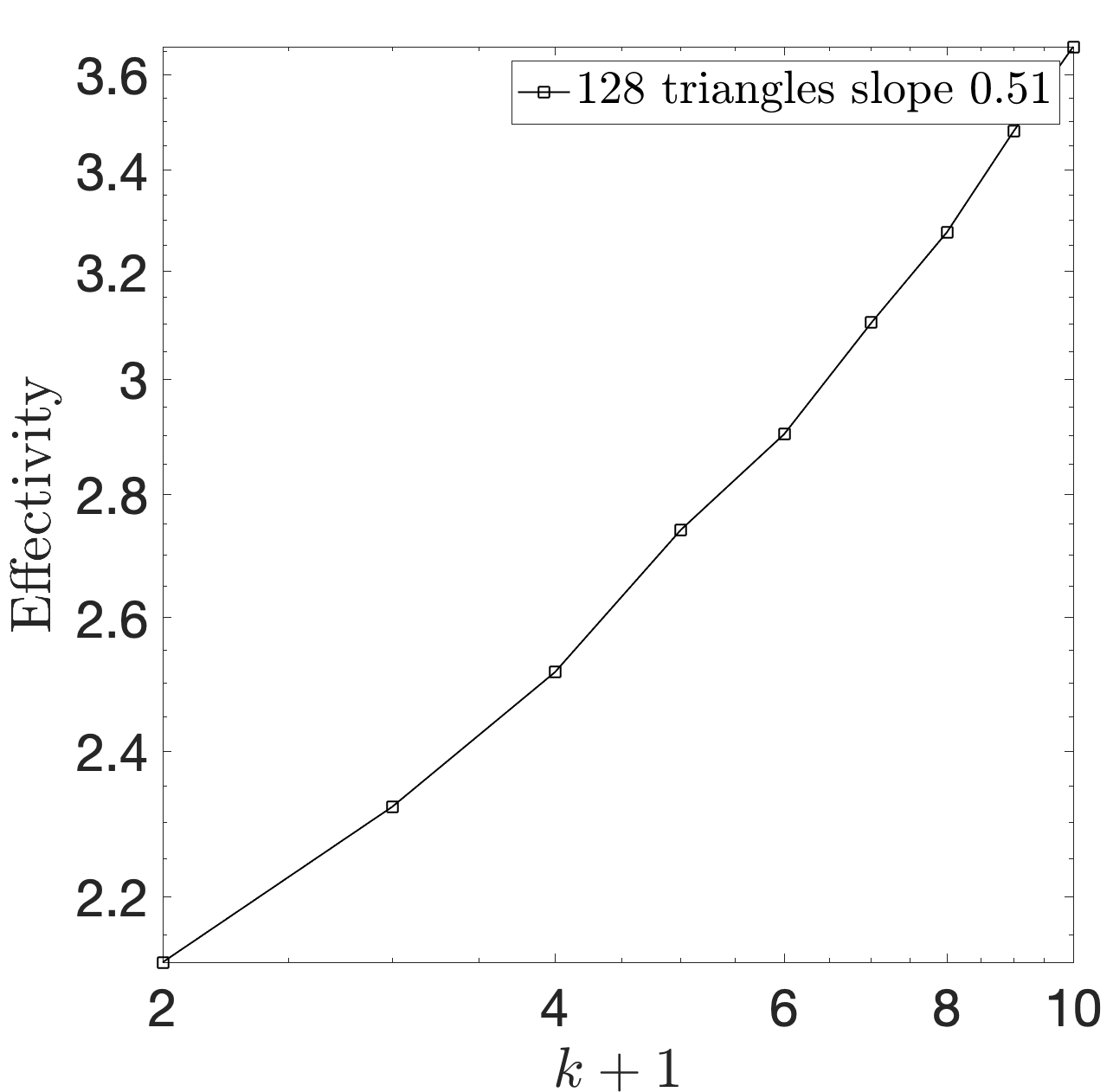}
\end{center}
\caption{Example 1. Effectivity index  with $k\in\{1, \dots,9\}$ on a mesh composed of 128 cells.}\label{Ex1:p-refienement}
\end{figure}

Finally, we compare the relative contributions of the various terms composing the a posteriori error estimator (all in percentage). Setting $\eta_{\rm X}:=\big\{\sum_{K\in\mathcal{T}}\eta_{K,{\rm X}}^2\big\}^{\frac12}$ for ${\rm X}\in\{\rm res,sta,nor,tan\}$, we report in Table~\ref{ex1:table Comparison} the relative contribution of the residual estimator $\eta_{\rm res}$, the stabilization estimator $\eta_{\rm sta}$, the normal flux jump estimator $\eta_{\rm nor}$, and the tangential flux jump estimator $\eta_{\rm tan}$, for polynomial degrees $k\in\{0,1,2\}$. For $k\in\{1,2\}$, the residual estimator dominates the total estimator (by more than 60\%), followed by $\eta_{\rm nor}$ (about $20\%$), while $\eta_{\rm tan}$ and $\eta_{\rm sta}$ are both below $10\%$. For $k=0$, the residual dominates the total estimator (by more than $50\%$), followed by $\eta_{\rm tan}$ ($25\%$), $\eta_{\rm nor}$ ($15\%$), and $\eta_{\rm sta}$ ($6\%$).

\begin{table}[!htb]
\centering
\begin{tabular}{|c|c|c|c|c||c|c|c|c||c|c|c|c|}
\hline
&\multicolumn{4}{|c||}{$k=0$}
&
\multicolumn{4}{|c||}{$k=1$}
&
\multicolumn{4}{|c|}{$k=2$}  \\
\hline
\hline
$\#$ cell  & $\eta_{\rm res}$ & $\eta_{\rm sta}$ &   $\eta_{\rm nor}$& $\eta_{\rm tan}$
& $\eta_{\rm res}$ & $\eta_{\rm sta}$ &   $\eta_{\rm nor}$&$\eta_{\rm tan}$
& $\eta_{\rm res}$ & $\eta_{\rm sta}$ &   $\eta_{\rm nor}$&$\eta_{\rm tan}$   \\
\hline
$128$ & $54$  & $6$ &   $15$ &  $25$
      & $64$  & $9$ &   $19$ &  $8$
      & $66$  & $8$ &   $19$ &  $7$  \\
\hline
$512$ & $54$  & $6$ & $15$  & $25$
      & $62$  & $9$ & $22$  & $7$
      & $65$  & $8$ & $20$  & $7$ \\
\hline
$2048$ & $53$  & $6$ & $15$  & $26$
       & $62$  & $9$ & $24$  & $5$
       & $65$  & $8$ & $20$  & $7$   \\
\hline
$8192$ & $53$ & $6$ &  $15$ & $26$
       &$61$ & $9$ &  $25$ & $5$
       &$65$ & $8$ &  $20$ & $7$  \\
\hline
\end{tabular}
\caption{Example 1. Relative contribution (in $\%$) of the various terms  composing the a posteriori error estimator  for $k\in \{0,1,2\}$.}
\label{ex1:table Comparison}
\end{table}

\subsection{Example 2: Adaptive algorithm for singular solution.}

\begin{figure}[!tb]
\begin{center}
\begin{tabular}{cc}
\includegraphics[width=0.46\linewidth]{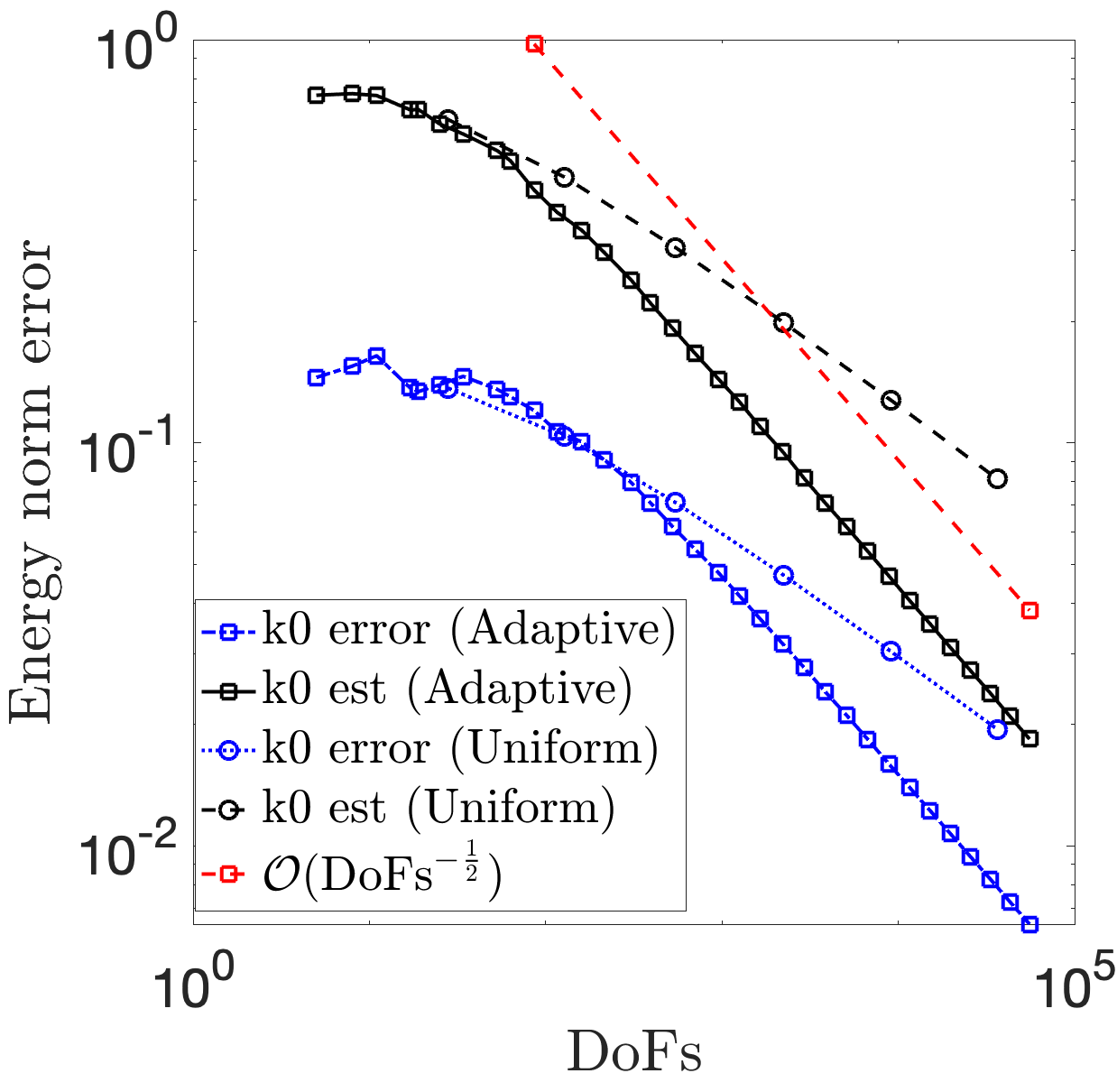} &
\includegraphics[width=0.46\linewidth]{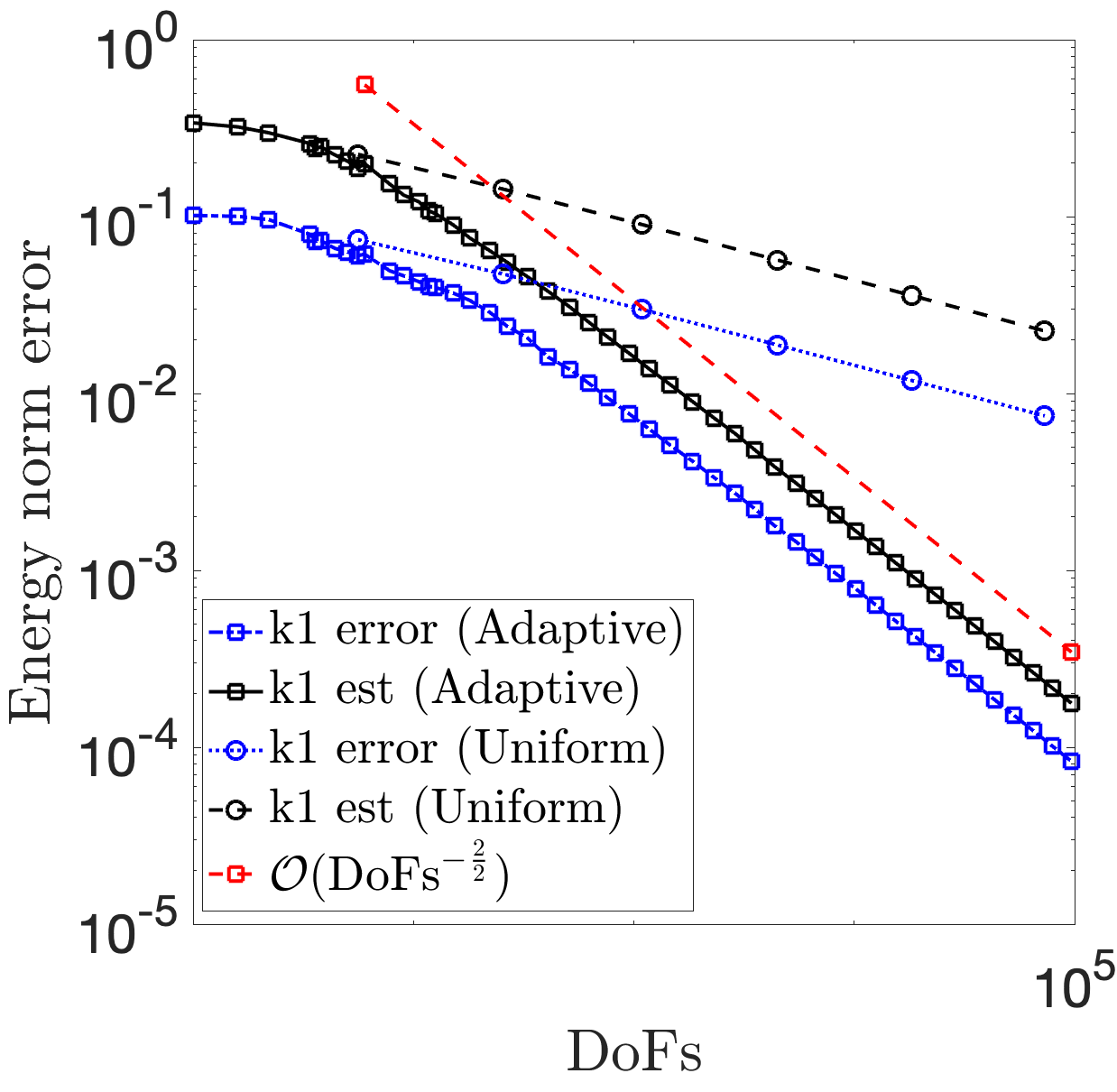} \\
\includegraphics[width=0.46\linewidth]{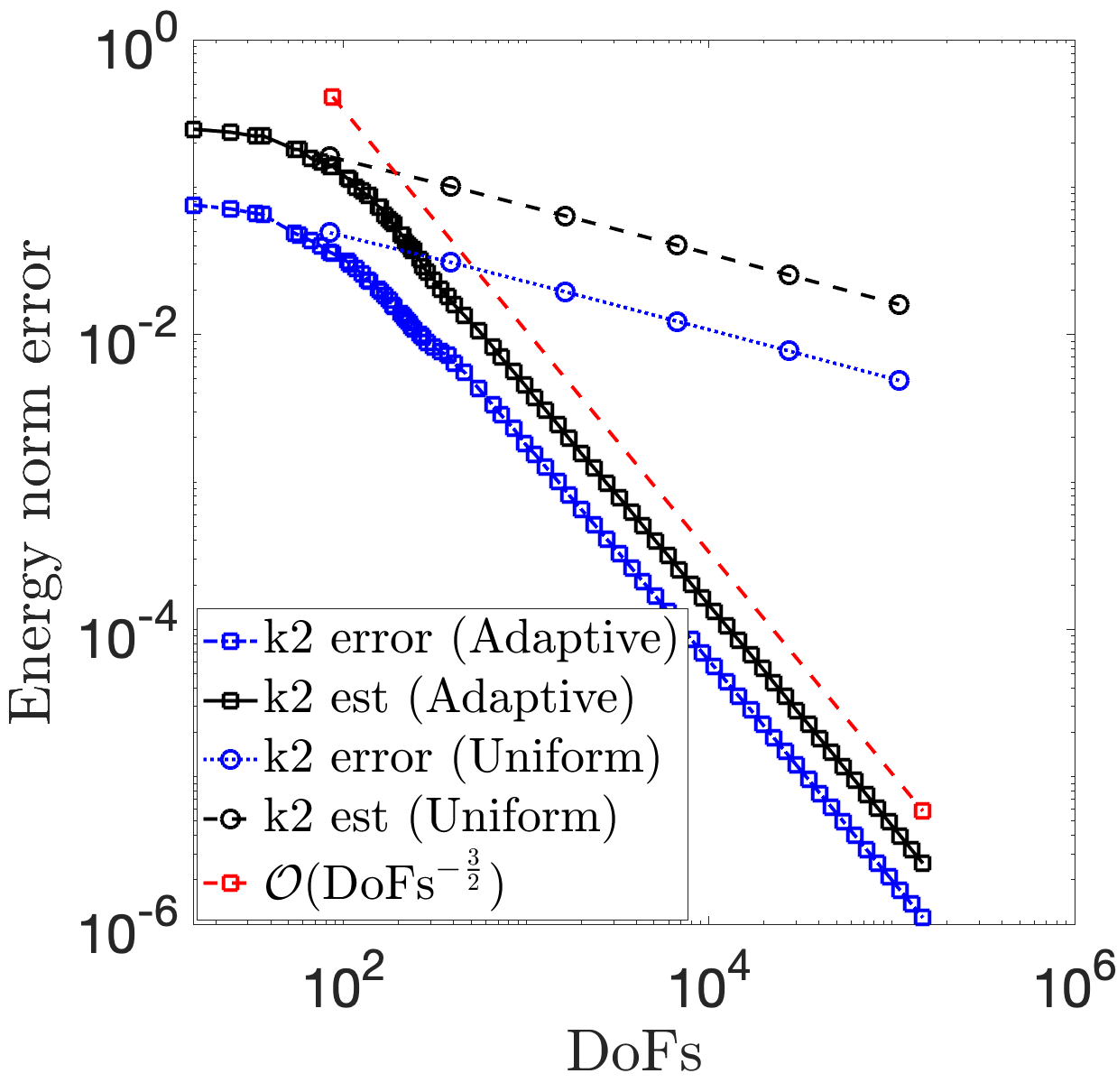} &
\includegraphics[width=0.46\linewidth]{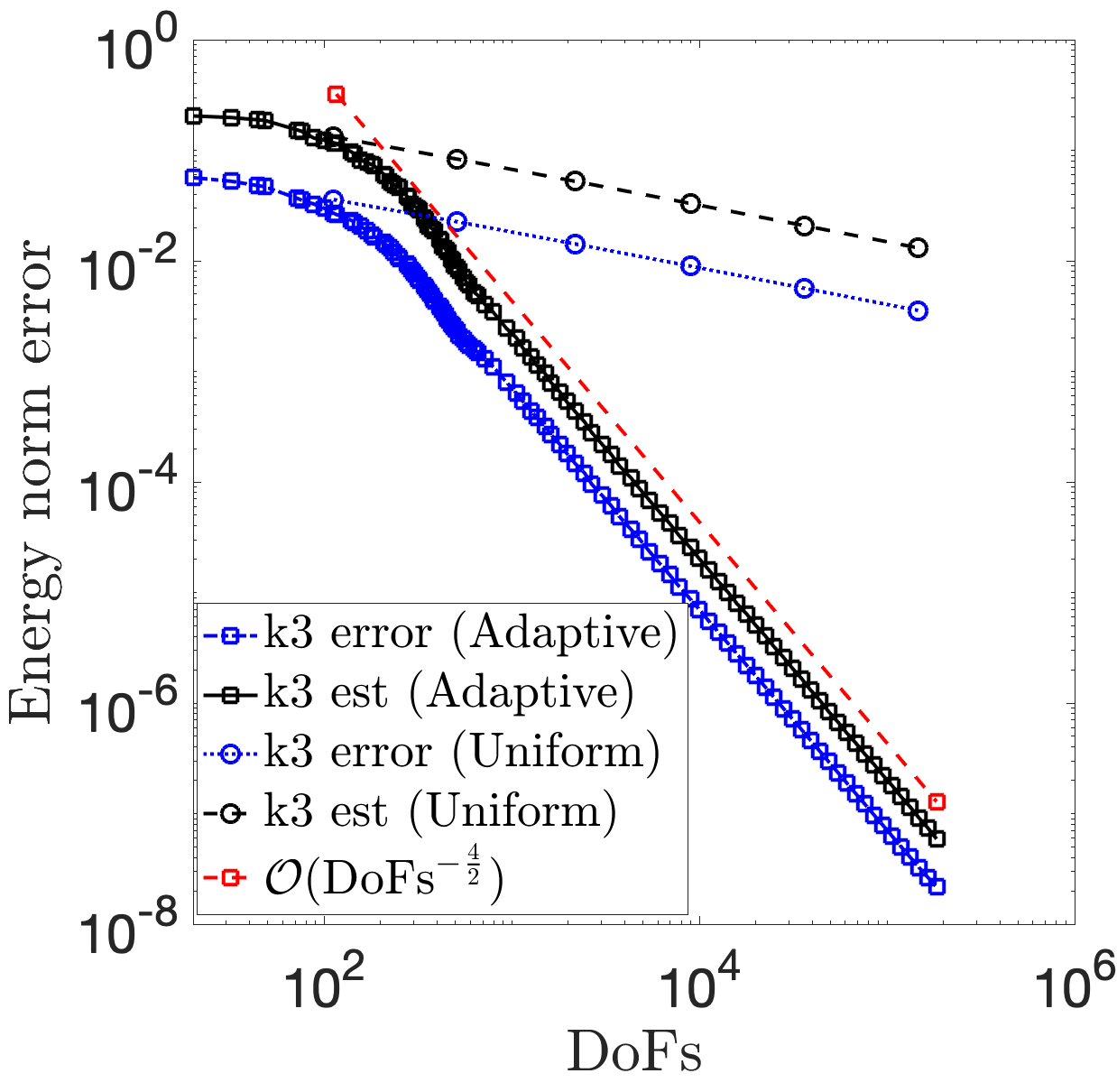}
\end{tabular}
\end{center}
\caption{Example 2. Energy error and a posteriori error estimator as a function of DoFs for $k\in\{0,1,2,3\}$.}\label{Ex2:error_estimator_k=0,1,2,3}
\end{figure}

\begin{figure}[!tb]
\begin{center}
\includegraphics[width=0.45\linewidth]{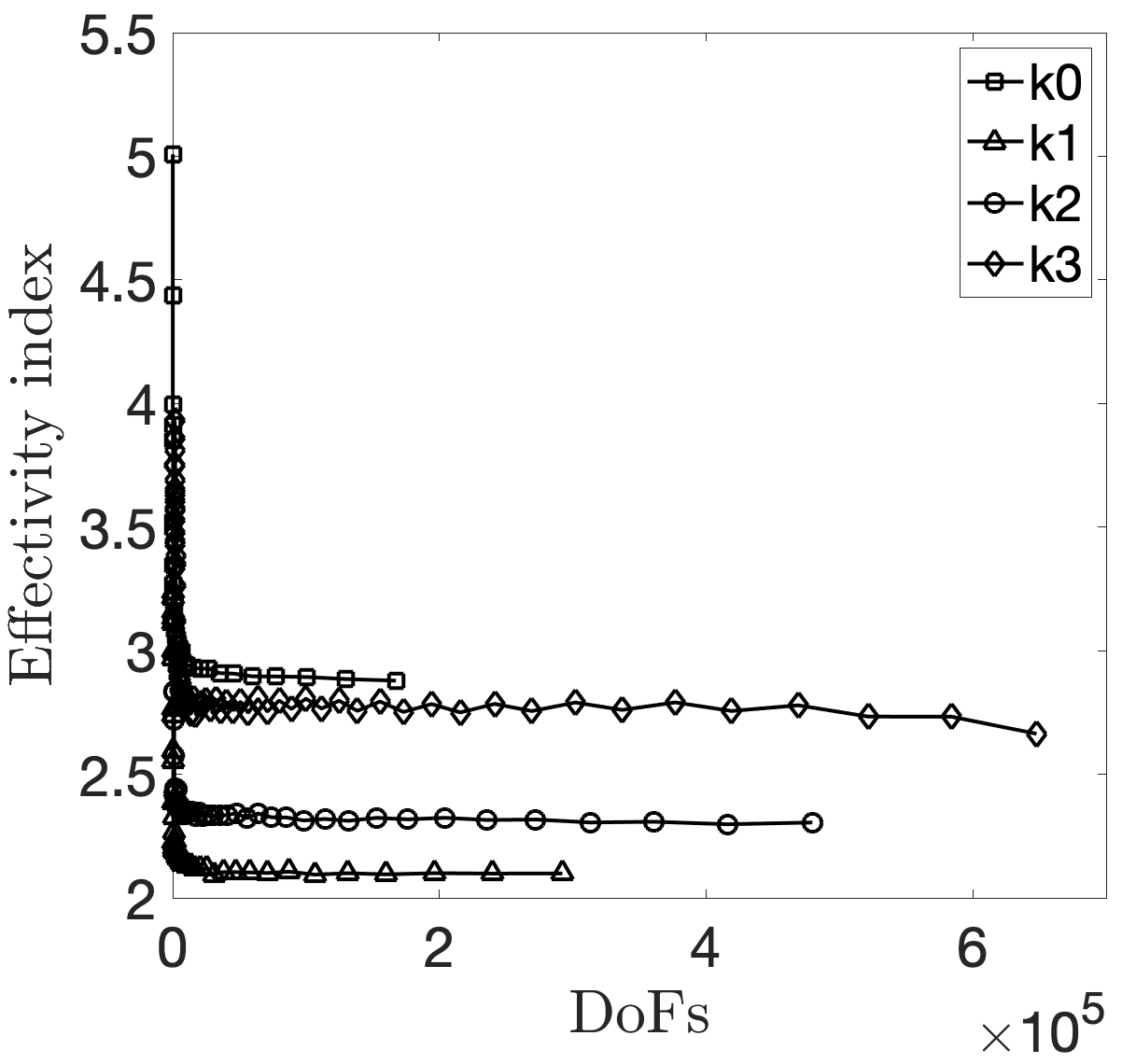}
\qquad
\includegraphics[width=0.45\linewidth]{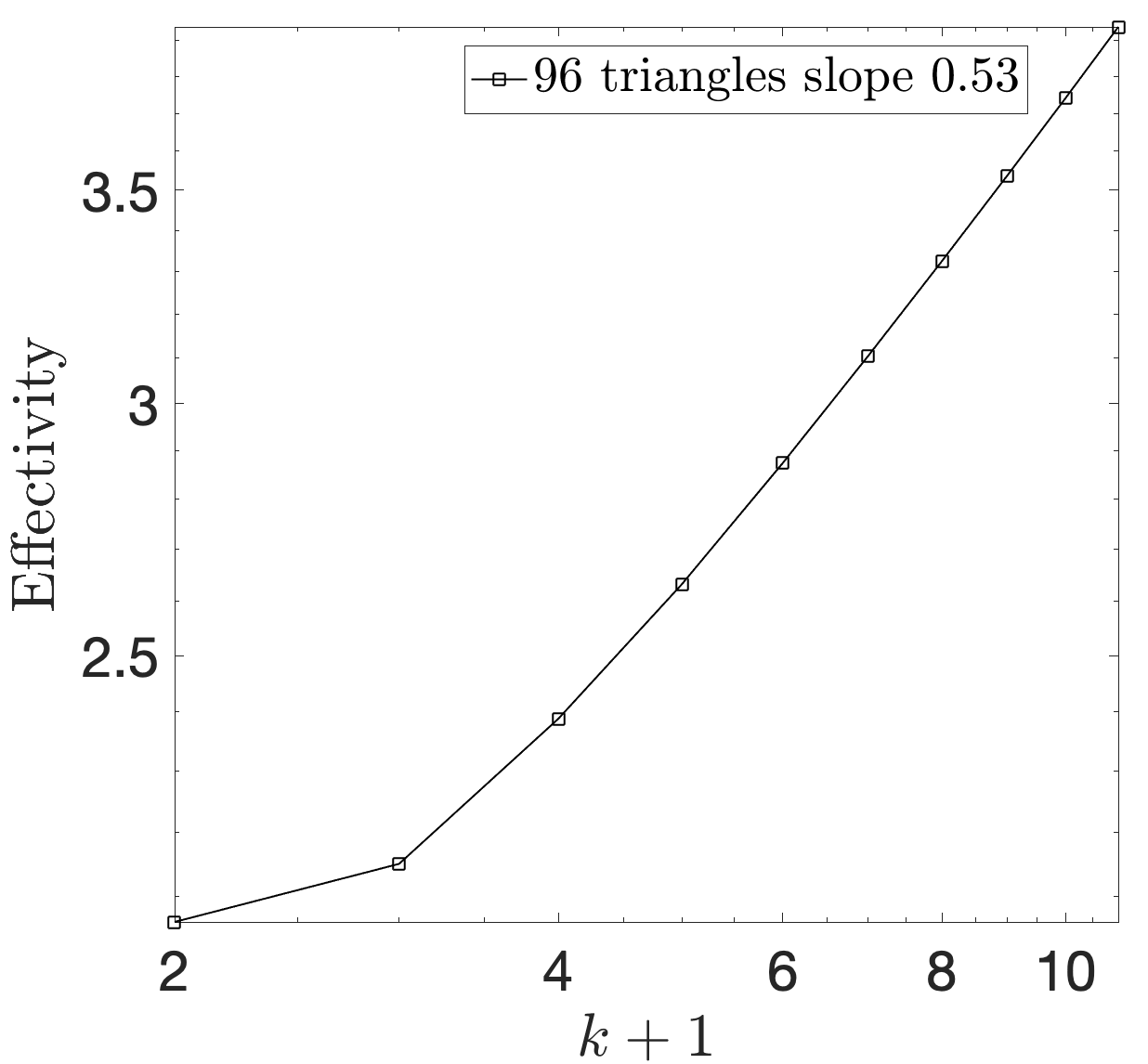}
\end{center}
\caption{Example 2. Effectivity index as a function of DoFs for $k\in\{0, 1,2,3\}$ (left). Effectivity index as a function of $k\in\{1, \dots,10\}$ on a mesh composed of 96 cells (right).}\label{Ex2:effectivity k=0, 1,2,3}
\end{figure}

We select $f$, $\diff := I_{2\times 2}$ and Dirichlet boundary conditions on the L-shaped domain $\Omega := (-1,1)^2 \setminus \{(0,1)\times(-1,0)\}$, so that the exact solution is in polar coordinates
\begin{equation}
u = r^{\frac23} \sin(2\theta/3).
\end{equation}
We test an $h$-adaptive algorithm driven by the a posteriori error estimator from Section~\ref{sec: a posteriori error analysis}. The adaptive algorithm starts from a coarse mesh and uses the estimator on the right-hand side of~\eqref{full bound} to mark mesh cells for refinement through a bulk-chasing criterion (also known as D\"orfler's marking). The adaptive algorithm can be classically described as
\[
\text{SOLVE} \longrightarrow \text{ESTIMATE} \longrightarrow  \text{MARK} \longrightarrow \text{REFINE}.
\]

We first test the convergence rate of the above adaptive algorithm with $k\in\{0,1,2,3\}$ and a bulk-chasing criterion set to $40\%$. The energy error and the a posteriori error estimator are reported in Figure~\ref{Ex2:error_estimator_k=0,1,2,3} with convergence rates computed as a function of DoFs. We observe that the energy error and the a posteriori error estimator converge at the optimal rate $\mathcal{O}(\textup{DoFs}^{-\frac{(k+1)}{2}})$. In contrast, both the energy error and the a posteriori error estimator converge at the suboptimal rate $\mathcal{O}(\textup{DoFs}^{-\frac{1}{3}})$ under uniform refinement, independently of $k$.
Moreover, we observe in the left panel of Figure~\ref{Ex2:effectivity k=0, 1,2,3} that the effectivity index remains well behaved as a function of DoFs and that it slightly increases with the polynomial degree $k\geq1$, with values between 2 and 2.8. For $k=0$, the effectivity index is almost 3.
To gain further insight, we report in the right panel of Figure~\ref{Ex2:effectivity k=0, 1,2,3} the effectivity index as a function of the polynomial degree $k\in\{1,\dots,10\}$ on a mesh consisting of $96$ triangular cells. We observe an algebraic rate of $p^{\frac12}$, again in agreement with Theorem~\ref{thm:upperbound}.

Finally, we report in Table \ref{ex2:table Comparison} the relative contributions of the various terms composing the a posteriori error estimator (all in percentage) with polynomial degrees $k\in\{0,1,2\}$. For $k=1$, the tangential jump estimator dominates the total estimator (by more than $58\%$), followed by $\eta_{\rm nor}$ ($25\%$),  $\eta_{\rm sta}$ ($10\%$) and $\eta_{\rm res}$ ($0\%$). For $k=2$, $\eta_{\rm tan}$, $\eta_{\rm nor}$ and $\eta_{\rm res}$ are fairly well equilibrated (all at about $30\%$), whereas $\eta_{\rm sta}$ is below $10\%$.  For $k=0$, only the tangential jump estimator is nonzero. Clearly, $\eta_{\rm res}=0$ since $f=0$ and $u_K$ is affine for $k=0$. Moreover, one can show that for $f=0$ and $k=0$, the Crouzeix--Raviart FEM solution is the cellwise component of the HHO solution, while the facewise component is the mean-value of the trace of the cell components. This explains why $\eta_{\rm sta}=0$ in Table \ref{ex2:table Comparison}, and consequently $\eta_{\rm nor}=0$ owing to the local conservation property \eqref{conservation property}.

\begin{table}[!htb]
\centering
\begin{tabular}{|c|c|c|c|c|c|c|c|c|c|c|c|c|c|c|}
\hline
\multicolumn{5}{|c|}{$k=0$}
&
\multicolumn{5}{|c|}{$k=1$}
&
\multicolumn{5}{|c|}{$k=2$} \\
\hline
$\#$ cell  & $\eta_{\rm res}$ & $\eta_{\rm sta}$ &   $\eta_{\rm nor}$& $\eta_{\rm tan}$
&$\#$ cell  & $\eta_{\rm res}$ & $\eta_{\rm sta}$ &   $\eta_{\rm nor}$& $\eta_{\rm tan}$
& $\#$ cell & $\eta_{\rm res}$ & $\eta_{\rm sta}$ &   $\eta_{\rm nor}$&$\eta_{\rm tan}$  \\
\hline
$116$ & $0$ & $0$ &  $0$ & $100$
&$118$ & $0$ & $10$ &  $32$ & $58$
& $172$ & $28$  & $9$ &   $29$ &  $34$\\
\hline
$1118$ & $0$ & $0$ &  $0$ & $100$
&$1207$ & $0$ & $10$ &  $26$ & $64$
& $1348$ & $28$  & $8$ & $30$  & $34$ \\
\hline
$5948$ & $0$ & $0$ &  $0$ & $100$
&$6098$ & $0$ & $10$ &  $26$ & $64$
& $5856$ & $27$  & $8$ & $32$  & $33$ \\
\hline
$22306$ & $0$ & $0$ &  $0$ & $100$
&$21762$ &  $0$ & $10$ &  $25$ & $65$
& $21574$ &$26$ & $8$ &  $33$  & $33$ \\
\hline
\end{tabular}
\caption{Example 2. Relative contribution (in $\%$) of the various terms composing the a posteriori error estimator for $k\in \{0,1,2\}$.}
\label{ex2:table Comparison}
\end{table}

\subsection{Example 3: Adaptive algorithm for Kellogg's test case}

Our last example is Kellogg's test case \cite{Kellogg1974poisson}, i.e.,
a diffusion problem on the square domain $\Omega := (-1,1)^2$
with a checkerboard pattern for the diffusion coefficient $A$, namely
$A:=b$ for $xy\geq0$ and $A:=1$ otherwise.
The exact solution (for zero right-hand side and suitable Dirichlet boundary conditions)
takes the form $u:= r^{\alpha}\phi(\theta)$ in polar coordinates, where the explicit expression for the function $\phi$ can be found in \cite[Section~5.3]{MR1980447}.
We select 
$\alpha = 0.1$ and $b = 161.4476387975881$, so that $u\in H^{1.1-\varepsilon}(\Omega)$ with $\varepsilon>0$ arbitrarily small. 


We first test the convergence rate of the $h$-adaptive algorithm described in the previous section with $k\in\{0,1,2,3\}$ and a bulk-chasing criterion set to $10\%$. The energy error and the a posteriori error estimator are reported in Figure~\ref{Ex3:error_estimator_k=0,1,2,3} with convergence rates reported as a function of DoFs. We observe that the energy error and the a posteriori error estimator converge at the optimal rate $\mathcal{O}(\textup{DoFs}^{-\frac{(k+1)}{2}})$. In contrast to Figure~\ref{Ex2:error_estimator_k=0,1,2,3}, we do not include here the convergence plots under uniform mesh refinement because the poor regularity of the exact solution leads to an extremely slow decay; for instance, the energy error is still above $0.1$ for $10^5$ DoFs, independently of $k$. Moreover, we observe in Figure~\ref{Ex3:effectivity k=0, 1,2,3} that the effectivity index remains well behaved as a function of DoFs and that it slightly increases with the polynomial degree $k\geq1$, with values between 1.9 and 2.6. For $k=0$, the effectivity index is almost 2.8. 

\begin{figure}[!tb]
\begin{center}
\begin{tabular}{cc}
\includegraphics[width=0.46\linewidth]{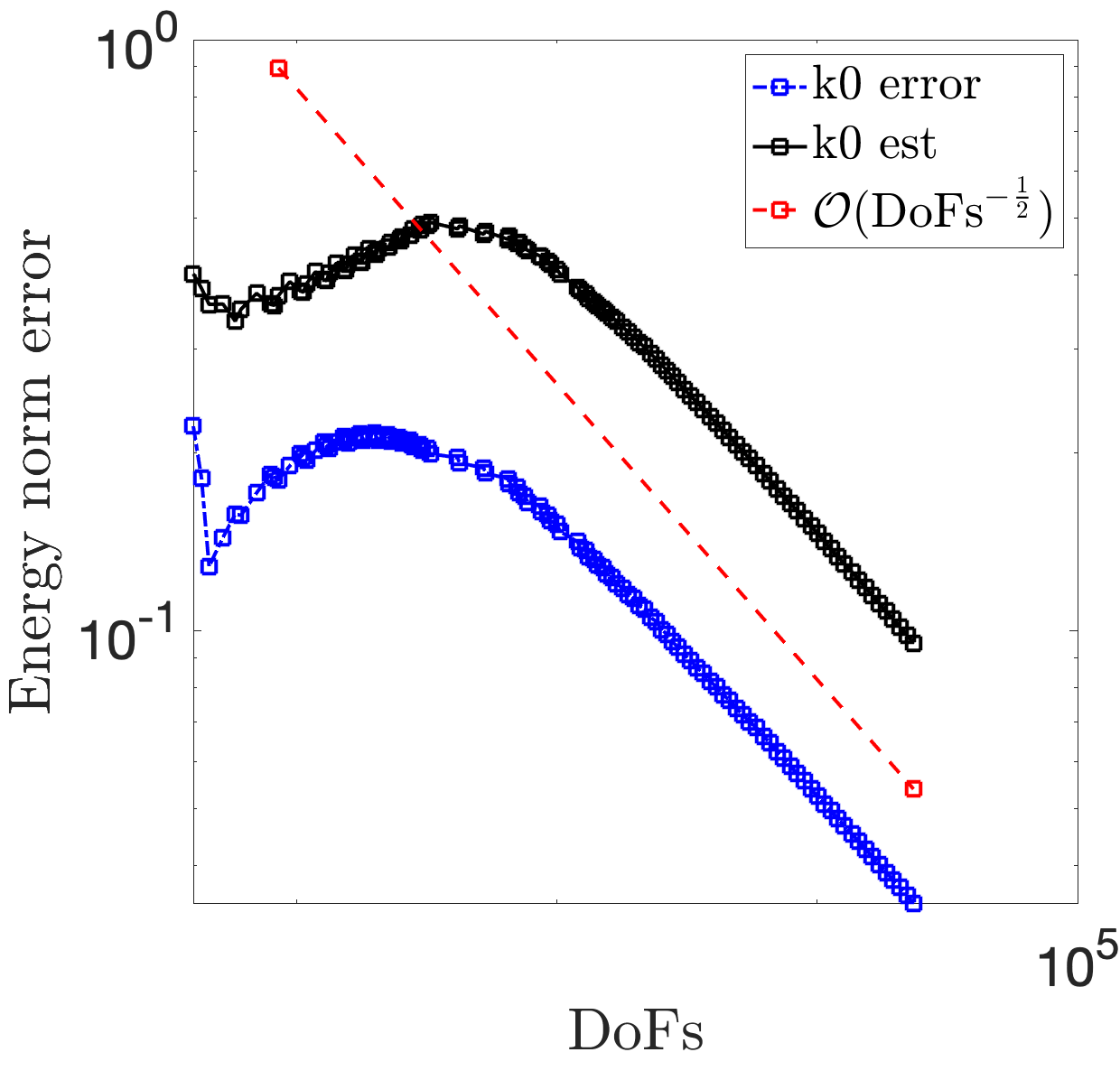} &
\includegraphics[width=0.46\linewidth]{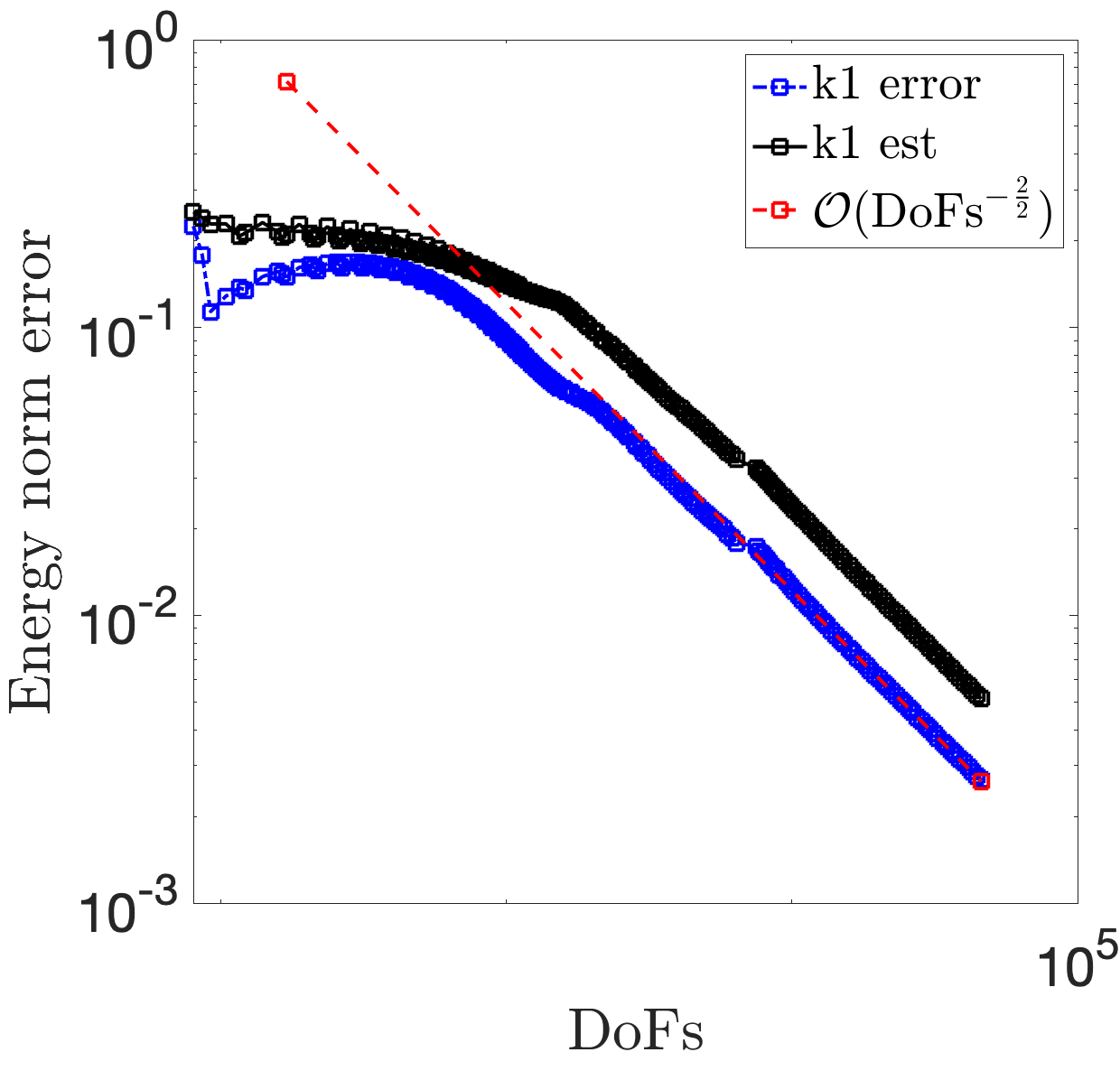}\\
\includegraphics[width=0.46\linewidth]{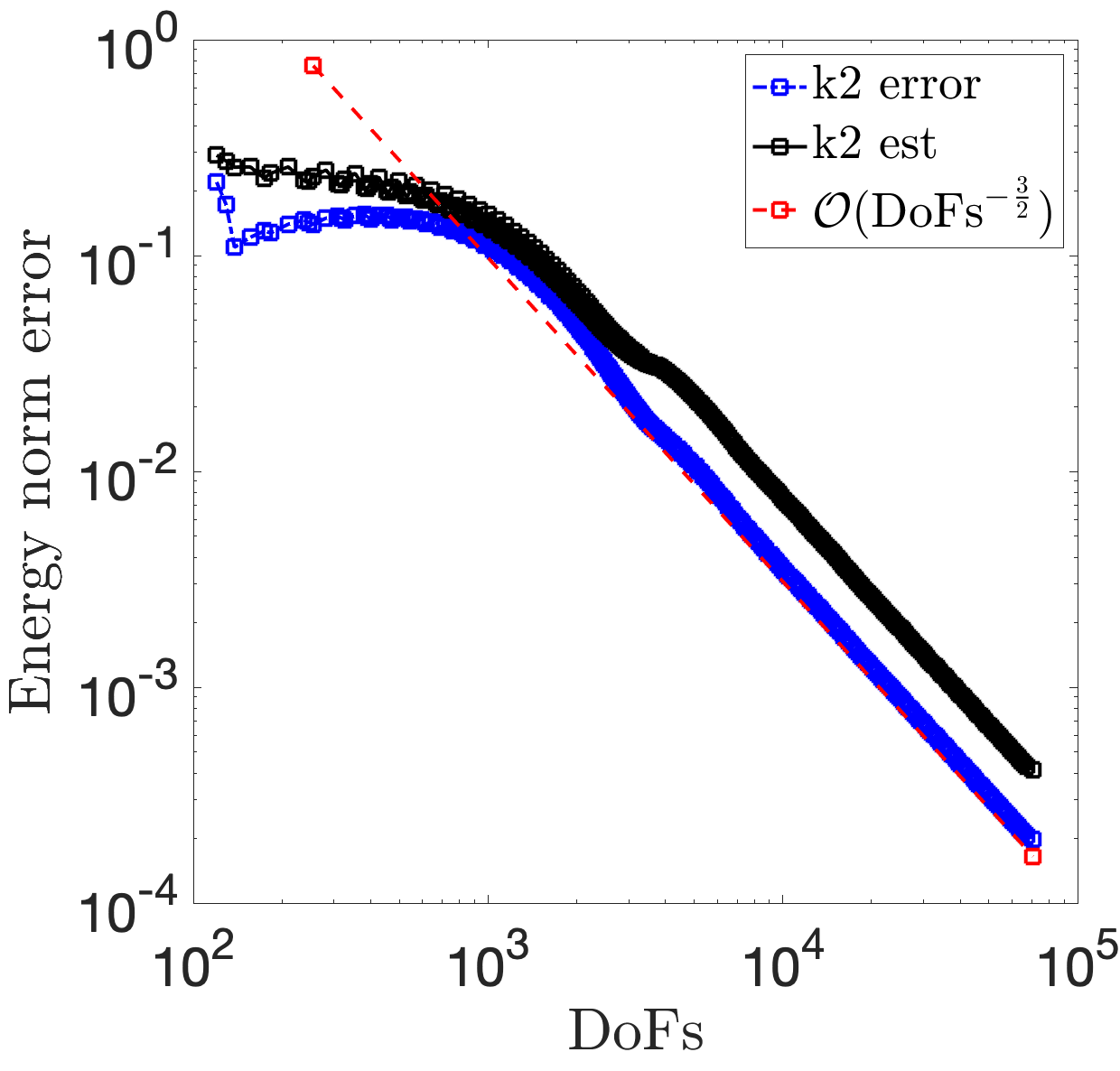} &
\includegraphics[width=0.46\linewidth]{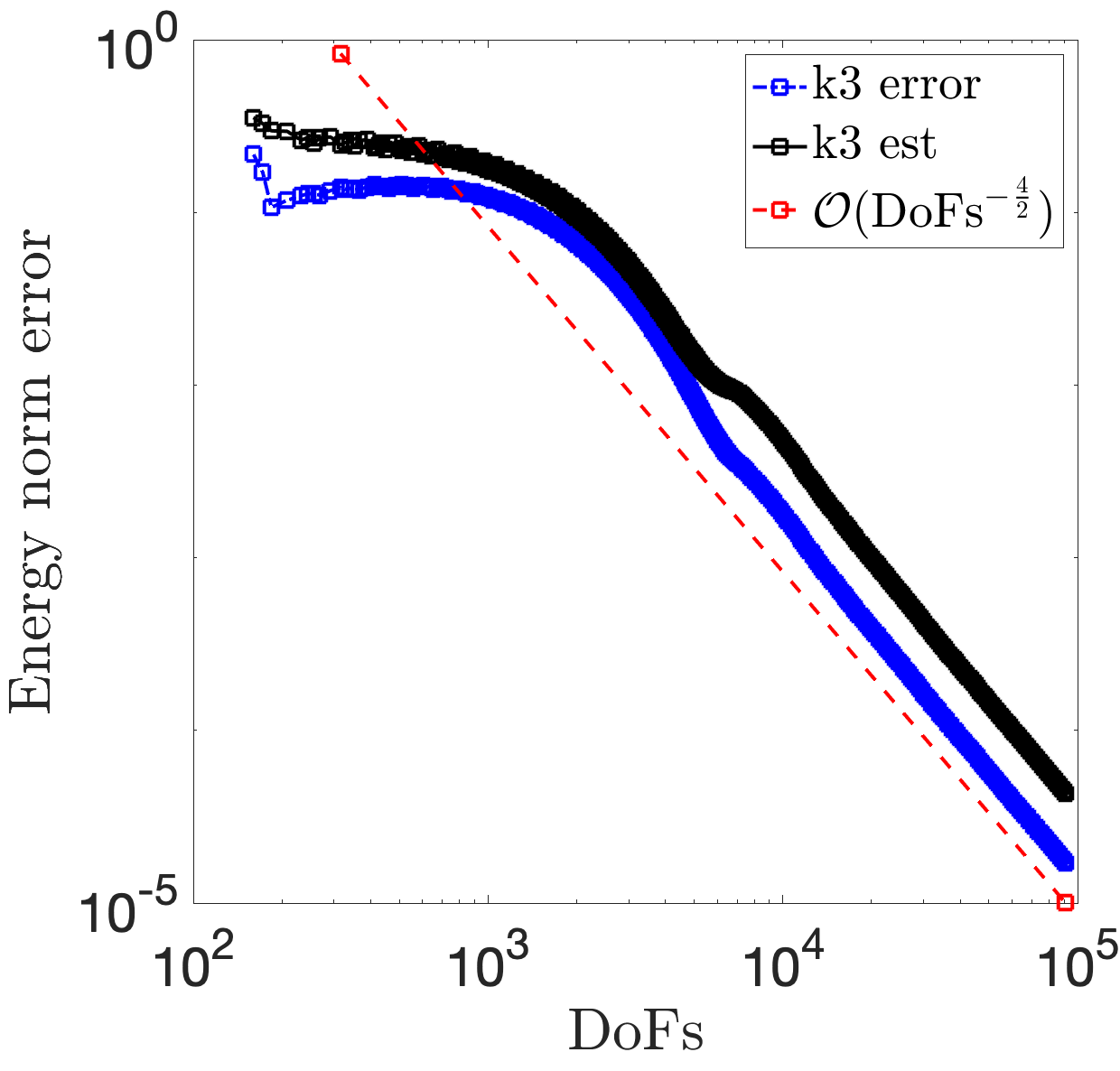}\\
\end{tabular}
\end{center}
\caption{Example 3. Energy error and a posteriori error estimator as a function of DoFs for $k\in\{0,1,2,3\}$.}\label{Ex3:error_estimator_k=0,1,2,3}
\end{figure}

\begin{figure}[!tb]
\begin{center}
\includegraphics[width=0.45\linewidth]{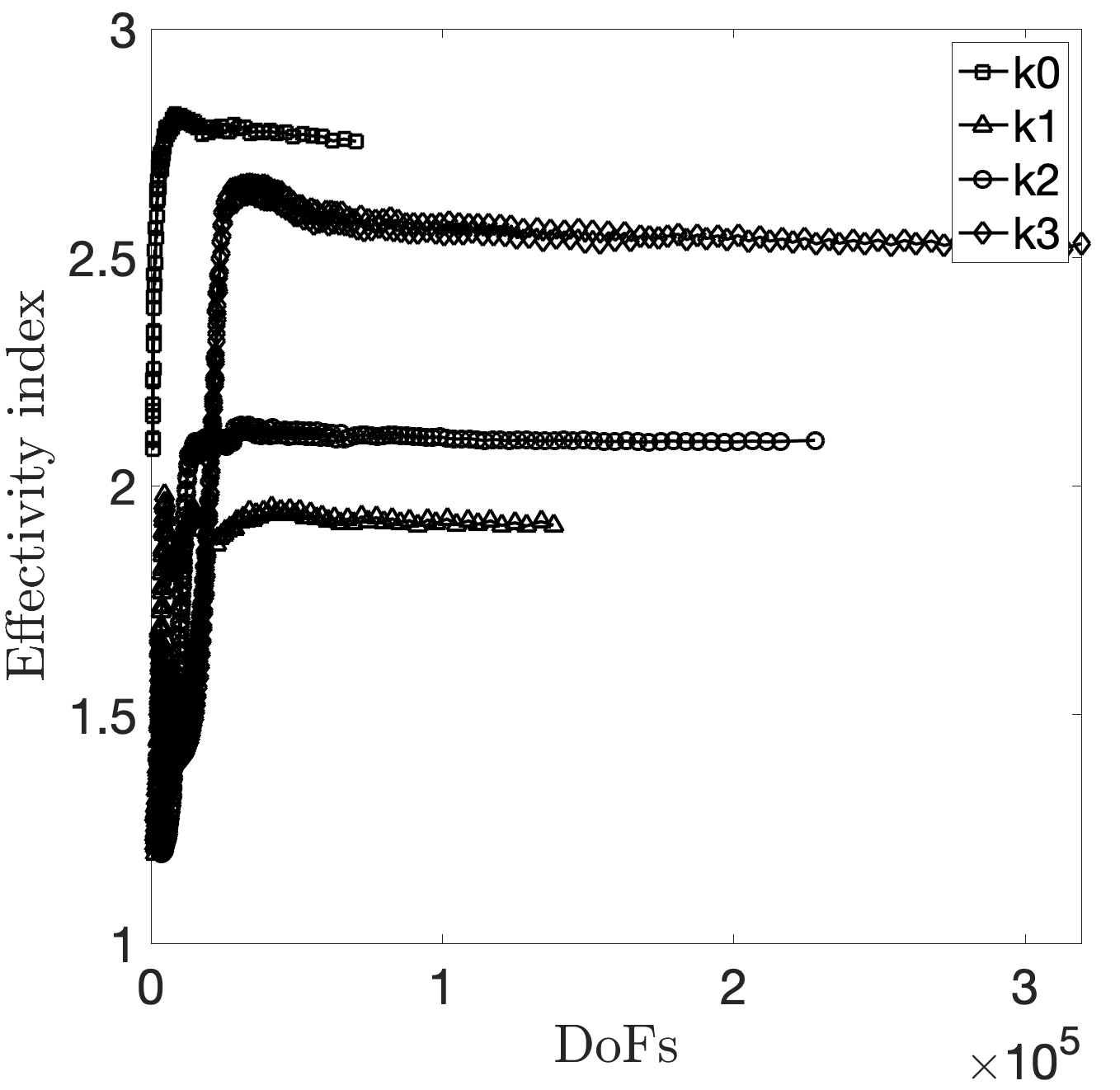}
\end{center}
\caption{Example 3. Effectivity index as a function of DoFs with $k\in\{0, 1,2,3\}$.}\label{Ex3:effectivity k=0, 1,2,3}
\end{figure}

Finally, we report in Table \ref{ex3:table Comparison} the relative contributions of the various terms composing the a posteriori error estimator (all in percentage) with polynomial degrees $k\in\{0,1,2\}$. For $k=1$, the tangential jump estimator dominates the total estimator (by more than $59\%$), followed by $\eta_{\rm nor}$ ($27\%$), while $\eta_{\rm sta}$ is below $10\%$ and $\eta_{\rm res}$ is zero. For $k=2$, the tangential jump estimator dominates the total estimator on the coarsest mesh, whereas the contributions of $\eta_{\rm tan}$, $\eta_{\rm nor}$ and $\eta_{\rm res}$ are fairly well equilibrated on the other meshes. For $k=0$, only the tangential jump residual is nonzero, for the same reasons as discussed in the previous example.
\begin{table}[!htb]
\centering
\begin{tabular}{|c|c|c|c|c|c|c|c|c|c|c|c|c|c|c|}
\hline
\multicolumn{5}{|c|}{$k=0$}
&
\multicolumn{5}{|c|}{$k=1$}
&
\multicolumn{5}{|c|}{$k=2$} \\
\hline
$\#$ cell  & $\eta_{\rm res}$ & $\eta_{\rm sta}$ &   $\eta_{\rm nor}$& $\eta_{\rm tan}$
&$\#$ cell  & $\eta_{\rm res}$ & $\eta_{\rm sta}$ &   $\eta_{\rm nor}$& $\eta_{\rm tan}$
& $\#$ cell & $\eta_{\rm res}$ & $\eta_{\rm sta}$ &   $\eta_{\rm nor}$&$\eta_{\rm tan}$  \\
\hline
$122$ & $0$ & $0$ &  $0$ & $100$
&$124$ & $0$ & $7$ &  $8$ & $85$
& $142$ & $22$  & $9$ &   $11$ &  $58$\\
\hline
$1386$ & $0$ & $0$ &  $0$ & $100$
&$1376$ & $0$ & $8$ &  $33$ & $59$
& $1326$ & $29$  & $9$ & $30$  & $32$ \\
\hline
$5962$ & $0$ & $0$ &  $0$ & $100$
&$5864$ & $0$ & $9$ &  $28$ & $63$
& $5486$ & $28$  & $8$ & $32$  & $32$ \\
\hline
$15642$ & $0$ & $0$ &  $0$ & $100$
&$15342$ & $0$ & $10$ &  $27$ & $63$
& $14894$ &$28$ & $8$ &  $32$  & $32$ \\
\hline
\end{tabular}
\caption{Example 3. Relative contribution (in $\%$) of the various terms composing the a posteriori error estimator for $k\in \{0,1,2\}$.}
\label{ex3:table Comparison}
\end{table}

\section{Proofs} \label{sec:proofs}

This section collects the proofs of all the preparatory results stated in the previous section, namely Lemma~\ref{Lemma: Conforming estimator p suboptimal}, Lemma~\ref{Lemma: Conforming estimator}, Lemma~\ref{lemma: nonconforming error bound on patches}, and
Theorem~\ref{thm:lowerbound}.

\subsection{Proof of Lemma~\ref{Lemma: Conforming estimator p suboptimal}}
We will prove the following result:
\begin{align}
\|A^{\frac12}\nabla e_c\|_{\Omega}  \leq{}& C_{\rm c,1} \bigg\{  \su   \Big\{ A_K^{-1} \Big(\frac{h_K}{k+1}\Big)^2\|\Pi_{K}^{k+1}(f)+ A_K\Delta R_K^{k+1}(\hat{u}_K)\|_K^2
+A_K(k+1)S_{\dK}(\hat{u}_K,\hat{u}_K) \nno \\
&+ O_K(f)^2 + O_K(g_\dn)^2  \Big\}  \bigg\}^{\frac12}  + \|A^{\frac12}\bnabla e_d\|_{\Omega}.
\end{align}%
The proof is split into two steps.

(i) Since $e_c \in H^1_{0,\ddd}(\Omega)$, using \eqref{def: error splitting} leads to
\begin{equation} \label{eq:error equation}
\begin{aligned}
\|A^{\frac12} \nabla e_c\|^2_{\Omega}  &=(A\bnabla e, \nabla e_c)_{\Omega}  - (A\bnabla e_d, \nabla e_c)_{\Omega}  .
\end{aligned}
\end{equation}
For the second term, the Cauchy--Schwarz inequality gives
\begin{equation}\label{eq:bound one}
|(A\bnabla e_d, \nabla e_c)|_{\Omega}  \leq \|A^{\frac12} \nabla e_c\|_{\Omega}  \|A^{\frac12} \bnabla e_d\|_{\Omega} .
\end{equation}
Next, we focus on bounding $(A\bnabla e, \nabla e_c)_{\Omega} $. Using the weak form of the PDE \eqref{eq:prob}, adding and subtracting the term $(A_K\nabla R_K^{k+1}(\hat{u}_K), \nabla e_c)_K$ for all $K\in\mesh$, and using the discrete problem \eqref{HHO:bilinear form} with some test function $\hat{w}_h\in \fesz$, we infer that
\begin{equation*}
\begin{aligned}
(A\bnabla e, \nabla e_c)_{\Omega} &= \su \Big\{ (f, e_c)_K + (g_{\rm N}, e_c)_{\dKN}
-  (A_K\nabla R_K^{k+1}(\hat{u}_K), \nabla e_c)_K \\
&  \qquad
+ (A_K\nabla (R_K^{k+1}(\hat{u}_K)- u_K), \nabla e_c)_K
+   (A_K\nabla R_K^{k+1}(\hat{u}_K), \nabla R_K^{k+1} (  \hat{w}_K ))_K\\
& \qquad
+ S_{\dK}(\hat{u}_K,  \hat{w}_K )
- (f, w_K)_K -  (g_{\rm N}, w_{\dK})_{\dKN}  \Big\}.
\end{aligned}
\end{equation*}
Let us set $\hat{w}_ {h}: =  \big((\ImKM(e_c)|_K)_{K\in\mesh}, (\ImKM(e_c)|_F)_{F\in\Fall}\big)$. This definition is meaningful
since $\ImKM(e_c)$ is single-valued at the mesh interfaces; moreover, since $\ImKM(e_c)$
vanishes on the boundary faces in $\FD$, we have $\hat{w}_h\in \fesz$. We observe that,
for all $K\in\mesh$,
the definition \eqref{reconstruction} of the reconstruction operator gives $R_K^{k+1}(\hat{w}_ {K}) = w_K$, and we also have $S_{\dK}(\hat{u}_K, \hat{w}_ {K})=0$. Hence, defining $\hat{\eta}_ {K}= \big( \eta_ {K}, \eta_ {\dK}\big)  := \big(e_c |_K-w_K , e_c |_{\dK} - w_{\dK}\big) $ for all $K\in\mesh$,  we infer that
\begin{equation*}
\begin{aligned}
(A\bnabla e, \nabla e_c)_{\Omega} &=  \su \Big\{ (f, \eta_K)_K + (g_{\rm N}, \eta_{\dK})_{\dKN}
-  (A_K\nabla R_K^{k+1}(\hat{u}_K), \nabla \eta_{K})_K\\
& \qquad
+ (A_K\nabla (R_K^{k+1}(\hat{u}_K)- u_K), \nabla e_c)_K\Big\}.
\end{aligned}
\end{equation*}
Integrating by parts and using that $\eta_{\dKD} = 0$ and that $A_K$ is constant on $K$ gives
\begin{align}
(A\bnabla e, \nabla e_c)_{\Omega}
&= \su \Big\{ (f+A_K  \Delta R_K^{k+1}(\hat{u}_K), \eta_{K})_K
-   (A_K\nabla R_K^{k+1}(\hat{u}_K) {\cdot} \n_K, \eta_{\dK})_{\dKi} \nonumber \\
&\qquad
 + (g_{\rm N}-A_K\nabla R_K^{k+1}(\hat{u}_K) {\cdot} \n_K, \eta_{\dK})_{\dKN}
+ (A_K\nabla (R_K^{k+1}(\hat{u}_K)- u_K), \nabla e_c)_K\Big\} \label{eq:error equation 2} .
\end{align}

(ii) Since $\eta_{\dK}|_F = (e_c-\ImKM(e_c))|_F$ is single-valued on each interface $F\in \FKi$,
using the local conservation property \eqref{conservation property} of the HHO method gives
\begin{equation*}
\begin{aligned}
(A\bnabla e, \nabla e_c)_{\Omega}  =&   \su \Big\{ (f+ A_K\Delta R_K^{k+1}(\hat{u}_K)), \eta_{K})_K
 - A_{K} \frac{(k+1)^2}{h_{K}}  ( \Pi_{\dK}^k (u_{K}|_{\dK}) - u_{\dK}, \eta_{\dK})_{\dKi }
\\
&
+ (g_{\rm N}- \Pi_{\dK}^k(g_{\rm N}|_{\dK}), \eta_{\dK})_{\dKN}
 - A_{K} \frac{(k+1)^2}{h_{K}}( \Pi_{\dK}^k (u_{K}|_{\dK}) - u_{\dK},  \eta_{\dK})_{\dKN}
\\
&+ (A_K\nabla (R_K^{k+1}(\hat{u}_K)- u_K), \nabla e_c)_K  \Big\}.
\end{aligned}
\end{equation*}
Invoking the Cauchy--Schwarz inequality leads to
\begin{equation*}
\begin{aligned}
\big|(A\bnabla e, \nabla e_c)_{\Omega}\big|
 \leq &\su \bigg\{ \Big(A_K^{-\frac12} \frac{h_K}{k+1}\Big)\|f+ A_K\Delta R_K^{k+1}(\hat{u}_K)\|_K
 \Big(A_K^{\frac12} \frac{k+1}{h_K}\Big) \| \eta_{K}\|_K
\\
& +    \Big(A_{K} \frac{(k+1)^{3}} {h_{K}}\Big)^{\frac12} \| \Pi_{\dK}^k (u_{K}|_{\dK}) - u_{\dK}\|_{\dKi }
\Big(A_K \frac{k+1}{h_K}\Big)^{\frac12} \|\eta_{\dK}\|_{\dKi }
\\
&
+
 \Big(A_K^{-1} \frac{h_K}{k+1}\Big)^{\frac12} \|g_{\rm N}- \Pi_{\dK}^k(g_{\rm N}|_{\dK})\|_{\dKN}
 \Big(A_K \frac{k+1}{h_K}\Big)^{\frac12} \|\eta_{\dK}\|_{\dKN }
\\
&
 +   \Big(A_{K} \frac{(k+1)^{3}} {h_{K}}\Big)^{\frac12} \|\Pi_{\dK}^k (u_{K}|_{\dK}) - u_{\dK}\|_{\dKN }
\Big(A_K \frac{k+1}{h_K}\Big)^{\frac12} \|\eta_{\dK}\|_{\dKN }
\\
&+ A_K^{\frac12} \|\nabla (R_K^{k+1}(\hat{u}_K)- u_K)\|_K A_K^{\frac12} \|\nabla e_c\|_K \bigg\}.
\end{aligned}
\end{equation*}
Using the approximation result \eqref{eq: Modified KM approximation}, we infer that, for all $K\in \mesh$,
\begin{equation}\label{useful relation}
A_K^{\frac12} \Big\{ \Big(\frac{k+1}{h_K}\Big) \| \eta_{K}\|_K + \Big(\frac{k+1}{h_K}\Big)^{\frac12} \| \eta_{\dK}\|_{\dK} \Big\}
\leq  CA_K^{\frac12} \| \nabla e_c\|_{\mathrm{es}(K)}
\leq C \chi_K(A) \| A^{\frac12}\nabla e_c\|_{\mathrm{es}(K)}.
\end{equation}
Using the above bounds, the bound \eqref{local HHO  bound} on $\|\nabla (R_K^{k+1}(\hat{u}_K)- u_K)\|_K$, the mesh shape-regularity and the triangle inequality for the residual term gives
\begin{align*}
\big|(A\bnabla e, \nabla e_c)_{\Omega} \big|   \leq {}&  C \max_{K\in\mesh} \chi_K(A)
\bigg\{  \su  \Big\{  A_K^{-1}  \Big(\frac{h_K}{k+1}\Big)^2 \|\Pi_K^{k+1}(f)+ A_K \Delta R_K^{k+1}(\hat{u}_K)\|_K^2
\\
&  + A_K(k+1)S_{\dK}(\hat{u}_K,\hat{u}_K)
+A_KS_{\dK}(\hat{u}_K,\hat{u}_K) + O_K(f)^2 + O_K(g_{\dn})^2
\Big\} \bigg\}^{\frac12} \|A^{\frac12}\nabla e_c\|_{\Omega} .
\end{align*}
Combining \eqref{eq:error equation} and \eqref{eq:bound one} with the above bound completes the proof.

\subsection{Proof of Lemma~\ref{Lemma: Conforming estimator}}

We will prove the following result:
\begin{align}
\|A^{\frac12}\nabla e_c\|_{\Omega}  \leq {}& C_{\rm c,2} \bigg\{  \su \Big\{  A_K^{-1} \Big(\frac{h_K}{k+1}\Big)^2\|\Pi_{K}^{k+1}(f)+ A_K\Delta R_K^{k+1}(\hat{u}_K)\|_K^2 +A_KS_{\dK}(\hat{u}_K,\hat{u}_K) \nno \\
& + A_K^{-1} \Big(\frac{h_K}{k+1}\Big) \Big(   \| \jumpK{A\nabla R_{\mesh}^{k+1}(\hat{u}_h)} {\cdot}\n_K \|_{\dKi}^2
+  \ltwo{A_K\nabla R_K^{k+1}(\hat{u}_K)  {\cdot}\n_{K} - \Pi_{\dK}^k (g_{\dn}|_{\dK})}{\dKN}^2 \Big) \nno \\
&+O_K(f)^2 +  O_K(g_{\dn})^2  \Big\} \bigg\}^{\frac12} + \|A^{\frac12}\bnabla e_d\|_{\Omega}.
\end{align}
The starting point is the identity \eqref{eq:error equation 2} obtained at the end of the first step of the proof of Lemma \ref{Lemma: Conforming estimator p suboptimal}. However, we no longer invoke the local conservation property \eqref{conservation property}, but consider the jump of $A\nabla R_{\mesh}^{k+1}(\hat{u}_h)$ across all the mesh interfaces and its value at all the Neumann boundary faces. Recalling that $\eta_{\dK}|_F = (e_c-\ImKM(e_c))|_F$ is single-valued on every interface $F\in \FKi$, we infer from \eqref{eq:error equation 2} that
\begin{equation*}
\begin{aligned}
(A\bnabla e, \nabla e_c)_{\Omega}
={}& \su  \Big\{ (f+ A_K \Delta R_K^{k+1}(\hat{u}_K), \eta_{K})_K
+ (A_K\nabla (R_K^{k+1}(\hat{u}_K)- u_K), \nabla e_c)_K \\
&-
 \sum_{F\in \FKi}   \frac12( \jump{A\nabla R_{\mesh}^{k+1}(\hat{u}_h)} {\cdot}\n_F , \eta_{\dK})_{F}
\\
&
-   (A_K\nabla R_K^{k+1}(\hat{u}_K)  {\cdot}\n_{K}-\Pi_{\dK}^k(g_{\rm N}|_{\dK}),\eta_{\dK})_{\dKN}
+(g_{\rm N}-\Pi_{\dK}^k(g_{\rm N}|_{\dK}), \eta_{\dK})_{\dKN}
\Big\}.
\end{aligned}
\end{equation*}
Invoking the Cauchy--Schwarz inequality leads to
\begin{align*}
\big|(A\bnabla e, \nabla e_c)_{\Omega}\big|
\leq {} &\su \bigg \{ \Big(A_K^{-\frac12} \frac{h_K}{k+1}\Big)\|f+ A_K\Delta R_K^{k+1}(\hat{u}_K)\|_K
\Big(A_K^{\frac12} \frac{k+1}{h_K}\Big) \| \eta_{K}\|_K
\\
& + A_K^{\frac12} \|\nabla (R_K^{k+1}(\hat{u}_K)- u_K)\|_K  \|A^{\frac12}\nabla e_c\|_K \\
& +   \sum_{F\in \FKi} \frac12 \Big(A_{K}^{-1} \frac{h_K} {k+1}\Big)^{\frac12} \| \jump{A\nabla R_{\mesh}^{k+1}(\hat{u}_h)} {\cdot}\n_F \|_{F }
\Big(A_K \frac{k+1}{h_K}\Big)^{\frac12} \|\eta_{\dK}\|_{F}
\\
&+
 \Big(A_K^{-1} \frac{h_K}{k+1}\Big)^{\frac12} \|A_K\nabla R_K^{k+1}(\hat{u}_K)  {\cdot}\n_{\Omega}-\Pi_{\dK}^k ({g}_{\dn}|_{\dK})\|_{\dKN}
 \Big(A_K \frac{k+1}{h_K}\Big)^{\frac12} \|\eta_{\dK}\|_{\dKN } \\
&
+
\Big(A_K^{-1} \frac{h_K}{k+1}\Big)^{\frac12} \|g_{\rm N}- \Pi_{\dK}^k(g_{\rm N}|_{\dK})\|_{\dKN}
\Big(A_K \frac{k+1}{h_K}\Big)^{\frac12} \|\eta_{\dK}\|_{\dKN }
\bigg\}.
\end{align*}
Using the Cauchy--Schwarz inequality, the bound~\eqref{local HHO  bound} on
$\|\nabla (R_K^{k+1}(\hat{u}_K)- u_K)\|_K$,
the approximation result \eqref{useful relation} on $\eta$, and the triangle inequality
for the residual term, we obtain
\begin{align*}
\big|(A\bnabla e, \nabla e_c)_{\Omega}\big|
\leq {}& C \max_{K\in\mesh}\chi_K(A) \Bigg\{  \su  \Big\{   A_K^{-1} \Big(\frac{h_K}{k+1}\Big)^2 \|\Pi_K^{k+1}(f)+ A_K \Delta R_K^{k+1}(\hat{u}_K)\|_{K}^2 + A_K S_{\dK}(\hat{u}_K,\hat{u}_K) \\
& +    A_K^{-1} \Big(\frac{h_K}{k+1}\Big)  \Big(  \|\jumpK{A\nabla R_{\mesh}^{k+1}(\hat{u}_h)} {\cdot}\n_K \|_{\dKi}^2
+ \ltwo{ A_K\nabla R_K^{k+1}(\hat{u}_K) {\cdot} \n_{\Omega}-\Pi_{\dK}^k ({g}_{\dn}|_{\dK})}{\dKN}^2 \Big) \\
& + O_K(f)^2 + O_K(g_{\dn})^2  \Big\} \Bigg\}^{\frac12}  \|A^{\frac12}\nabla e_c\|_{\Omega}.
\end{align*}
Combining \eqref{eq:error equation} and \eqref{eq:bound one}  with the above bound completes the proof.

\subsection{Proof of Lemma~\ref{lemma: nonconforming error bound on patches}} \label{sec:proof_nonconf}

We will prove the following result: For all $\ba \in \vertice$,
\begin{align}
\|A^{\frac12}\bnabla e_{d}^{\ba}\|_{\oma} \le {}&
C_{\rm d}^{\ba} \bigg\{ \!\sua\!  \Big\{ A_K S_{\dK}(\hat{u}_K,\hat{u}_K) + O_K(g_{\ddd})^2 \Big\}
+ \!\! \sum_{F\in \Fa \cap \Fint  } \!\!
A_F^\flat \Big(\frac{h_F}{k+1}\Big) \ltwo{\jump{ \nabla   u_{\mesh}} {\times} \n_F}{F}^2 \nno \\
& + \!\!\sum_{F\in \Fa \cap \FD }\!\! A_F^\flat \Big(\frac{h_F}{k+1}\Big)
\ltwo{ \nabla ( u_{\mesh}-\Pi_{\dK}^{k+1}(g_\ddd|_{\dK}))  {\times} \n_{\Omega}}{F}^2 \bigg\}^{\frac12}.
\end{align}%
To fix the ideas, we assume that $d=3$; all the results also hold for $d=2$.

(1) First, we recall a result on the Helmholtz decomposition on simply connected domains \cite[Theorem 3]{HP19_822}.
Let $\omega$ be a bounded, simply connected, Lipschitz domain and consider a partition of its boundary $\partial\omega$
into two disjoint Lipschitz parts $\gamma_{\ddd}$ and $\gamma_{\dn}$. Then,
for all $\w \in \bL^2(\omega)$, there exist
\begin{subequations} \begin{align}
\xi\in H^1_{0,\ddd} (\omega) &:= \{ \zeta \in H^1(\omega) \,|\, \zeta|_{\gamma_{\ddd}} = 0\}, \\
\bphi\in  \bH^1_{0,\dn}(\omega)&:=\{\bpsi\in \bH^1(\omega)\,|\, \bpsi|_{\gamma_{\dn}} = {\bm 0}  \},
\end{align}\end{subequations}
such that
\begin{equation}\label{eq:helm}
A\w=A\nabla\xi+\textbf{curl}\,\bphi\quad \text{in }\, \omega.
\end{equation}
Moreover, the following holds with a constant $C_{\omega}$ only depending on $\omega$ and its boundary partition:
\begin{subequations} \label{eq: HD stability}
\begin{align}
\ltwo{A^{\frac12}\w}{\omega}^2&=\ltwo{A^{\frac12}\nabla \xi}{\omega}^2+\ltwo{A^{-\frac12}\textbf{curl}\,\bphi}{\omega}^2,\label{eq:helm-norm} \\
\ltwo{\nabla \bphi}{\omega}&\leq C_{\omega} \ltwo{\textbf{curl}\,\bphi}{\omega}. \label{eq:curl norm}
\end{align}
\end{subequations}
For $d=2$, one has $C_{\omega}=1$. For $d=3$, the constant $C_{\omega}$ is more delicate to estimate.
However, when $\omega$ is a vertex patch, estimates are available.
For an interior vertex, the associated patch is star-shaped with respect to a ball whose radius is comparable to the diameter of the patch (see \cite[Proposition~8.2]{chaumontfrelet:hal-05204325}). Consequently, an explicit bound on $C_{\omega}$ follows from \cite[Corollary~29]{GuzSal21}.
For a boundary vertex, the associated patch can be viewed as a chain of star-shaped domains. If the patch is subject only to homogeneous Dirichlet boundary conditions, \cite[Theorem~35]{GuzSal21} applies, whereas, in the presence of mixed boundary conditions, an explicit bound on $C_{\omega}$ follows from \cite[Theorem~1.4]{BottiMascotto2026}.

(2) We apply the Helmholtz decomposition \eqref{eq:helm} to $\w:=\bnabla e_{d}^{\ba}$ on the vertex star $\omega:=\oma$ with the boundary partition such that $\gamma_{\ddd}:=\partial\oma \cap (\Omega\cup\Gamma_{\ddd})$ and $\gamma_{\dn}:=\partial\oma \cap \Gamma_{\dn}$. This
gives $\xi\in H^1_{0,\ddd} (\oma)$ and $\bphi\in \bH^1_{0,\dn}(\oma)$ such that
$A\bnabla e_{d}^{\ba} = A\nabla\xi + \textbf{curl}\,\bphi$ in $\oma$. Notice that $H^1_{0,\ddd} (\oma)$ is indeed the functional space defined in~\eqref{def: Vertx space} for $g=0$.
Notice also that the stability constant $C_\omega$ in~\eqref{eq: HD stability} only depends on the
mesh shape-regularity, since, in particular, there is only a finite number of possible partitions of the boundary of $\oma$ and this number is bounded in terms of the mesh shape-regularity.
Taking the $L^2(\oma)$-inner product with $\bnabla e_{d}^{\ba}$ and observing that
$(A \bnabla e_{d}^{\ba}, \nabla \xi)_{\oma}=0$ owing to the
definition~\eqref{def: Patch reconstruction} of $u_c^{\ba}$,
we infer that
\[
\ltwo{A^{\frac12}\bnabla e_{d}^{\ba}}{\oma}^2
=  ( \bnabla e_{d}^{\ba},  \textbf{curl}\,\bphi )_{\oma}.
\]

(3) We consider a modified $hp$-Karkulik--Melenk interpolation operator
on the vertex star $\oma$ preserving the homogeneous boundary condition on $\gamma_{\dn} = \partial\oma \cap \Gamma_{\dn}$, say $\IamKM$ (recall that $k\geq1$ by assumption).
This operator is constructed as in the proof of Corollary~\ref{cor: global interpolation of KM} and leads to
the following local approximation result: For all $\ba\in\vertice$, all
$\bphi\in \bH^1_{0,\dn} (\oma)$, and all $K\in\mesha$,
\begin{equation} \label{eq:approx_IamKM}
\Big(\frac{k}{h_K}\Big)^{2}\|\bphi - \IamKM (\bphi)\|^2_{K} + \Big(\frac{k}{h_K}\Big)\|\bphi - \IamKM (\bphi)\|^2_{\partial K}  + \|\nabla \IamKM (\bphi)\|^2_{K}
 \leq C_{\rm mKM}^{\vertice} \| \nabla \bphi\|^2_{\oma},
\end{equation}
where the constant $C_{\rm mKM}^{\vertice}$ is bounded in terms of the mesh
shape-regularity. Notice that the bound on $C_{\rm mKM}^{\vertice}$ is uniform over the mesh vertices; this follows by observing that the
Karkulik--Melenk interpolation operator is itself constructed vertex-wise and that the
nodal averaging operator is uniformly stable over the mesh vertices.
Notice also that $\mathrm{es}(K) = \oma$ since $\IamKM$
is locally constructed in the vertex star $\oma$.
Considering the vector-valued version of the above operator
(simply acting componentwise), we obtain
\[
\ltwo{A^{\frac12}\bnabla e_{d}^{\ba}}{\oma}^2 =
( \bnabla e_{d}^{\ba}, \textbf{curl}\,({\bphi}-\IamKM({\bphi})) )_{\oma}
+ ( \bnabla e_{d}^{\ba}, \textbf{curl}\,\IamKM ({\bphi}) )_{\oma} =: T_1+T_2,
\]
and it remains to bound $T_1$ and $T_2$.

(4) Bound on $T_1$. Since $\textbf{curl}\,\nabla u^{\ba}_c =\mbf{0}$,
$({\bphi}-\IamKM({\bphi}))|_{\partial\oma\cap\Gamma_{\dn}}=\mbf{0}$, and
$u_{c}^{\ba} \in H^1_{g,\ddd} (\oma)$ (see~\eqref{def: Vertx space}), integrating by parts
the curl operator gives
\[
( \nabla u_{c}^{\ba}, \textbf{curl}\,({\bphi}-\IamKM({\bphi})) )_{\oma}
= ({\bphi} -\IamKM ({\bphi} ), \nabla (\psi_{\ba} g_\ddd) {\times} \n_\Omega  )_{\partial \oma\cap \Gamma_{\ddd}}.
\]
Similar arguments, together with the fact that ${\bphi}- \IamKM({\bphi})$ is single-valued across every interface $F\in\Fa\cap \Fint$, give
\begin{align*}
( \bnabla (\psi_{\ba}u_{\mesh}), \textbf{curl}\,({\bphi}-\IamKM({\bphi})) )_{\oma}
= {}& \sum_{F\in \Fa\cap \Fint} ({\bphi}-\IamKM ({\bphi}) ,  \jump{\bnabla  (\psi_{\ba}  u_{\mesh})} {\times} \n_F  )_{F} \\
& + \sum_{F\in\Fa \cap \FD} ({\bphi} -\IamKM ({\bphi}),  \nabla (\psi_{\ba} u_{\mesh}) {\times} \n_\Omega)_{F}.
\end{align*}
Putting the above two identities together and using that
\[ \bnabla e_{d}^{\ba} =
\nabla u_{c}^{\ba} - \bnabla (\psi_{\ba}u_{\mesh}),
\qquad
\bnabla  (\psi_{\ba}  u_{\mesh})
= \psi_{\ba}\bnabla u_{\mesh} + u_{\mesh}\nabla\psi_{\ba},
\]
and that $(\nabla \psi_{\ba}  {\times} \n_F )_F$ is single-valued across every interface $F\in\Fa \cap \Fint$, we infer that
\begin{align*}
T_1
= {}& -\!\!\sum_{F\in \Fa\cap \Fint} \!\!({\bphi} -\IamKM ({\bphi}) ,\psi_{\ba}\jump{\nabla   u_{\mesh}} {\times} \n_F)_{F}
- \!\!\sum_{F\in\Fa \cap \FD} \!\!({\bphi} -\IamKM ({\bphi}), \psi_{\ba} \nabla (  u_{\mesh}-g_\ddd) {\times} \n_\Omega)_F \\
&
 - \!\!\sum_{F\in \Fa\cap \Fint} \!\!({\bphi} -\IamKM ({\bphi}) , ( \nabla \psi_{\ba}  {\times} \n_F ) \jump{  u_{\mesh}})_{F}
- \!\!\sum_{F\in\Fa\cap \FD} \!\!({\bphi} -\IamKM ({\bphi}), ( \nabla \psi_{\ba} {\times} \n_\Omega) (  u_{\mesh}-g_\ddd))_F.
\end{align*}
Let $T_{11}$ and $T_{12}$ denote, respectively, the terms on the first and second lines on the
above right-hand side.
Using the Cauchy--Schwarz inequality and $\norm{\psi_{\ba}}{L^\infty(F)}=1$, we infer that
\begin{align*}
|T_{11}|
\leq {}&  \bigg\{ \!\! \sum_{F\in \Fa \cap \Fint  } \!\!
\Big(\frac{h_F}{k+1}\Big) \ltwo{\jump{ \nabla   u_{\mesh}} {\times} \n_F}{F}^2
+ \!\!\sum_{F\in \Fa \cap \FD }\!\! \Big(\frac{h_F}{k+1}\Big)
\ltwo{ \nabla ( u_{\mesh}-g_\ddd)  {\times} \n_{\Omega}}{F}^2 \bigg\}^{\frac12} \\
& \times \bigg\{ \!\!  \sum_{F\in \Fa \cap (\Fint \cup \FD) } \!\!
\Big(\frac{k+1}{h_F}\Big)\ltwo{{\bphi} -\IamKM ({\bphi})}{F}^2 \bigg\}^{\frac12}.
\end{align*}
For every $F\in \Fa \cap (\Fint \cup \FD)$, we can pick a mesh cell $K\in\mesha$
of which $F$ is a face and obtain from~\eqref{eq:approx_IamKM} that
\[
\Big(\frac{k+1}{h_F}\Big)^{\frac12}\ltwo{{\bphi} -\IamKM ({\bphi})}{F}
\le C \|\nabla \bphi\|_{\oma}.
\]
Putting the above two bounds together gives
\[
|T_{11}|
\leq C \bigg\{ \!\! \sum_{F\in \Fa \cap \Fint  } \!\!
\Big(\frac{h_F}{k+1}\Big) \ltwo{\jump{ \nabla   u_{\mesh}} {\times} \n_F}{F}^2
+ \!\!\sum_{F\in \Fa \cap \FD }\!\! \Big(\frac{h_F}{k+1}\Big)
\ltwo{ \nabla ( u_{\mesh}-g_\ddd)  {\times} \n_{\Omega}}{F}^2 \bigg\}^{\frac12}
\|\nabla \bphi\|_{\oma}.
\]
Owing to~\eqref{eq: HD stability}, we infer that
\begin{equation}\label{eq:bounds on Phi}
\ltwo{ \nabla \bphi }{\oma }
\leq C_{\oma} \ltwo{\textbf{curl}\bphi}{\oma}
\leq C_{\oma}  (A_{\ba}^{\sharp})^{\frac12} \ltwo{A^{-\frac12} \textbf{curl}\bphi }{\oma }
\leq   C_{\oma}   (A_{\ba}^{\sharp})^{\frac12} \ltwo{\diff^{\frac12}\bnabla e_d}{\oma}.
\end{equation}
Combining the above bounds and recalling the definition~\eqref{eq:def_contrast} of $\chi_{\ba}(A)$ and the definition~\eqref{eq:def_AflatF} of $A_F^\flat$, we obtain
\begin{align*}
|T_{11}|
\leq {}& C \chi_{\ba}(A)^{\frac12} \bigg\{ \!\! \sum_{F\in \Fa \cap \Fint  } \!\!
A_F^\flat \Big(\frac{h_F}{k+1}\Big) \ltwo{\jump{ \nabla   u_{\mesh}} {\times} \n_F}{F}^2
+ \!\!\sum_{F\in \Fa \cap \FD }\!\! A_F^\flat \Big(\frac{h_F}{k+1}\Big)
\ltwo{ \nabla ( u_{\mesh}-g_\ddd)  {\times} \n_{\Omega}}{F}^2 \bigg\}^{\frac12} \\
& \times \|\diff^{\frac12}\bnabla e_d\|_{\oma}.
\end{align*}
Let us now bound $T_{12}$. Letting $\mathbb{I}$ be the identity operator, we write
$T_{12}=T_{12a}+T_{12b}$ with
\begin{align*}
T_{12a} ={}& - \!\!\sum_{F\in \Fa\cap \Fint}\!\! ({\bphi} -\IamKM ({\bphi}) , ( \nabla \psi_{\ba} {\times} \n_F ) (\mathbb{I} - \Pi_{F}^{k})\jump{  u_{\mesh}})_{F}	\\
&-\!\!\sum_{F\in \Fa  \cap \FD}\!\! ({\bphi} -\IamKM ({\bphi}), ( \nabla \psi_{\ba}  {\times} \n_\Omega) (\mathbb{I} - \Pi_{F}^{k})(  u_{\mesh}-g_\ddd)|_F)_{F}\\
T_{12b} = {}& - \!\!\sum_{F\in \Fa\cap \Fint}\!\! ({\bphi} -\IamKM ({\bphi}) , (\nabla \psi_{\ba} {\times} \n_F)  \Pi_{F}^{k} \jump{  u_{\mesh}})_{F}	\\
&- \!\!\sum_{F\in \Fa \cap \FD}\!\! ({\bphi} -\IamKM ({\bphi}), ( \nabla \psi_{\ba} {\times} \n_\Omega) ( \Pi_{F}^k (  u_{\mesh}-g_\ddd)|_F)_{F}.
\end{align*}
Using the Cauchy--Schwarz inequality and $\norm{ \nabla \psi_{\ba}  {\times} \n_F}{L^\infty(F)}\leq C h^{-1}_F$ gives
\begin{align*}
|T_{12a}| \le {}&C
\bigg\{ \!\!\sum_{F\in \Fa \cap \Fint  }\!\! \Big(\frac{h_F}{k+1}\Big) h_F^{-2}\ltwo{ (\mathbb{I} - \Pi_{F}^{k})\jump{  u_{\mesh}}}{F}^2 + \!\!\sum_{F\in \Fa \cap \FD } \!\!
\Big(\frac{h_F}{k+1}\Big)  h_F^{-2} \ltwo{(\mathbb{I} - \Pi_{F}^{k})(  u_{\mesh}-g_\ddd)|_F}{F}^2
\bigg\}^{\frac12} \\
& \times \bigg\{ \!\!\sum_{F\in \Fa \cap (\Fint \cup  \FD) }\!\!
\Big(\frac{k+1}{h_F}\Big)\ltwo{{\bphi} -\IamKM ({\bphi})}{F}^2\bigg\}^{\frac12}.
\end{align*}
Invoking the Poincar\'e inequalities
$\norm{(\mathbb{I} - \Pi_{F}^{k})\jump{  u_{\mesh}}}{F} \leq Ch_F\norm{\jump{ \nabla  u_{\mesh}} {\times}  \n_F }{F}$ for all $F\in \Fa \cap \Fint$ and
$\ltwo{(\mathbb{I} - \Pi_{F}^{k})(  u_{\mesh}-g_\ddd)|_F}{F} \leq Ch_F\ltwo{  \nabla ( u_{\mesh}-g_\ddd) {\times} \n_\Omega}{F}$ for all $F\in \Fa \cap \FD$, we infer that
\begin{align*}
|T_{12a}| \le {}&C
\bigg\{ \!\!\sum_{F\in \Fa \cap \Fint  }\!\! \Big(\frac{h_F}{k+1}\Big) \norm{\jump{ \nabla  u_{\mesh}} {\times}  \n_F }{F}^2 + \!\!\sum_{F\in \Fa \cap \FD } \!\!
\Big(\frac{h_F}{k+1}\Big) \ltwo{  \nabla ( u_{\mesh}-g_\ddd) {\times} \n_\Omega}{F}^2
\bigg\}^{\frac12} \\
& \times \bigg\{ \!\!\sum_{F\in \Fa \cap (\Fint \cup  \FD) }\!\!
\Big(\frac{k+1}{h_F}\Big)\ltwo{{\bphi} -\IamKM ({\bphi})}{F}^2\bigg\}^{\frac12}.
\end{align*}
Invoking the same arguments as above to bound the term involving $\bphi$, we conclude that
\begin{align*}
|T_{12a}| \le {}&C\chi_{\ba}(A)^{\frac12}
\bigg\{ \!\!\sum_{F\in \Fa \cap \Fint  }\!\! A_F^\flat \Big(\frac{h_F}{k+1}\Big) \norm{\jump{ \nabla  u_{\mesh}} {\times}  \n_F }{F}^2 + \!\!\sum_{F\in \Fa \cap \FD } \!\!
A_F^\flat \Big(\frac{h_F}{k+1}\Big) \ltwo{  \nabla ( u_{\mesh}-g_\ddd) {\times} \n_\Omega}{F}^2
\bigg\}^{\frac12} \\
& \times \|\diff^{\frac12}\bnabla e_d\|_{\oma}.
\end{align*}
Turning our attention to $T_{12b}$, we observe that $u_F$ is single-valued at every interface $F\in\Fa\cap\Fint$, and that $u_F = \Pi_{F}^{k} (g_{\ddd}|_F )$ on every Dirichlet boundary face $F\in\Fa\cap\FD$. Therefore, we have
\begin{align*}
T_{12b} = {}& - \!\!\sum_{F\in \Fa\cap \Fint}\!\! ({\bphi} -\IamKM ({\bphi}) , (\nabla \psi_{\ba} {\times} \n_F)  \Pi_{F}^{k} \jump{  u_{\mesh} - u_F})_{F}	\\
&- \!\!\sum_{F\in \Fa \cap \FD}\!\! ({\bphi} -\IamKM ({\bphi}), ( \nabla \psi_{\ba} {\times} \n_\Omega) ( \Pi_{F}^k (  u_{\mesh}|_F)-u_F)_{F}.
\end{align*}
Using the Cauchy--Schwarz inequality and $\norm{\nabla \psi_{\ba}  {\times} \n_F}{L^\infty(F)}\leq C h^{-1}_F$ gives
\begin{align*}
|T_{12b}| \le {}& C \bigg\{\sua \sum_{F\in \FK} \Big(\frac{h_F}{k+1}\Big) h_F^{-2}
\ltwo{ \Pi_{\dK}^{k} ( u_K|_{\dK})-u_F}{F}^2 \bigg\}^{\frac12} \\
& \times \bigg\{ \sum_{F\in \Fa \cap (\Fint \cup  \FD) }
\Big(\frac{k+1}{h_F}\Big)\ltwo{{\bphi} -\IamKM ({\bphi})}{F}^2\bigg\}^{\frac12},
\end{align*}
where, in the first factor, we re-organized the sum over the mesh faces in $\Fa$ as a sum
over the mesh cells in $\mesha$.
Invoking the same arguments as above to bound the factor involving $\bphi$ and recalling
the definition~\eqref{def: stabilization} of the stabilization bilinear form $S_{\dK}$, we obtain
\begin{align*}
|T_{12b}|\le {}& C\chi_{\ba}(A)^{\frac12} (k+1)^{-\frac32} \bigg\{\sua  A_K S_{\dK}(\hat{u}_K,\hat{u}_K)
\bigg\}^{\frac12} \|\diff^{\frac12}\bnabla e_d\|_{\oma} \\
\le {}& C\chi_{\ba}(A)^{\frac12} \bigg\{\sua  A_K S_{\dK}(\hat{u}_K,\hat{u}_K)
\bigg\}^{\frac12} \|\diff^{\frac12}\bnabla e_d\|_{\oma}.
\end{align*}
Notice that we dropped the (favorable) factor $(k+1)^{-\frac32}$ since the bound
on $T_2$ derived below does not involve this factor. Putting everything
together, we conclude that
\begin{align*}
|T_1| \le {}& C \chi_{\ba}(A)^{\frac12} \bigg\{ \!\! \sum_{F\in \Fa \cap \Fint  } \!\!
A_F^\flat \Big(\frac{h_F}{k+1}\Big) \ltwo{\jump{ \nabla   u_{\mesh}} {\times} \n_F}{F}^2
+ \!\!\sum_{F\in \Fa \cap \FD }\!\! A_F^\flat \Big(\frac{h_F}{k+1}\Big)
\ltwo{ \nabla ( u_{\mesh}-g_\ddd)  {\times} \n_{\Omega}}{F}^2 \\
& + \sua  A_K S_{\dK}(\hat{u}_K,\hat{u}_K) \bigg\}^{\frac12}
\times \|\diff^{\frac12}\bnabla e_d\|_{\oma}.
\end{align*}

(5) Bound on $T_2$.
Since $\IamKM({\bphi})\in \bH^1_{0,\dn}(\oma)$, $\jump{\textbf{curl} \,\IamKM ({\bphi}) {\cdot} \n}=0$ across every interface $F\in\Fa \cap \Fint$, ${\textbf{curl} \,\IamKM ({{\bphi}}) {\cdot} \n_F}=0$ on every boundary face $F\in \partial \oma \cap \Gamma_{\dn}$, and  $\nabla{\cdot}\textbf{curl}\, \IamKM({\bphi})= 0$ on $\oma$, integrating by parts the gradient operator gives
$$
( \nabla u^{\ba}_c, \textbf{curl}\, \IamKM({\bphi}))_{\oma}
= (\psi_{\ba} g_\ddd, \textbf{curl} \,\IamKM ({\bphi}) {\cdot} \mbf{n}_{\Omega})_{\partial \oma \cap\Gamma_\ddd} = \sua (\psi_{\ba}g_\ddd,  \textbf{curl}\, \IamKM ({\bphi}) {\cdot} \n_\Omega)_{\dKD}.
$$
Similarly, we have
\[
( \bnabla (\psi_\ba u_{\mesh}), \textbf{curl}\, \IamKM({\bphi}))_{\oma}
= \sua (\psi_{\ba}u_K, \textbf{curl}\, \IamKM ({\bphi})  {\cdot} \n_K )_{\dK}.
\]
Recalling that $T_2=( \bnabla e_{d}^{\ba}, \textbf{curl}\, \IamKM ({\bphi}))_{\oma}$
and $e_{d}^{\ba}=u^{\ba}_c-\psi_\ba u_{\mesh}$, we infer that
\begin{align*}
T_2 ={}& \sua
\Big\{
(g_\ddd- u_K, \psi_{\ba}\textbf{curl}\, \IamKM ({\bphi})  {\cdot} \n_\Omega )_{\dKD}
-(u_K, \psi_{\ba} \textbf{curl}\, \IamKM ({\bphi})  {\cdot} \n_K )_{\dKi}
\Big\}\\
={}& \sua
\Big\{
(\Pi_{\dK}^{k}(g_\ddd|_{\dK}-u_K|_{\dK}), \psi_{\ba} \textbf{curl}\, \IamKM ({\bphi})  {\cdot} \n_\Omega )_{\dKD}
-(\Pi_{\dK}^k (u_K|_{\dK}), \psi_{\ba}  \textbf{curl}\, \IamKM ({\bphi})  {\cdot} \n_K )_{\dKi} \Big\}.
\end{align*}
Notice that we used here that $\psi_{\ba} \textbf{curl}\, \IamKM ({\bphi}){\cdot} \n_K$
is in $\mathbb{P}^k(\FK)$ to introduce the $L^2$-orthogonal projection $\Pi_{\dK}^{k}$ in
both terms on the right-hand side.
Since $u_{\dK}$ is single-valued on every interface $F\in\FKi$ and ${u_{\dK}}|_F = \Pi_{F}^{k} (g_{\ddd}|_F)$ on every boundary face $F\in\FKD$, and since $\norm{\psi_{\ba}}{L^\infty(F)}=1$, we obtain
\[
T_2 = \sua -(\Pi_{\dK}^{k} (u_K|_{\dK})-u_{\dK}, \psi_{\ba} \textbf{curl}\, \IamKM ({\bphi}) {\cdot} \n_K )_{\dKi\cup\dKD}.
\]
Invoking the Cauchy--Schwarz inequality,
the discrete trace inverse inequality \eqref{trace_inv}, and the
definition of the stabilization bilinear form $S_{\dK}$,
we infer that
\begin{align*}
|T_2| \leq {}& (A_{\ba}^{\flat})^{-\frac12} \bigg\{\!\sua\! A_K\frac{(k+1)^2}{h_K}\ltwo{ \Pi_{\dK}^{k}(u_K|_{\dK})-u_{\dK}}{\dKi\cup\dKD }^2  \bigg\}^{\frac12}
\bigg\{ \!\sua\! \frac{h_K}{(k+1)^2}\ltwo{\textbf{curl}\, \IamKM ({\bphi})}{ \dKi\cup\dKD }^2  \bigg\}^{\frac12}  \\
\leq {}& C(A_{\ba}^{\flat})^{-\frac12}
\bigg\{ \!\sua\! A_KS_{\dK}(\hat{u}_K,\hat{u}_K) \bigg\}^{\frac12}
\ltwo{\textbf{curl}\, \IamKM ({\bphi})}{\oma} \\
\leq {}& C' (A_{\ba}^{\flat})^{-\frac12}
\bigg\{ \!\sua\! A_K S_{\dK}(\hat{u}_K,\hat{u}_K) \bigg\}^{\frac12}
\|\nabla \bphi\|_{\oma},
\end{align*}
where the last bound follows from $\ltwo{\textbf{curl}\, \IamKM ({\bphi})}{\oma} \le
(d-1) \|\nabla \IamKM ({\bphi})\|_{\oma}$ and the $H^1$-stability of $\IamKM$
(see~\eqref{eq:approx_IamKM}).
Proceeding as above to bound $\|\nabla \bphi\|_{\oma}$, we conclude that
\[
|T_2| \leq C \chi_{\ba}(A)^{\frac12} \bigg\{ \!\sua\! A_K S_{\dK}(\hat{u}_K,\hat{u}_K) \bigg\}^{\frac12}
\times \|\diff^{\frac12}\bnabla e_d\|_{\oma}.
\]

(3) Putting the bounds on $T_1$ and $T_2$ together and invoking the triangle inequality to introduce the oscillation term on $g_{\ddd}$ proves the assertion.

\begin{remark}[$k=0$]\label{remark: k=0}
Recall that the assumption $k\ge1$ is needed to invoke the
(modified) Karkulik--Melenk interpolation operator.
In the case $k=0$, the nonconforming error can be bounded by using a (piecewise affine)
nodal-averaging operator.
\end{remark}

\subsection{Proof of Theorem~\ref{thm:lowerbound}}

We will prove the following result: For all $K\in\mesh$,
\begin{subequations} \begin{align}
\eta_{K,{\rm res}} \leq {}& C_{\rm l} (k+1) \big( \|A^{\frac12} \nabla e\|_K  A_K^{\frac12} S_{\dK}(\hat{u}_K,\hat{u}_K)^{\frac12} + O_K(f) \big), \label{eq:proof_lb1} \\
\eta_{K,{\rm nor}} \leq {}& C_{\rm l} (k+1)^{\frac12} \bigg\{ \sum_{\mathcal{K} \in \omega_K}
A_{\mathcal{K}}S_{\partial \mathcal{K}}(\hat{u}_\mathcal{K},\hat{u}_\mathcal{K})
\bigg\}^{\frac12},  \label{eq:proof_lb2} \\
\eta_{K,{\rm tan}} \leq {}& C_{\rm l} (k+1)^{\frac32} \bigg\{ \sum_{\mathcal{K} \in \omega_K}
\ltwo{A^{\frac12}\nabla e}{\mathcal{K}}^2 \bigg\}^{\frac12}.  \label{eq:proof_lb3}
\end{align} \end{subequations}%
We consider the following bubble functions from \cite{MelenkWohlmuth01}: For all
$K\in\mesh$ and all $F\in \Fall$, we set
\begin{equation}
b_{K}(\bx) : = h_K^{-1} {\rm dist} (\bx, \dK), \; \forall \bx\in K, \qquad
b_{F}(\bx) : = h_F^{-1} {\rm dist} (\bx, \dF), \; \forall \bx \in F.
\end{equation}
We recall the following $hp$-inverse inequalities from \cite[Theorem 2.5]{MelenkWohlmuth01}:
For all $v\in \mathbb{P}^{k+1}(K)$, all $K\in \mesh$, and all $F\in \FK$, we have
\begin{subequations}
\begin{align}
\ltwo{v}{K} &\leq  C (k+1)\ltwo{b_K^{\frac12}v}{K}, \label{eq: L2_norm_cell_bubble} \\
\ltwo{\nabla (v b_{K})}{K} &\leq  C \Big( \frac{k+1}{h_K} \Big) \ltwo{b_K^{\frac12}v}{K}, \label{eq:  inverse estimates for bubble} \\
\ltwo{v}{F} &\leq  C (k+1)\ltwo{b_F^{\frac12}v}{F}.  \label{eq:L2_norm_face_bubble}
\end{align}
\end{subequations}
We also recall the following $hp$-polynomial extension result from
\cite[Lemma~2.6]{MelenkWohlmuth01} (with $\varepsilon:=(k+1)^{-2}$ and $\alpha:=1$):
For all $F\in \FK$ and all $K\in \mesh$,
there exists an extension operator $E_{K,F}:H^1_0(F) \rightarrow H^1(K)$ such that,
for all $v\in \mathbb{P}^k(F)$,
\begin{equation} \label{eq: extension bound}
E_{K,F}(b_Fv)|_{F} = b_F v, \qquad E_{K,F}(b_Fv)|_{\dK \backslash F}  =0,
\qquad \|\nabla E_{K,F}(b_F v)\|_K \leq C \Big(\frac{(k+1)^2}{h_K}\Big)^{\frac12}  \|b_F^{\frac12}v\|_{F}.
\end{equation}

(1) Proof of~\eqref{eq:proof_lb1}.
Set $v_K := \Pi_{K}^{k+1}(f)+A_K\Delta R_K^{k+1}(\hat{u}_K)\in \mathbb{P}^{k+1}(K)$,
$w_K:= b_Kv_K$, observe that $w_K$ vanishes
on $\dK$, and let $w$ be the zero-extension $w_K$ to $\Omega$.
Since $(f,w_K)_K-(A_K\nabla u,\nabla w_K)_K=(f,w)_K-(A_K\nabla u,\nabla w)_K=0$, an integration by parts for $\Delta R_K^{k+1}(\hat{u}_K)$ gives
\begin{align*}
\ltwo{b_K^{\frac12}v_K}{K}^2 &= \ltwo{b_K^{\frac12}(\Pi_{K}^{k+1}(f)+ \diff_K\Delta R_K^{k+1}(\hat{u}_K))}{K}^2 \\
&=(\Pi_{K}^{k+1}(f)+ \diff_K\Delta R_K^{k+1}(\hat{u}_K),w_K)_K \\
&= ((\Pi_{K}^{k+1}(f) -f),w_K)_{K} + (\diff_K \nabla e , \nabla w_K )_K  + (\diff_K \nabla (u_K - R_K^{k+1}(\hat{u}_K)) , \nabla w_K )_K \nno \\
& \leq \ltwo{\Pi_{K}^{k+1}(f) -f}{K}\ltwo{b_Kv_K}{K} + A_K^{\frac12} \big(\ltwo{A^{\frac12}\nabla e}{K} + A_K^{\frac12} \ltwo{\nabla (u_K - R_K^{k+1}(\hat{u}_K))}{K}\big) \ltwo{\nabla (b_Kv_K)}{K} \\
& \leq
\Big(\ltwo{\Pi_{K}^{k+1}(f)-f}{K}
+ CA_K^{\frac12} \Big( \frac{k+1}{h_K}\Big)  \Big(\ltwo{A^{\frac12} \nabla e}{K} +  (A_K S_{\dK}(\hat{u}_K,\hat{u}_K))^{\frac12}    \Big)
 \Big)
\ltwo{b_K^{\frac12}v_K}{K},
\end{align*}
where the last bound follows from $b_K\leq 1$,
\eqref{eq:  inverse estimates for bubble} and \eqref{local HHO  bound}. Hence,
\begin{align*}
\ltwo{b_K^{\frac12}v_K}{K}
&\leq \ltwo{\Pi_{K}^{k+1}(f)-f}{K}
+ CA_K^{\frac12} \Big( \frac{k+1}{h_K}\Big)  \Big(\ltwo{A^{\frac12} \nabla e}{K} +  (A_K S_{\dK}(\hat{u}_K,\hat{u}_K))^{\frac12} \Big)\\
&= A_K^{\frac12} \Big( \frac{k+1}{h_K} \Big) \bigg( O_K(f) + C\Big( \ltwo{A^{\frac12} \nabla e}{K} + A_K^{\frac12} S_{\dK}(\hat{u}_K,\hat{u}_K)^{\frac12} \Big) \bigg),
\end{align*}
where we used the definition~\eqref{eq:def_Okf} of $O_K(f)$.
Invoking~\eqref{eq: L2_norm_cell_bubble}, we infer that
\[
\eta_{K,{\rm res}} = A_K^{-\frac12}\Big(\frac{h_K}{k+1}\Big) \ltwo{v_K}{K}
\leq C(k+1)A_K^{-\frac12}\Big(\frac{h_K}{k+1}\Big) \ltwo{b_K^{\frac12}v_K}{K}.
\]
Combining the above two bounds proves~\eqref{eq:proof_lb1}.

(2) Proof of~\eqref{eq:proof_lb2}.
For every interface $F\in\FKi$, using the local conservation property \eqref{conservation property} and the triangle inequality gives
\begin{align*}
\|\jump{A\nabla R_{\mesh}^{k+1}(\hat{u}_h)}{\cdot}\n_F \|_{F}^2
\leq \sum_{\mathcal{K} \in \{K,K'\}} 2\Big( \frac{A_{\mathcal{K}} (k+1)^2}{h_{\mathcal{K}}}\Big)^2\|\Pi_{\partial \mathcal{K}}^k (u_{\mathcal{K}}|_{\partial\mathcal{K}}) - u_{\partial \mathcal{K}}\|_{F}^2,
\end{align*}
where $K'$ denotes the mesh cell sharing $F$ with $K$.
Using the mesh shape-regularity, we infer that
\begin{align*}
A_K^{-1}\Big( \frac{h_{K}}{k+1}\Big)\|\jump{A\nabla R_{\mesh}^{k+1}(\hat{u}_h)}{\cdot}\n_F \|_{F}^2
&\leq
\sum_{\mathcal{K} \in \{K,K'\}} 2\Big(\frac{A_{\mathcal{K}} }{A_K}\Big) \Big( \frac{(k+1)h_K}{h_{\mathcal{K}}}\Big) \Big( \frac{A_{\mathcal{K}} (k+1)^2}{h_{\mathcal{K}}}\Big)\|\Pi_{\partial \mathcal{K}}^k (u_{\mathcal{K}}|_{\partial\mathcal{K}}) - u_{\partial \mathcal{K}}\|_{F}^2\\
&\leq C \chi'_K(A)(k+1)  \sum_{\mathcal{K} \in \{K,K'\}} A_{\mathcal{K}}S_{\partial \mathcal{K}}(\hat{u}_\mathcal{K},\hat{u}_\mathcal{K}).
\end{align*}
Moreover, for every boundary face $F\in \FKN$, we have
\begin{align*}
A_K^{-1} \Big( \frac{h_{K}}{k+1}\Big) \|A\nabla R_K^{k+1}(\hat{u}_K){\cdot}\n_F  - \Pi^{k}_{\dK}(g_{\dn}|_{\dK})\|_{F}^2
&=
(k+1)A_{K} \frac{ (k+1)^2}{h_{K}}\|\Pi_{\dK}^k (u_{K}|_{\dK}) - u_{\dK}\|_{F}^2\\
&\leq  (k+1) A_K S_{\dK}(\hat{u}_K,\hat{u}_K).
\end{align*}
Summing over all the faces $F\in\FKi\cup \FKN$ completes the
proof of~\eqref{eq:proof_lb2}.

(3) Proof of~\eqref{eq:proof_lb3}.
For every interface $F\in\FKi$, we set $v_F := \jump{\nabla {u_{\mesh}}}{\times} \n_F$ and $w_{\mathcal{K},F} := E_{\mathcal{K},F}(b_Fv_F)$ for all $\mathcal{K}\in \{K,K'\}$, where $K'$ denotes, as above, the mesh cell sharing $F$ with $K$.
Since $\jump{\nabla {u_{\mesh}}}{\times} \n_F = \sum_{\mathcal{K} \in \{K,K'\}}
\nabla u_{\mathcal{K}} {\times} \n_{\mathcal{K}}$ and $w_{\mathcal{K},F}|_F = b_Fv_F$
owing to~\eqref{eq: extension bound}, we infer that
\[
\|b_F^{\frac12} v_F\|_F^2
=\sum_{\mathcal{K} \in \{K,K'\}} ( w_{\mathcal{K},F} , \nabla u_{\mathcal{K}} {\times} \n_{\mathcal{K}})_{F} =
\sum_{\mathcal{K} \in \{K,K'\}} ( \textbf{curl}\, w_{\mathcal{K},F} , \nabla u_{\mathcal{K}} )_{\mathcal{K}}
= \sum_{\mathcal{K} \in \{K,K'\}}  (\textbf{curl}\, w_{\mathcal{K},F} , \nabla (u_{\mathcal{K}}-u))_{\mathcal{K}},
\]
where the second equality follows by integration by parts and the third equality additionally
uses that $\jump{\nabla u}{\times} \n_F = \mbf{0}$.
The Cauchy--Schwarz inequality gives
\[
\|b_F^{\frac12} v_F\|_F^2 \leq (A_F^{\flat}  )^{-\frac12} \sum_{\mathcal{K} \in \{K,K'\}} \ltwo{A^{\frac12}\nabla e}{\mathcal{K}}\ltwo{ \textbf{curl}\, w_{\mathcal{K},F}}{\mathcal{K}}.
\]
Invoking~\eqref{eq: extension bound} gives
\[
\ltwo{ \textbf{curl}\, w_{\mathcal{K},F}}{\mathcal{K}} \le
C \| \nabla E_{\mathcal{K},F}(b_F v_F)\|_{\mathcal{K}} \le C
(k+1)h_K^{-\frac12} \|b_F^{\frac12}v_F\|_F.
\]
Hence,
\[
\|b_F^{\frac12} v_F\|_F\leq C
(A_F^{\flat}  )^{-\frac12} (k+1) h_K^{-\frac12} \sum_{\mathcal{K} \in \{K,K'\}} \ltwo{A^{\frac12}\nabla e}{\mathcal{K}}.
\]
Since $\ltwo{v_F}{F} \leq  C (k+1)\ltwo{b_F^{\frac12}v_F}{F}$ owing to~\eqref{eq:L2_norm_face_bubble}, combining the above bounds gives
\[
\|v_F\|_F\leq C (A_F^{\flat}  )^{-\frac12}
(k+1)^2h_K^{-\frac12}\sum_{\mathcal{K} \in \{K,K'\}} \ltwo{A^{\frac12}\nabla e}{\mathcal{K}}.
\]
This shows that
\[
(A_{F}^{\flat})^{\frac12} \Big(\frac{h_K}{k+1}\Big)^{\frac12}
\ltwo{ \jump{\nabla  u_{\mesh}} {\times} \n_F }{F}
= (A_{F}^{\flat})^{\frac12} \Big(\frac{h_K}{k+1}\Big)^{\frac12} \|v_F\|_F
\leq C
(k+1)^{\frac32}\sum_{\mathcal{K} \in \{K,K'\}} \ltwo{A^{\frac12}\nabla e}{\mathcal{K}}.
\]
A similar bound can be established for all $F\in \FKD$. Summing over all $F\in\FKi\cup \FKD$
completes the proof of~\eqref{eq:proof_lb3}.

\subsection*{Acknowledgment}
The use of the Cleps computing platform at INRIA Paris is gratefully acknowledged.


\ifHAL
\bibliographystyle{siam}
\else
\fi

\bibliography{bibliography}


\end{document}